%% file: arxiv-main.tex
\documentclass{article}
\usepackage{arxiv}

\usepackage[utf8]{inputenc} 
\usepackage[T1]{fontenc}    
\usepackage{hyperref}       
\usepackage{url}            
\usepackage{amsfonts}       
\usepackage{nicefrac}       
\usepackage{microtype}      
\usepackage{graphicx}
\usepackage{natbib}
\usepackage{doi}

\usepackage{subcaption}
\usepackage{float} 
\usepackage{amsmath}
\usepackage{amsthm}
\usepackage{amssymb}
\usepackage{color}
\usepackage{bm}
\usepackage{algorithm} 
\usepackage{algorithmic}
\usepackage{sidecap}
\sidecaptionvpos{figure}{t}

\newtheorem{remark}{Remark}

\newcommand{\KO}{\mathcal{K}}
	
\newcommand{\R}{\mathbb{R}}
\newcommand{\C}{\mathbb{C}}
\newcommand{\tol}{\texttt{tol}}
\newcommand{\x}{\mathbf{x}}
\newcommand{\y}{\mathbf{y}}
\newcommand{\s}{\mathbf{s}}
\newcommand{\q}{\mathbf{q}}
\newcommand{\z}{\mathbf{z}}
\def\roff  {\mbox{\boldmath$\varepsilon$}}
\newcommand{\norm}[1]{\left\lVert#1\right\rVert}

\newtheorem{theorem}{Theorem}[]
\numberwithin{theorem}{section}

\newtheorem{proposition}[theorem]{Proposition}

\theoremstyle{definition}  

\newtheorem{example}[theorem]{Example}

\title{New Robust Streaming DMD with Forecasting }

\author{ 
Zlatko Drmač \\
	Department of Mathematics\\
    Faculty of Science\\
    University of Zagreb\\
	\texttt{drmac@math.hr} \\
	\And
    Ela Đimoti\\
	Department of Mathematics\\
    Faculty of Science\\
    University of Zagreb\\
	\texttt{ela.dimoti@math.hr} \\
}

\date{}

\renewcommand{\headeright}{}
\renewcommand{\undertitle}{Preprint}
\renewcommand{\shorttitle}{Streaming DMD and dynamic forecasting}

\begin{document}
\maketitle

\begin{abstract}
\input{abstract.tex}
\end{abstract}

\input{sections/introduction}
\input{sections/streaming-hemati-review}
\input{sections/one_basis.tex}
\input{sections/zhang_updates}

\section*{Acknowledgements}
This research is supported by the Croatian Science Foundation (HRZZ) Grant IP-2025-02-3733 \emph{``Data driven identification and dimension reduction of dynamical systems''}, 
and by the European Union – NextGenerationEU through the National Recovery and Resilience Plan 2021-2026,  institutional grant of University of Zagreb Faculty of Science (IK IA 1.1.3. Impact4Math).

\input{sections/appendix}
\newpage
\bibliographystyle{abbrv}
\bibliography{References}

\end{document}

%% file: abstract.tex
The Dynamic Mode Decomposition (DMD) and the more general Extended DMD (EDMD) are powerful tools for computational analysis of dynamical systems in data-driven scenarios. They are built on the theoretical foundation of the Koopman composition operator and can be considered as numerical methods for data snapshot-based extraction of spectral information of the composition operator associated with the dynamics, spectral analysis of the structure of the dynamics, and for forecasting.
In high fidelity numerical simulations, the state space is high dimensional and efficient numerical methods leverage the fact that the actual dynamics evolves on manifolds of much smaller dimension. This motivates computing low rank approximations in a streaming fashion and the DMD matrix is adaptively updated with newly received data. In this way, large number of high dimensional snapshots can be processed very efficiently. Low dimensional representation also requires fast updating for online applications. 
This paper revisits the pioneering  works of Hemati,  Williams and Rowley (Physics of Fluids, 2014),  and Zhang, Rowley, Deem and Cattafesta (SIAM Journal on Applied Dynamical Systems, 2019) on the streaming DMD and proposes
improvements in functionality (using residual bounds, Exact DMD vectors), computational efficiency (more efficient algorithm with smaller memory footprint) and numerical robustness (smaller condition numbers and better forecasting skill). 

%% file: sections/introduction.tex
\section{Introduction}
Dynamic Mode Decomposition (DMD) is a simple, yet powerful and versatile, data-driven tool for analyzing complex nonlinear dynamical systems. It was originally introduced in the context of computational fluid dynamics (CFD) by Schmid and Sesterhenn in their seminal works \cite{Schmid-Sesterhenn-DMD-2008, SCHMID_2010}.
It can be used to identify, analyze and forecast/control the dynamics $\x_{i+1}=T(\x_i)$, governed by a mapping $T: \R^m\longrightarrow \R^m$ that may not be known, but an abundance of data snapshots can be generated. 
 CFD applications motivated further theoretical and algorithmic development and extensions of the method, see 
\cite{ROWLEY_MEZIC_BAGHERI_SCHLATTER_HENNINGSON_2009}, \cite{Schmid2011}, \cite{schmid-2011-dmd-exp-fluids}, \cite{SCHMID2021243}, \cite{mezic_annual_reviews-2013}, \cite{tu-rowley-dmd-theory-appl-2014},  \cite{Williams_2015-EDMD}, \cite{schmid-2022-dmd-variants}, \cite{Jovanovic:2012wy}, \cite{jovschnicPOF14}, \cite{GIANNAKIS2020132211}. 

The key feature of the DMD is that it is entirely data driven and that the core algorithm is oblivious to the nature of the data. This makes it a multi-purpose tool with a wide palette of applications, including robotics, aeroacoustics, video processing, machine learning, and non-intrusive model reduction. For an introduction to DMD and its applications, the reader is referred to \cite{Kutz-SIAM-BOOK-2016}, \cite{Akshay2021}, \cite{brunton-budisic-kaiser-kutz-sirew-2022}. 

In a pure data-driven scenario, the data snapshots are acquired by sophisticated high tech measurements, such as e.g. the Particle Image Velocimetry (PIV) \cite{Chaugule-PIV-DMD-2023}, using fast-response pressure sensitive paint (PSP) for measuring  acoustic pressure field \cite{Ali-PSPD-DMD-2016}, or collecting detailed vehicle trajectory data through a network of synchronized digital video cameras and customized software application that transcribes the vehicle trajectory data from the video \cite{usdot_ngsim_2016}, \cite{Avilla-Mezic-2020}. In another setting, the model is known and high-fidelity numerical simulations are available and are used to study the structure of the dynamics, motivate theoretical development or e.g. to train a reduced order model. In the physics informed DMD, a priori known information on the structure of the underlying operator can be
built in the algorithm \cite{Baddoo-pidmd-2023}.

To set the stage and for the reader’s convenience,  in \S\ref{SS=DMDReview} we first review the key elements of Dynamic Mode Decomposition (DMD). Then, in \S\ref{SS=StreamingDMD-intro}, we define the specific problem addressed in this paper---adapting DMD to perpetual streams of incoming data snapshots. This introduction concludes with \S \ref{SS=Contributions}, where we list our contributions to the streaming DMD.

\subsection{DMD - a review}\label{SS=DMDReview}
Typically, the mathematical model is expressed as an autonomous system of ordinary differential equations $\dot x(t)=F(x(t))$, and in numerical simulations it is replaced by a discrete system $x_{i+1}=T(x_i)$. In other applications, due to the nature of the data, the model is discrete, such as modeling the spread of infectious diseases in order to forecast potential outbreaks \cite{mezic_etal-covid_2024}. 

Suppose that we have gathered data in the form $(f(\x_i),f(T(\x_i)))$, $i=1,2,\ldots$, where $f$ is some observable of interest -- a scalar or vector-valued function of the states. The observable $f$ can be simply the full state, $f(\x)=\x$ and the data is $(\x_i,\y_i=T(\x_{i}))$. The data may originate from a single long trajectory, or from several trajectories of different lengths, obtained by a sequence of numerical simulations or measurements of the system, using different initial conditions. If $n$ data pairs are available, they will be arranged in two 
$m\times n$ matrices $X$ and $Y$ such that 
\begin{equation}\label{eq:Y=f(T(X))}
Y(:,i)=f(T(X(:,i)))=(f\circ T)(X(:,i)),\;\;\; i=1,\ldots, n. 
\end{equation}
Relation (\ref{eq:Y=f(T(X))}) defines the (Koopman) composition operator $\KO f = f\circ T$ on the suitable space of observables, and this connection provides theoretical underpinnings for the DMD.  
For more details we refer the reader to \cite{Mezic-Spectral-MOR-2005}, \cite{Mezic-Koop-Spectrum-2020}.

To simplify the notation, we will use $\x_i=X(:,i)$, $\y_i=Y(:,i)$.
We will be mainly interested in the case $m>n$, possibly $m\gg n$, that is, both $X$ and $Y$ are tall and skinny matrices. 	
The assumption is that $X$ and $Y$ are related by a linear transformation $A$ such that $Y\approx A X$. 
This relation is in practical computation accomplished in the least squares sense, that is,  $A\in\mathrm{argmin}_{A}\|AX-Y\|_F$.  
It should be noted that in this setting ($m>n$) the least squares problem is under-determined and the solution set is a linear manifold. Namely, if $X_{\perp}$ is such that $X_{\perp}^T X=0$, then $(A+E X_{\perp}^T)X=AX$ for any $E$ of compatible dimensions. The particular choice $A=YX^\dagger$ has the unique property of having the smallest Frobenius norm among all solutions.  However, for any choice of $A$ from the solution manifold, the product $AX=Y P_{\mathcal{R}(X^T)}$ is determined uniquely.
(If $\mathrm{rank}(X)=n$, then $AX=Y$.)

\subsubsection{Task $\sharp 1$: Approximate eigenpairs}\label{SSS=DMD-task-1}
The first task in DMD is to compute the eigenvalues and eigenvectors of $A$. When doing that, 
one has to take into account the following constraints: 

(1) Whatever the choice of $A$, the provided data define only how $A$ acts in the range of $X$: 
$AXw = Y P_{\mathcal{R}(X^T)} w$, $w\in\R^n$. 

(2) The range of $X$ is not necessarily $A$--invariant, so the computed eigenvalues and eigenvectors of $A$ are only approximations, with possibly unacceptably large residuals. 

To identify an approximate eigenpair $(\lambda, Xw)$ in this setting, the residual
$r = A Xw - \lambda Xw$ should be small with some judicious choice of $w$ and $\lambda$. One way of minimizing the residual $r$ is to annihilate its component in the range of $X$, so that what it remains is in the orthogonal complement: $X^T r=0$. In other words, 
\begin{equation}\label{eq:GEVP}
X^T A X w = \lambda X^T X w.\;\;(X^T Y P_{\mathcal{R}(X^T)}w = \lambda X^T X w)
\end{equation}
This is an $n\times n$ generalized eigenvalue problem, where $X^T X$ is possibly singular (if $X$ is rank deficient) or highly ill--conditioned (numerically singular). In addition, $X$, as well as $Y$, may be contaminated by noise. To tackle the (numerical) rank  of $X$, 
let $\mathrm{rank}(X)=r_X$, and let $X = U \Sigma V^T$ be the economy--size SVD of $X$, with $r_X\times r_X$ matrix $\Sigma$ of strictly positive singular values $\sigma_1\geq\cdots\geq\sigma_{r_X}>0$. Note that $X^T=V\Sigma U^T$, and that $P_{\mathcal{R}(X^T)}=VV^T$. This means that
\begin{equation}\label{eq:AX=YPXT}
AX=YP_{\mathcal{R}(X^T)} \Longleftrightarrow A U\Sigma V^T = Y VV^T,
\end{equation}
i.e. $A U \Sigma = YV$. 
Since the columns of $U$ are an orthonormal basis for the range of $X$, the Rayleigh--Ritz method can be applied using $U$ which, instead of \eqref{eq:GEVP}, yields the eigenvalue problem 
\begin{equation}\label{eq:UTAU}
U^T A U w = \lambda w, \;\;\mbox{where}\;\;U^T A U = U^T YV\Sigma^{-1}. 
\end{equation}
Each eigenpair $(\lambda,w)$ gives an approximation, the Ritz pair $(\lambda,Uw)$. This is the essence of the
DMD method -- it is a data-driven Rayleigh--Ritz spectral extraction
from the range of $X$. In practical computation, the smallest singular values of $X$ will be computed with large relative errors. A numerically prudent course of action is to use truncated SVD, i.e. to carefully set a tolerance threshold $\tol$ and instead of $r_X$ use the $k$-truncated SVD, where
\begin{equation}\label{eq:truncation}
k = \max\{ i : \sigma_i > \tol\, \sigma_1 \}
\end{equation}
denotes the numerical rank. Using truncated SVD can be interpreted as replacing $X$ with $X+\Delta X=U_{k}\Sigma_{k} V_{k}^T$, where $U_{k}=U(:,1:k)$,
$V_{k}=V(:,1:k)$, $\Sigma_{k}=\Sigma(1:k,1:k)$, and\footnote{Here $\|X\|_2=\sigma_1$ is the matrix spectral norm.} $\|\Delta X\|_2\leq \tol \|X\|_2$. 
In fact, computation of the SVD in floating-point arithmetic already introduces backward error $\delta X$ which 
obeys an upper bound similar as $\Delta X$. 
For a detailed discussion on this topic see \cite{dmd_enhancements-qrdmd},  \cite{Drmac-DMD-TOMS-2024}.  

Hence, in practical computation, the formula (\ref{eq:UTAU}) is used with $U\leftarrow U_{k}$, $V\leftarrow V_{k}$, $\Sigma\leftarrow \Sigma_{k}$.
For each Ritz pair $(\lambda_i,\z_i=Uw_i)$, $i=1,\ldots,k$, the residual $r_i=AUw_i-\lambda Uw_i$ is in general not zero. An important addition to the DMD, introduced in \cite{dmd_enhancements-qrdmd}, is computation of the residuals within the DMD scheme: It suffices to note that 
\begin{equation}\label{eq:DMD-residuals}
r_i=A\z_i-\lambda_i \z_i = AUw_i -\lambda_i Uw_i=YV\Sigma^{-1}w_i-\lambda_i U w_i. 
\end{equation}
Then, all norms $\|r_i\|_2$ can be computed and Ritz pairs can be ordered starting with the best ones, i.e. those with smallest residuals.
For identifying coherent states, the best candidates are (at this point at least intuitively) the Ritz pairs with smallest residuals. 
This is an important part of the DMD algorithm that is, unfortunately, missing in most of the literature and software solutions. An implementation that includes the residuals has recently been included in the LAPACK library \cite{Drmac-DMD-TOMS-2024}. 
Algorithm \ref{ALG:DMD:RRR} summarizes the preceding discussion.
\begin{algorithm}
\caption{$(Z_k, \Lambda_k, r_k, [B_k], [Z_k^{(ex)}])=\mathrm{DMD}(X,Y; \tol)$}
	\label{ALG:DMD:RRR}
	\begin{algorithmic}[1]
		\REQUIRE 
		$X=(\x_1,\ldots,\x_n), Y=(\y_1,\ldots,\y_n)\in {\R}^{m\times n}$ that define a sequence of snapshots 
		pairs $(x_i,y_i)$. (Tacit assumption is that $m$ is large and that $n \ll m$.)
		Tolerance $\tol$ for the truncation (\ref{eq:truncation}).	
		\STATE $[U,\Sigma, V]=\mathrm{svd}(X, \text{"econ"})$ ; \COMMENT{\emph{$X = U \Sigma V^T$, $U\in{\R}^{m\times n}$, $V\in{\R}^{n\times n}$, $\Sigma={diag}(\sigma_i)_{i=1}^n$.}}
		\STATE Determine numerical rank $k$, using (\ref{eq:truncation}) with the threshold $\tol$.
		\STATE $U_k=U(:,1:k)$, $V_k=V(:,1:k)$, $\Sigma_k=\Sigma(1:k,1:k)$;
		\STATE ${B}_k = Y (V_k\Sigma_k^{-1})$; \COMMENT{\emph{Data driven formula for $A U_k$ [optional output].}}
		\STATE $S_k = U_k^* B_k$; \COMMENT{\emph{$S_k = U_k^* A U_k$ is the Rayleigh quotient.}}
		\STATE $[W_k, \Lambda_k] = \mathrm{eig}(S_k)$; \COMMENT{$\Lambda_k=\mathrm{diag}(\lambda_i)_{i=1}^k$; $S_k W_k(:,i)=\lambda_i W_k(:,i)$.}
		\STATE $Z_k = U_k W_k$; 
		\STATE $Z_k^{(ex)} = B_k W_k$; \COMMENT{\emph{The (unscaled) Exact DMD vectors [optional output].}}
		\STATE $r_k(i) = \|B_k W_k(:,i) - \lambda_i Z_k(:,i)\|_2$, $i=1,\ldots, k$; \COMMENT{\emph{The residuals (\ref{eq:DMD-residuals}).}}
		\ENSURE $Z_k$, $\Lambda_k$, $r_k$, $[B_k]$, $[Z_k^{(ex)}]$.
	\end{algorithmic}
\end{algorithm}

\noindent It should be noted that the matrix $YX^\dagger$ is never used in this procedure. For the definition and properties of the
Exact DMD vectors (Line 8 in Algorithm \ref{ALG:DMD:RRR}) see \cite{tu-rowley-dmd-theory-appl-2014} and \cite{Drmac-DMD-TOMS-2024}.

\subsubsection{Task $\sharp 2$: Modal analysis of the data and forecasting}\label{SSS=DMD-task2}
After identifying approximate eigenpairs, the second task is to analyze the data (modal decomposition) and develop forecasting skill. For a selection of $\ell\leq k$ Ritz pairs $(\lambda_{\varsigma_j},\z_{\varsigma_j})$, $j=1,\ldots,\ell$, the
task is to find coefficients $\alpha_1, \ldots, \alpha_\ell$ that minimize the loss function
\begin{equation}\label{eq:modal-dcomp}
L(\alpha_1,\ldots,\alpha_{\ell})=\sum_{i=1}^{n+1} w_i^2 \| \x_i - \sum_{j=1}^{\ell} \z_{\varsigma_j} \alpha_j \lambda_{\varsigma_j}^{i-1}\|_2^2 \longrightarrow \min_{\alpha_j} \;\;\;(\x_i\approx \sum_{j=1}^{\ell} \z_{\varsigma_j} \alpha_j \lambda_{\varsigma_j}^{i-1}),
\end{equation}
where $w_i\geq 0$ is a weight that determines the importance of the $i$-th snapshot. 
The structure of the dynamics is better revealed and understood if a faithful representation (\ref{eq:modal-dcomp}), called the Koopman Mode Decomposition (KMD), 
can be achieved with small $\ell$.
Then, for instance, if 
$\x_{n+1}=\y_n \approx \sum_{j=1}^{\ell} \z_{\varsigma_j} \alpha_j \lambda_{\varsigma_j}^{n}$, 
a reasonable forecast for $\y_{n+k}$ is $A^k \y_n \approx \sum_{j=1}^{\ell} \z_{\varsigma_j} \alpha_j \lambda_{\varsigma_j}^{n+k}$, $k=1, 2, \ldots$. For this application of (\ref{eq:modal-dcomp}) the weights $w_i$ should favor the most recent data snapshots.

The selection of a small number of modes that capture data snapshots can be achieved by imposing a sparsity constraint to minimization (\ref{eq:modal-dcomp}). This is accomplished in the Sparsity Promoting DMD \cite{jovschnicPOF14}, where the selection of the modes is delegated to a black-boxed sparsity-constrained optimization software. An alternative is to select modes that have small residuals as shown in \cite{Colbrook-Drmac-Horning-2025}.
For detailed discussion and another approach, see \cite{Sayadi-etal-param-DMD-2015}.

\subsection{DMD of  large dimensional perpetual data streams}\label{SS=StreamingDMD-intro}

In numerical simulations with large $m$,  memory footprint and computational complexity make the DMD analysis expensive. If $m$ is, say, in millions then accommodating many snapshots in fast memory is unfeasible, and out-of-core and low rank approximation techniques may be necessary.
This becomes even more challenging task in an online (streaming) application of the DMD, when the new data are delivered one snapshot or a block of snapshots at a time. 
Hence, in streaming scenarios, the DMD/KMD analysis (optionally with forecasting) is performed continuously in a window of dynamically expanding width.

The applications include onboard ocean and acoustic forecasting of temparature, salinity, velocity and transmission loss fields on unmanned autonomous ocean platforms \cite{Ryu-etal-OnboardDMD-2022}, model predictive control of fluid flow \cite{Li_Deng_2024} or e.g. the UAV's \cite{Quadcoper-dmd-2023}, construction of reduced order model in flow simulations \cite{2023arXiv231118715S}, \cite{HUHN2023111852},  control of wind farms \cite{windfarms-2022}, or e.g. real-time motion detection \cite{Mignacca-videoDMD-2025}.

Clearly, computing a new DMD from scratch after adding one more snapshots is not optimal and it may not be feasible in some real-time applications. This problem was first addressed by Hemati, Williams and Rowley \cite{hemati_williams_rowley_2014}, who proposed low rank approximations of the data, based on a rank revealing  decomposition of data matrices. Such an approach is most effective in cases of dynamics described in large state space dimension $m$, but evolving on low dimensional manifolds. Zhang, Rowley, Deem and Cattafesta \cite{zhang_DMD} developed this further and described fast updating scheme for the case of large number of low dimensional data snapshots (such that $X$ is of full row rank).
The key idea of this approach is to use the Sherman-Morrison-Woodbury formula that is used for fast update of the pseudoinverse $X^\dagger=X^T(XX^T)^{-1}$.
Some approaches to the streaming DMD use techniques for updating the SVD of $X$, see \cite{Matsumoto-fly-dmd-2017}, \cite{Nedzhibov-SVDupdate-DMD-2023}. Since our main interest are the cases of large dimensional data, the SVD based updating is less attractive and we focus on \cite{hemati_williams_rowley_2014}, \cite{zhang_DMD}.

\subsection{Organization of the paper and our contributions}\label{SS=Contributions}
The streaming DMD from  \cite{hemati_williams_rowley_2014} is reviewed in \S \ref{SS=Hemati-review}, with detailed discussions of some subtle numerical details. In \S \ref{S=Modifications}, we first point out that the condition number that determines the numerical accuracy of the methods in 
\cite{hemati_williams_rowley_2014}, \cite{zhang_DMD} is $\kappa_2(XX^T)=\kappa_2(X)^2$. 
Then, in \S \ref{SS=SquaredCond}, we present two major modifications of the streaming DMD: \emph{(i)}  To avoid squaring the condition number when updating $X^\dagger$, we introduce new methods based on
fast updates of certain orthogonal factorization of $X$ (\S \ref{SS=AvoidkappaG}) and on updated Cholesky factor of $XX^T$ (\S \ref{SS:UseUpdatedCholFactor}). Similarly as in \cite{hemati_williams_rowley_2014}, we can enforce maximal rank (\S \ref{SSS=TQ-Rank-Enforce}).
Numerical example in \S \ref{SSS:example-vorticity} illustrates the superiority of the new approach. \emph{(ii)} It is shown how to compute the Exact DMD vectors (\S \ref{SS=ExactDMD}) and use them for efficient KMD in the framework of \cite{hemati_williams_rowley_2014}. This section concludes with a discussion (\S \ref{SS=DiscussFurtherImprove}) that motivates further improvements. 

The second part of this work (\S \ref{S=Hemati_ONE_BASIS}) presents a new streaming DMD algorithm based on \cite{dmd_enhancements-qrdmd}, \cite{Drmac-DMD-TOMS-2024}. Its main difference from \cite{hemati_williams_rowley_2014} is that it uses a single orthonormal basis for both $X$ and $Y$ (\S \ref{SS=LowRankSnapshotsBasisChange}). However, it has the same functionalities, including residual bounds and using the Exact DMD vectors (\S \ref{SS=OneBasisExactDMDResiduals}).
The updating, outlined in \S \ref{SS:hemati_one_basis_updating},
is based on orthogonal decomposition from \S \ref{SS=AvoidkappaG}. The superior numerical robustness of the method is illustrated in \S \ref{S=SingleBasisForecasting} using the problem of forecasting.
In addition to this enhanced robustness, the one-basis approach reduces memory footprint and the \emph{flop} count, and allows for more efficient KMD, selection of Ritz pairs with small residuals, and efficient application of other variations of the DMD (sparsity promoting DMD \cite{jovschnicPOF14}, forward-backward DMD \cite{Dawson2016}, Optimal DMD \cite{wynn_pearson_ganapathisubramani_goulart_2013}, Koopman-Schur DMD \cite{drmac-koopman-schur}). Section \ref{S=Zhang-New} reviews the updating strategy of \cite{zhang_DMD}, and \S \ref{S6=SMforNgtM} provides explicit formulas for using \cite{zhang_DMD}
in the reduced low rank computation in \cite{hemati_williams_rowley_2014}. Then, we propose two alternatives to the numerically sensitive Sherman-Morrison updating of the inverse of $XX^T$: implicit updating of the Cholesky factor in \S \ref{SS=NewZhang-CholeskyUpdate} and using the TQ factorization in  \S \ref{SS=NewZhang-LQUpdate}. Numerical examples
confirm theoretical predictions that the new proposed approaches in numerical accuracy outperform existing methods. Final remarks and plans for further development are given in \S \ref{S=Conclude}.

%% file: sections/streaming-hemati-review.tex
\section{Streaming DMD}\label{SS=Hemati-review}

The key idea of \cite{hemati_williams_rowley_2014} is to keep $X$ and $Y$ in  factored rank--revealing factorizations 
\begin{equation}\label{eq:RRF-XY}
X=Q_x \widetilde X, \;\; Y=Q_y\widetilde Y,
\end{equation}
where $Q_x\in\R^{m\times r_x}$, $Q_y\in\R^{m\times r_y}$ are orthonormal, and $\widetilde X\in\R^{r_x\times n}$, $\widetilde Y\in\R^{r_y\times n}$ are full row rank matrices with $r_x=\mathrm{rank}(X)$, $r_y=\mathrm{rank}(Y)$. Using this notation, (\ref{eq:AX=YPXT}) becomes 
\begin{equation}\label{eq:AQxtildeX}
A Q_x\widetilde X = Q_y \widetilde{Y} P_{\mathcal{R}(X^T)}= Q_y \widetilde{Y}\widetilde X^T \widetilde{X}^{T\dagger}, 
\end{equation}
where we have used that  the orthogonal projector onto the range of $X^T$ equals
$P_{\mathcal{R}(X^T)} = \widetilde X^T \widetilde{X}^{T\dagger}$. Multiplying equation (\ref{eq:AQxtildeX}) from the right with $\widetilde X^T$ yields 
\begin{equation}\label{eq:AQx}
AQ_x \widetilde X\widetilde X^T = Q_y \widetilde Y\widetilde X^T,\;\;\mbox{i.e.}\;\; A Q_x = Q_y\widetilde Y \widetilde X^T (\widetilde X\widetilde X^T)^{-1}.
\end{equation}
The Rayleigh quotient reads, instead of (\ref{eq:UTAU}),
\begin{equation}\label{eq:QxTAQX}
Q_x^T AQ_x = (Q_x^T Q_y)(\widetilde Y \widetilde X^T)(\widetilde X\widetilde X^T)^{-1}.
\end{equation}
The Ritz vectors are obtained by lifting the eigenvectors $w_i$ of $Q_x^T AQ_x$ by $Q_x w_i$. 
To assimilate the new data $(\x_{new},\y_{new})$, \cite{hemati_williams_rowley_2014} first updates the factorizations (\ref{eq:RRF-XY}), and then the ingredients in the first formula in (\ref{eq:QxTAQX}). Note that the dimensions of all three factors are determined by the ranks $r_x$, $r_y$ and the other two dimensions $m$, $n$ are inner dimensions in the cross products. 

\noindent The key variables in the algorithm, in addition to the bases $Q_x$, $Q_y$ are the following cross-products  
\begin{equation}\label{eq-cross-products}
Q_{x,y} = Q_x^T Q_y,\;\; G_{y,x} = \widetilde Y \widetilde X^T,\;\; G_x = \widetilde X\widetilde X^T,\;\; G_y = \widetilde Y \widetilde Y^T .
\end{equation}
\subsection{Adding new data}\label{SS=AddNewData}
Suppose a new data pair $(\x_{new},\y_{new})$ becomes available. The goal is to update the factorizations (\ref{eq:RRF-XY}) and the Rayleigh quotient (\ref{eq:QxTAQX}). Consider $X_{new}$. The standard procedure for the Gram-Schmidt orthogonalization is as follows:
\begin{eqnarray}
\!\!\! X_{new} &=& \begin{pmatrix}Q_x\widetilde X & \x_{new}\end{pmatrix} = \begin{pmatrix} Q_x & \x_{new}-Q_x Q_x^T \x_{new}\end{pmatrix}\begin{pmatrix} \widetilde X & Q_x^T \x_{new}\cr 0 & 1 \end{pmatrix} \nonumber \\
&=& \begin{pmatrix} Q_x & q_{new} \end{pmatrix} \begin{pmatrix} \widetilde X & g_x \cr 0 & \gamma_x\end{pmatrix}=Q_{x,new} \widetilde X_{new},\;\;\; g_x=Q_x^T \x_{new}, \;\; \gamma_x=\|\x_{new} - Q_x g_x\|_2 , \label{eq:Xnew=QX} 
\end{eqnarray}
where the construction of the new basis vector $q_{new}$ depends on $\gamma_x$ as follows:
\begin{itemize}
\item[(1)] \framebox{$\gamma_x\neq 0$}. In this case $q_{x,new} = (\x_{new} - Q_x g_x)/\gamma_x$, and 
\begin{equation}\label{eq:tildeXnew}
Q_{x,new}=\begin{pmatrix} Q_x & q_{x,new}\end{pmatrix}, \;\;\; \widetilde X_{new}=\begin{pmatrix} \widetilde X & g_x \cr 0 & \gamma_x\end{pmatrix};
\end{equation}
\begin{equation}\label{eq:XmewXnewT-1}
\widetilde X_{new}\widetilde X_{new}^T = \begin{pmatrix} \widetilde X\widetilde X^T +g_x g_x^T &  g_x \gamma_x\cr \gamma_x g_x^T & \gamma_x^2\end{pmatrix} = \begin{pmatrix} \widetilde X\widetilde X^T & 0 \cr 0 & 0\end{pmatrix} + \begin{pmatrix} g_x \cr \gamma_x\end{pmatrix}
\begin{pmatrix} g_x^T & \gamma_x\end{pmatrix} .
\end{equation}

\item[(2)] \framebox{$\gamma_x=0$} In this case $Q_{x,new}=Q_x$, $\widetilde X_{new} = \begin{pmatrix} \widetilde X & g_x\end{pmatrix}$, and
\begin{equation}\label{eq:XmewXnewT-2}
\widetilde X_{new}\widetilde X_{new}^T = \widetilde X\widetilde X^T + g_xg_x^T.
\end{equation}
\end{itemize}

\noindent $Q_{y,new}$ and $\widetilde Y_{new}$ are computed in the same way, and $Q_{x,y}^{(new)}$ can be computed as  (e.g. if $\gamma_x \neq 0, \gamma_y \neq 0$)
\begin{equation}\label{eq:QXTQY-update}
Q_{x,new}^T Q_{y,new} = \begin{pmatrix} Q_x^T \cr  q_{x,new}^T\end{pmatrix}\begin{pmatrix} Q_y & q_{y,new}\end{pmatrix} = \begin{pmatrix} Q_x^T Q_y & Q_{x}^T q_{y,new} \cr q_{x,new}^T Q_y & q_{x,new}^T q_{y,new}\end{pmatrix} .
\end{equation}
\noindent An important detail is that the updated matrices $\widetilde X_{new}$, $\widetilde Y_{new}$ are not explicitly computed. Instead, the algorithm computes the Rayleigh quotient using the formula (\ref{eq:QxTAQX}) and keeps the updated products $\widetilde X_{new}\widetilde X_{new}^T$ (using (\ref{eq:XmewXnewT-1}), (\ref{eq:XmewXnewT-2})), and analogously $\widetilde Y_{new}\widetilde Y_{new}^T$ and $\widetilde Y_{new}\widetilde X_{new}^T$, whose dimensions depend on the numerical ranks (column dimensions of $Q_{x,new}$ and $Q_{y,new}$). This is advantageous in the case of large number of high-dimensional data snapshots and small numerical ranks.

\subsubsection{Reorthogonalization}\label{SSS:GS-reorth-strategy}

In finite precision computation, the test $\gamma_x \neq 0$ is implemented using a caller-provided tolerance $\tol{}_1$, i.e. if $\gamma_x < \tol{}_1 \|\x_{new}\|_2$, then $Q_{x,new}=Q_x$ (case (2) applies). 
There is a subtlety here.
The Gram-Schmidt procedure in floating-point arithmetic may fail to compute numerically orthogonal vectors, see e.g. \cite{Gram-schmidt-numerics}. A well-known remedy is reorthogonalization that is indicated when $\gamma_x$ is small relative to $\|\x_{new}\|_2$. ($\gamma_x/\|\x_{new}\|_2$ is the sine of the angle between $\x_{new}$ and $\mathrm{range}(Q_x)$.) But in the context of low rank representation of the data, small $\gamma_x$ means that $\x_{new}$ can be approximated from the range of $Q_x$ and that adding new vector to the orthonormal basis is not necessary. As a result,  instead of reorthogonalization, $\x_{new}$ is replaced by its projection onto $\mathrm{range}(Q_x)$. In other words, we do not insist in generating new direction, orthogonal to $\mathrm{range}(Q_x)$.

This opens the question of the choice of the threshold $\tol{}_1$. To that end, note that $\gamma_x$ is the norm of the minimal change $\Delta \x_{new}$ such that $\x_{new}+\Delta\x_{new}\in\mathrm{range}(Q_x)$. Assume that the data snapshots contain an initial uncertainty such that $\x_{new}$ contains noise $\delta\x_{new}$ such that $\|\delta \x_{new}\|_2\leq \eta \|\x_{new}\|_2$ for some $\eta\in [0,1)$. 
Hence, it is reasonable to set $\tol{}_1=\max\{\eta, m\roff\}$, where $\roff$ is the roundoff unit of the working precision.

The reorthogonalization works as follows. Let in (\ref{eq:Xnew=QX}) $\x^\prime=computed(\x_{new}-Q_x Q_x^T \x_{new})$ denote the computed vector, which is not necessarily numerically orthogonal to the columns of $Q_x$. Then
\begin{equation*}
\begin{pmatrix}
Q_x & \x^\prime
\end{pmatrix} = \begin{pmatrix}
Q_x & \x^\prime -Q_xQ_x^T \x^\prime 
\end{pmatrix} \begin{pmatrix}
I & Q_x^T \x^\prime \cr \mathbf{0} & 1
\end{pmatrix} 
= \begin{pmatrix}
Q_x & q_{x,new}^{\prime}
\end{pmatrix} \begin{pmatrix}
I & Q_x^T \x^\prime \cr \mathbf{0} & \gamma_x^\prime
\end{pmatrix},\;\;\gamma_x^\prime=\|\x^\prime-Q_xQ_x^T \x^\prime\|_2 ,
\end{equation*}
and, with $g_x^\prime=Q_x^T \x^\prime$,
\begin{eqnarray*}
X_{new} &=&  \begin{pmatrix}
Q_x & q_{x,new}^{\prime}
\end{pmatrix}   \begin{pmatrix}
I & g_x^\prime \cr \mathbf{0} & \gamma_x^\prime
\end{pmatrix} \begin{pmatrix} \widetilde X & g_x\cr \mathbf{0} & 1 \end{pmatrix} 
= \begin{pmatrix}
Q_x & q_{x,new}^{\prime}
\end{pmatrix} \begin{pmatrix} \widetilde X & g_x+ g_x^\prime\cr \mathbf{0} & \gamma_x^\prime \end{pmatrix}.
\end{eqnarray*}
In practice, reorthogonalization can be set as default, or (to avoid unnecessary computation) it can be switched on only when necessary, based on the comparison of the ratio $\gamma_x/\|\x_{new}\|_2$ with a suitably set threshold value $\tol{}_2$ (e.g. $\tol{}_2=0.1$). 
By the Kahan--Parlett's ``twice is enough'' theorem, it suffices to apply the reorthogonalization once, so that the Gram--Schmidt step is applied twice.\footnote{In \cite{hemati_williams_rowley_2014}, the function \texttt{update} uses $5$ steps by default.} The threshold value $\tol{}_2$
can be tuned to set an acceptable level of numerical orthogonality. The reorthogonalization can also be simply set as default, but in the case of large dimension $m$ applying $Q$ even when it is not necessary incurs costly data movement and flops.

\begin{algorithm}[H]
    \caption{\label{ALG:GS-Reorthog} $[\q,g, \gamma]$ = \texttt{GS\_update}($Q, \x, \tol{}_1, \tol{}_2$ )} 
\begin{algorithmic}[1]
\REQUIRE Orthonormal $Q\in\R^{m\times k}$; $\x\in\R^m$; tolerance thresholds $\tol{}_1, \tol{}_2$. 
\ENSURE $\q$, $g$, $\gamma$
\STATE $g=Q^T \x$; $\q = \x - Q g$; $\gamma=\|\q\|_2$;
\IF{$\gamma > \tol{}_1 \|\x\|_2$}
\IF{$\gamma > \tol{}_2 \|\x\|_2$} 
\STATE $\q=\q / \gamma$; 
\ELSE 
\STATE $g^\prime=Q^T \q$; \ $\q = \q - Q g^\prime$; \ $g=g + g^{\prime}$; \ $\gamma=\|\q\|_2$;  \ $\q=\q / \gamma$;  
\ENDIF
\ELSE
\STATE $\gamma=0$;  \COMMENT{$\x$ brings no new direction}
\ENDIF 
\end{algorithmic}
\end{algorithm}

\begin{remark}\label{REM:pinv-svd-eig}
{\em
The numerical rank of each $\widetilde X$ is tested when computing the Rayleigh quotient (\ref{eq:QxTAQX}). The software in \cite{hemati_williams_rowley_2014}\footnote{Matlab function \texttt{compute\_modes} in the class \texttt{StreamingDMD}.} uses Matlab's function \texttt{pinv(.)} to explicitly compute the pseudoinverse of $G_x$, where 
$G_x=\widetilde X\widetilde X^T$ is updated using the formulas
(\ref{eq:XmewXnewT-1}), (\ref{eq:XmewXnewT-2}). Recall that $\texttt{pinv}(G_x)$ computes the SVD
$G_x = U_g\Sigma_g V_g^T$, and discards the singular values that are less or equal to\footnote{In Matlab notation, this threshold reads $\max(\texttt{size}(G_x))*\texttt{eps}(\texttt{norm}(G_x))$.}  $\varepsilon r_x \|G_x\|_2$, where $\varepsilon$ is the round-off unit. This default threshold value can be changed into caller specified value \texttt{tol} by an additional input parameter, $\texttt{pinv}(G_x,\texttt{tol})$, which is important in the case of noisy data. It should be noted that the \texttt{svd(.)} and \texttt{pinv(.)}
functions assume non-symmetric matrix, and that it would be more efficient to compute the SVD as the spectral decomposition of $G_x$ using the \texttt{eig(.)} function, then truncate small singular values\footnote{For symmetric positive semidefinite matrices, the eigenvalues and the singular values are the same set; the spectral decomposition is the SVD.} using a suitable threshold.
Since in each step $G_x$ changed by a rank-one correction, fast updates can be used. 
}
\end{remark}

\subsubsection{Low rank constraints}\label{SS=Hemati-low_rank_constraints}
Suppose that, given the dimensions of the problem and the operational constraints, there is an upper limit $\widehat{r}$ for the number of columns in $Q_x$. In relation to this, \cite{hemati_williams_rowley_2014} proposes additional steps to control the numerical ranks of $\widetilde X$ and $\widetilde Y$. We now describe a variation of the procedure from \cite{hemati_williams_rowley_2014}. 

Let the number of columns of the new updated $Q_{x,new}$ reach the maximal dimension $\widehat{r}$. Consider the factorization (\ref{eq:Xnew=QX}) of $X_{new}$, and let $\widetilde X_{new}=\Omega\Sigma \Xi^T$ be the SVD. 
To make the next step feasible (under the constraint of maximal rank $\widehat{r}$), 
it suffices to replace $\widetilde{X}_{new}$ with 
$$
\widetilde X_{new}\prime=\Omega{(:,1:\widehat r-1)}\Sigma(1:\widehat r-1,1:\widehat r-1) \Xi(:,1:\widehat r-1)^T .
$$
This strategy is implemented in the software accompanying \cite{hemati_williams_rowley_2014}. By the Eckart-Young-Mirsky theorem, this is the best rank--$(\widehat r-1)$ approximation of $\widetilde X_{new}$, and the maximal rank constraints is enforced with minimal damage to the data. 

On the other hand,
the SVD of $\widetilde X_{new}$ may reveal that the numerical rank is even smaller in the sense that for some index $\rho_x < \widehat r-1$ and acceptably small tolerance $\tol{}_3$, 
$\sigma_{\rho_x+1}(\widetilde X_{new}) < \tol{}_3 \sigma_1(\widetilde X_{new})$. This yields a rank--$\rho_x$ approximation 
$
\widetilde X_{new}\approx \Omega(:,1:\rho_x)\Sigma(1:\rho_x,1:\rho_x) \Xi(:,1:\rho_x)^T .
$
Since $X_{new} = (Q_{x,new}\Omega)\Sigma \Xi^T$ is the SVD, we have optimal low rank approximation
\begin{equation*}
X_{new} \approx Q_{x,new} \Omega(:,1:\rho_x)\Sigma(1:\rho_x,1:\rho_x) \Xi(:,1:\rho_x)^T 
= \widehat Q_{x,new} \widehat X_{new},\;\;\widehat Q_{x,new} = Q_{x,new} \Omega(:,1:\rho_x),\; 
\end{equation*} 
and $\widehat X_{new} = \Sigma(1:\rho_x,1:\rho_x) \Xi(:,1:\rho_x)^T=\Omega(:,1:\rho_x)^T\widetilde X_{new}$ is used through the cross product
$\widehat X_{new}\widehat X_{new}^T = \Sigma(1:\rho_x,1:\rho_x)^2$. This means that the matrix $\Xi$ of the right singular
vectors is not needed and that it suffices to compute the spectral decomposition 
$\widetilde X_{new}\widetilde X_{new}^T = \Omega \Lambda \Omega^T$ with decreasing eigenvalues along the diagonal of $\Lambda=\Sigma^2$. 
Matrix $G_x$ is then changed to $\Lambda(1:\rho_x,1:\rho_x)$.
This formulation is in accordance with the structure of the algorithm that keeps and updates only the cross--products 
(\ref{eq:XmewXnewT-2}). The cross product $G_{y,x}=\widetilde Y_{new}^T\widetilde X_{new}$ is changed to $\widetilde Y_{new}^T\widehat X_{new}$ implicitly by updating $G_{y,x}$ into $G_{y,x}\Omega(:,1:\rho_x)$. Similarly, $Q_{x,y}=Q_{x,new}^TQ_{y,new}$ is changed into 
$\Omega(:,1:\rho_x)^T Q_{x,y}$.

The value $\rho_x$ is determined using the caller's supplied tolerance $\tol{}_3$, so the rank will be reduced only if it can be done within the prescribed tolerance. If the memory space is critical, the rank reduction can be simply enforced.  It should be noted that enforcing the hard-coded maximal rank may introduce unacceptably large error -- one option is to override the constraint and/or issue a warning message. Further, if an estimate of the level of noise in the data is available,\footnote{See \S \ref{SSS:GS-reorth-strategy}.} it can be used to reduce the rank whenever possible,  as described above.

\begin{remark}\label{REM-ry-constraint}
{\em
The formula (\ref{eq:QXTQY-update}) efficiently updates the first factor in (\ref{eq:QxTAQX}) after adding $(\x_{new},\y_{new})$. 
In the case of reducing the basis $Q_y$ the update is analogous, using the rank revealing spectral decomposition  
$
\widetilde Y_{new} \widetilde Y_{new}^T \approx \Omega_Y(:,1:\rho_y) \Sigma_Y(1:\rho_y,1:\rho_y)^2 \Omega_Y(:,1:\rho_y)^T. 
$
The first factor in \eqref{eq:QxTAQX} is computed as
$\widehat Q_{x,new}^T \widehat Q_{y,new} = \Omega_X^T \left( Q_{x,new}^T Q_{y,new} \right) \Omega_Y$,
where the inner factor is updated using $\eqref{eq:QXTQY-update}$.  
If the maximum rank $\widehat r$ is not imposed, it is not necessary to compute and update the matrix $G_y$.
}
\end{remark}
\vspace{-1mm}
\noindent The updating procedure is summarized in Algorithm \ref{ALG:Hemati-DMD}. The details of enforcing the maximal rank
(as described above) that correspond to Line 14 are omitted for the sake of brevity.
\begin{algorithm}[H]
\caption{\label{ALG:Hemati-DMD} 
$[Q_x, Q_y, Q_{x,y}, G_{y,x}, G_x, G_y]$ = \\ 
\texttt{Addxy\_compress}($Q_x, Q_y, Q_{x,y}, G_{y,x}, G_x, G_y, \x, \y , \widehat r, \tol{}_1, \tol{}_2, \tol{}_3$ )}
    \begin{algorithmic}[1]
    \REQUIRE orthonormal $Q_x, Q_y$ from (\ref{eq:RRF-XY}); cross products $Q_{x,y},G_{y,x}, G_x, G_y$ from (\ref{eq-cross-products}); new data  $\x, \y$; maximal dimension $\widehat r$; tolerances $\tol{}_1, \tol{}_2, \tol{}_3$ 
    \ENSURE  Updated matrices $Q_x, Q_y, Q_{x,y}, G_{y,x}, G_x, G_y$ 
    \STATE $[\q_{x},g_x, \gamma_x]$ = \texttt{GS\_update}($Q_x, \x, \tol{}_1, \tol{}_2$);
    \STATE $[\q_{y},g_y, \gamma_y]$ = \texttt{GS\_update}($Q_y, \y, \tol{}_1, \tol{}_2$);
    \STATE $G_x = G_x + g_x g_x^T$; $G_y=G_y + g_y g_y^T$;  $G_{y,x}=G_{y,x} + g_y g_x^T$;
    \IF{$\gamma_x > 0$ and $\gamma_y > 0$}
    \STATE $Q_{x,y} = \left(\begin{smallmatrix}
        Q_{x,y} & Q_x^T \q_{y} \\ \q_{x}^T Q_y & \q_{x}^T \q_{y}
    \end{smallmatrix}\right)$; $Q_{x}=(Q_x \; \q_{x})$; $Q_{y}=(Q_y \; \q_{y})$;
    \STATE $G_{y,x} = \left(\begin{smallmatrix}
            G_{y,x} & g_y\gamma_x \cr
            \gamma_y g_x^T & \gamma_x\gamma_y
        \end{smallmatrix}\right)$;  $G_x = \left(\begin{smallmatrix}
        G_x  & g_x\gamma_x \cr \gamma_x g_x^T & \gamma_x^2
    \end{smallmatrix}\right)$;   $G_y = \left(\begin{smallmatrix}
        G_y & g_y\gamma_y \cr \gamma_y g_y^T & \gamma_y^2
    \end{smallmatrix}\right)$;
    \ENDIF
    \IF{$\gamma_x > 0$ and $\gamma_y = 0$}
    \STATE  $Q_{x,y} = \left(\begin{smallmatrix}
        Q_{x,y}  \\ \q_{x}^T Q_y 
    \end{smallmatrix}\right)$; $Q_{x}=(Q_x \; \q_{x})$;  
    $G_{y,x} = \left(\begin{smallmatrix}
        G_{y,x}  & g_y\gamma_x
    \end{smallmatrix}\right)$; $G_x = \left(\begin{smallmatrix}
        G_x  & g_x\gamma_x \\ \gamma_x g_x^T & \gamma_x^2
    \end{smallmatrix}\right)$;
    \ENDIF
    \IF{$\gamma_x = 0$ and $\gamma_y > 0$}
    \STATE  
    $Q_{x,y} = \left(\begin{smallmatrix}
        Q_{x,y} & Q_x^T \q_{y} 
    \end{smallmatrix}\right)$; $Q_{y}=(Q_y \; \q_{y})$; 
     $G_{y,x} = \left(\begin{smallmatrix}
        G_{y,x} \\ \gamma_y g_x^T
    \end{smallmatrix}\right)$; $G_{y} = \left(\begin{smallmatrix}
        G_{y}  & g_y\gamma_y \\ \gamma_y g_y^T & \gamma_y^2
    \end{smallmatrix}\right)$; 
    \ENDIF
    \STATE \textit{If $Q_x$ or $Q_y$ have reached $\widehat r$ columns, do the procedure described in \S\ref{SS=Hemati-low_rank_constraints}}
    \end{algorithmic}
\end{algorithm}
\subsection{The compressed streaming DMD algorithm}
\noindent Equipped with Algorithm  \ref{ALG:Hemati-DMD}, the compressed streaming DMD of the new data, after receiving $\x, \y$, can be compactly written as Algorithm \ref{ALG:Hemati-DMD-2}.
\begin{algorithm}[H]
\caption{\label{ALG:Hemati-DMD-2} $[Z,\Lambda]$ = \\ \texttt{DMDstream\_compress}($Q_x,Q_y,Q_{x,y},G_{y,x},G_x, G_y,\x,\y,\widehat r, \tol{}_1,\tol{}_2,\tol{}_3$)}
  \begin{algorithmic}[1]
\REQUIRE orthonormal $Q_x, Q_y$ from (\ref{eq:RRF-XY}); cross products $Q_{x,y},G_{y,x}, G_x, G_y$ from (\ref{eq-cross-products}); new data  $\x, \y$; maximal dimension $\widehat r$; tolerances $\tol{}_1, \tol{}_2, \tol{}_3$ 
\ENSURE  Ritz pairs $(Z,\Lambda)$ 
\STATE $[Q_x, Q_y, Q_{x,y}, G_{y,x}, G_x, G_y] =$ \\ $ \texttt{Addxy\_compress}(Q_x, Q_y, Q_{x,y}, G_{y,x}, G_x, G_y, \x, \y , \widehat r_x, \widehat r_y, \tol{}_1, \tol{}_2, \tol{}_3)$;
\STATE $\widetilde A = Q_{x,y}G_{y,x}G_x^{-1}$;
\STATE $[W,\Lambda]=\texttt{eig}(\widetilde{A})$;
\STATE $Z=Q_x W$.
\end{algorithmic}
\end{algorithm}

\begin{remark}\label{REM-OnGxGy}
{\em
Since $G_x$ and $G_y$ are symmetric, it suffices to keep and update their upper triangles (diagonals included).
Efficient and robust eigensolvers such as the LAPACK's functions \texttt{xSYEV, xSYEVD} on input require only the upper
or the lower triangle, and the computed eigenvectors (when requested) overwrite the initial array, so that the computations involving  $G_x$ and $G_y$ can be optimized for memory. Note that $G_y$ is computed only for the purpose of reducing the rank.
}
\end{remark}

\begin{example}\label{EX:Cylinder-kappaG}
To illustrate the key numerical parameters and their behavior (low rank representation, orthogonality of the bases, the  residuals (\ref{eq:DMD-residuals}), condition numbers) in the streaming DMD, we use the numerical simulation of the two--dimensional cylinder wake data, available in the supplementary materials \cite{DMD-book-supplement} to the book \cite[\S 2.3]{Kutz-SIAM-BOOK-2016}. 
The observable is the vorticity, and the results are shown in Figure \ref{fig:CondXXT-cylinder-1}. The low rank constraints are $\widehat r = 30$ (first row) and $\widehat r=100$ (second row).

\begin{figure}[ht]
	\includegraphics[width=0.34\linewidth]{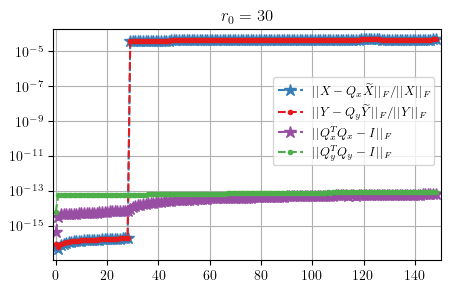}
	\includegraphics[width=0.34\linewidth]{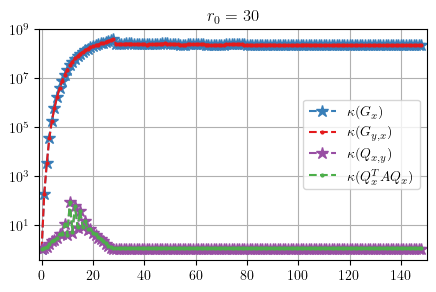}
    \includegraphics[width=0.28\linewidth]{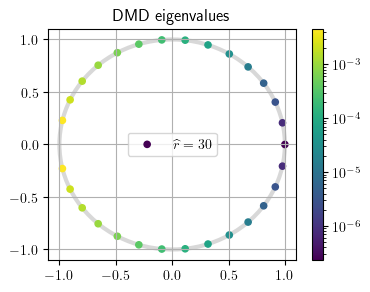} 
    \includegraphics[width=0.34\linewidth]{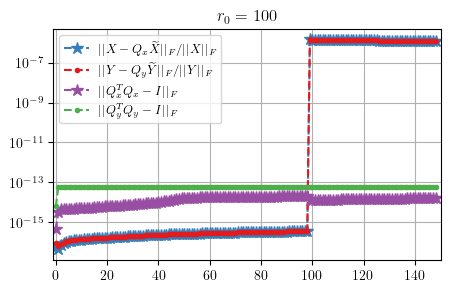}
	\includegraphics[width=0.34\linewidth]{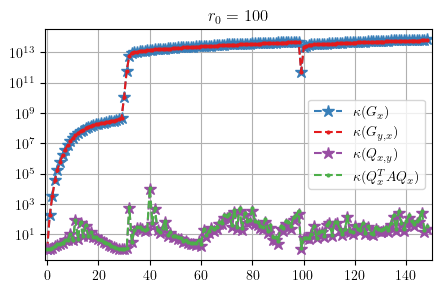}
    \includegraphics[width=0.28\linewidth]{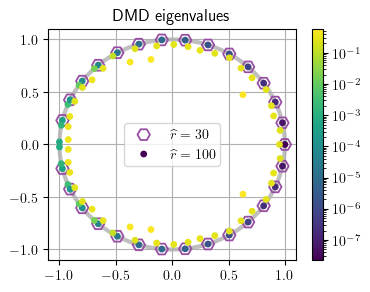}
	\caption{\label{fig:CondXXT-cylinder-1} \emph{Left panels:}
    Test of the low rank representation (\ref{eq:RRF-XY}) and the orthogonality of the bases $Q_x$, $Q_y$. Note how the low rank constraints (\S \ref{SS=Hemati-low_rank_constraints}) here set to $30$ (first row) and $100$ (second row)) impact the accuracy of the representation.
		\emph{Middle panels:} The condition numbers of the key matrices. 
          \emph{Right panels:} The DMD eigenvalues computed after receiving the last snapshot. The colors encode the corresponding residuals. In the second row, the eigenvalues from the first row are shown as purple hexagons. (The system is measure preserving and the underlying Koopman operator is unitary. The true eigenvalues are on the unit circle.)
	} 
\end{figure}
\end{example}

\section{Modifications of the compressed streaming  DMD}\label{S=Modifications}
The compressed streaming DMD performed well on real world data \cite{hemati_williams_rowley_2014}, \cite{Yang-Zhang-strDMD-2022}. 
In this section, we show that, in addition to the details discussed in \S \ref{SS=Hemati-review}, the method can be further improved in terms of numerical robustness and functionality. 

The relations (\ref{eq:AQx}), (\ref{eq:QxTAQX}) actually express $A Q_x$ and $Q_x^T AQ_x$ as
\begin{eqnarray}
A Q_x &=& Q_y\widetilde Y \widetilde X^\dagger = Q_y\widetilde Y \widetilde X^T (\widetilde X\widetilde X^T)^{-1}  , \label{eq:AQx-2}\\
Q_x^T AQ_x &=& Q_x^T Q_y\widetilde Y \widetilde X^\dagger = (Q_x^T Q_y)(\widetilde Y \widetilde X^T)(\widetilde X\widetilde X^T)^{-1} = Q_{x,y}G_{y,x}G_x^{-1} ,\label{eq:QxTAQX-2}
\end{eqnarray}
and the formula (\ref{eq:QxTAQX-2}) is used in \cite{hemati_williams_rowley_2014} as a dimension-reduction technique in cases of large number of data snapshots that evolve on a low dimensional manifold. The column dimensions of $Q_x$ and $Q_y$ are much smaller than the number $n$ of data pairs $(\x,\y)$, and the matrices $\widetilde X$, $\widetilde Y$ are not computed.  
A drawback of this is, as illustrated in Example \ref{EX:Cylinder-kappaG}, that the condition number of $\widetilde X$ is squared: $\kappa_2(G_x)=\kappa_2(\widetilde X\widetilde X^T)=\kappa_2(\widetilde X)^2$.
\noindent Hence, we have to resolve the tradeoff between the numerical stability and the efficiency in the case of large number $m$ of data snapshots of low-dimensional dynamics. This is accomplished in \S \ref{SS=SquaredCond}.

In \S \ref{SS=ExactDMD}, we show how to further enhance the method by computing the Exact DMD vectors (in the framework developed in \S \ref{SS=SquaredCond}), and using them for the the KMD. We conclude this section with the discussion in \S \ref{SS=DiscussFurtherImprove}, that motivates further development in \S \ref{S=Hemati_ONE_BASIS}.
\subsection{Avoiding squared condition number}\label{SS=SquaredCond}
The condition number of $G_x$  may become critical in large dimensional applications, using GPU hardware with lower working precision, e.g. 32 instead of 64 bit arithmetic, or in applications that combine low and high precision. Low/mixed precision reduces the volume of the data and the memory traffic, and it is a way to tackle large scale problems on parallel computers.
In such a setting, it is in particular important to avoid operations that square the condition number. 

\subsubsection{Theory behind the growth of $\kappa_2(G_x)$}\label{SSS=kappaGrowthTheory}
The behaviour of the condition number $\kappa_2(G_x)$, shown in Example \ref{EX:Cylinder-kappaG}, is an insidious and potentially serious problem.
It is easily explained using the following classical theorem from matrix analysis (see e.g. \cite[Theorem 4.3.4]{Horn-Johnson-MA-90}).
\begin{theorem}
Let $H$ be $n\times n$ Hermitian, $v\in\C^n$, and let $\lambda_i(\cdot)$ denote the $i$-th smallest eigenvalue of a Hermitian matrix ($\lambda_1(\cdot)\leq\cdots\leq\lambda_n(\cdot)$). Then the eigenvalues of $H$ and
$H\pm vv^*$ interlace:
\begin{eqnarray*}
&& \lambda_i(H\pm vv^*) \leq \lambda_{i+1}(H)\leq \lambda_{i+2}(H\pm vv^*),\;\;\;i=1,\ldots, n-2 ;\\
&& \lambda_i(H) \leq\lambda_{i+1}(H\pm vv^*) \leq\lambda_{i+2}(H),\;\;\; i=1,\ldots, n-2 .
\end{eqnarray*}
Also, by the Weyl's theorem, $\max_i |\lambda_i(H)-\lambda_i(H\pm vv^*)|\leq \|vv^*\|_2=\|v\|_2^2$, and by the monotonicity principle 
$\lambda_i(H+vv^*)\geq \lambda_i(H)$ for all $i$.
\end{theorem}

\begin{proposition}\label{PROP:kappaG-growth}
\noindent In the case of the rank-one update (\ref{eq:XmewXnewT-1}), with $\gamma_x\neq 0$,
$$
\lambda_{\min}(\widetilde X_{new}\widetilde X_{new}^T) \leq \lambda_{\min}(\widetilde X\widetilde X^T),\;\;\;\lambda_{\max}(\widetilde X_{new}\widetilde X_{new}^T) \geq \lambda_{\max}(\widetilde X\widetilde X^T),
$$
i.e. $\kappa_2(\widetilde X_{new}\widetilde X_{new}^T)\geq \kappa_2(\widetilde X\widetilde X^T)$. In the case of $\gamma_x = 0$, $\lambda_{\max}(\widetilde X_{new} \widetilde X_{new}^T) \geq \lambda_{\max}(\widetilde X \widetilde X^T)$, and the smallest eigenvalue satisfies $\lambda_{\min}(\widetilde X_{new} \widetilde X_{new}^T) \in [ \lambda_{\min}(\widetilde X \widetilde X^T), \lambda_2(\widetilde X \widetilde X^T)]$. (Hence, in the case $\gamma_x=0$ both increase and decrease in condition number are possible.)
\end{proposition}

\subsubsection{Orthogonal factorization based updates}\label{SS=AvoidkappaG}
Suppose we start with the data $(X,Y)$ and we compute the low rank representations
$X=Q_x\widetilde X$, $Y=Q_y\widetilde Y$ \eqref{eq:RRF-XY}. The goal is to apply the DMD for this and subsequently updated/augmented matrices, but without squaring the condition number and, analogously to (\ref{eq:QxTAQX-2}), with complexity that increases with the rank instead of the number of data snapshots $n$. That is, the pseudoinverse $\widetilde X^\dagger$ must be computed without using $(\widetilde X\widetilde X^T)^{-1}$. The solution are suitable orthogonal factorizations.
In the sequel, we will use the notation from \S \ref{SS=AddNewData}.

If we factor $\widetilde X = T\widetilde Q^T$ into orthogonal factorization (RQ with upper triangular $T=R$ or LQ with lower triangular $T=L$, $\widetilde Q^T \widetilde Q=I$) it is possible to avoid the squaring of the condition number that arises when using $\widetilde X \widetilde X^T$. Note that $T$ is square nonsingular.
Then $\widetilde X^\dagger = \widetilde Q T^{-1}$, but this alone is inadequate since $\widetilde Q$ is of same dimensions $n\times r_x$ as $\widetilde X$---the matrix that is not computed because of the dimension $n$. A simple trick resolves this: using \eqref{eq:QxTAQX-2}, we can write
\begin{equation}\label{QxTAQx-sa-LQ}
    Q_x^T AQ_x =  Q_x^T Q_y(\widetilde Y \widetilde Q)T^{-1} =  Q_{x,y}G_{y,q}T^{-1},
\end{equation}
and it remains to devise an updating scheme for $G_{y,q} = \widetilde Y \widetilde Q \in \R^{r_y \times r_x}$ and $T \in \R^{r_x \times r_x}$.
\begin{proposition}\label{PROP:NewUpdate}
    Let $X = Q_x \widetilde X$, $Y = Q_y \widetilde Y$ be the low-rank decompositions (\ref{eq:RRF-XY}) of $X$ and $Y$ with orthonormal $Q_x$, $Q_y$. Assume that the factor $T$ from the orthogonal factorization (RQ or LQ) $\widetilde X=T\widetilde Q^T$, and the product $G_{y,q}=\widetilde Y \widetilde Q$ are available. Let for $X_{new}=(X\;\x_{new})$ and $Y_{new}=(Y\;\y_{new})$ the low rank factors 
    $\widetilde X_{new}$, $\widetilde Y_{new}$ be as in \S \ref{SS=AddNewData}.
    Then, $T_{new}$ and $G_{y,q}^{(new)} = \widetilde Y_{new}\widetilde Q_{new}$ can be computed, respectively, as
    \begin{equation}\label{update-L}
        T_{new} = \left\{ \begin{array}{ll}\begin{pmatrix}
        T & g_x \\
        \mathbf{0} & \gamma_x 
    \end{pmatrix}
        U_G\left(:, 1:r \right) , & \gamma_x >0\cr
       \begin{pmatrix}
        T & g_x \cr
        \end{pmatrix}
        U_G\left(:, 1:r \right), & \gamma_x=0
        \end{array}\right.  \;\;\mbox{and}\;\;
        G_{y,q}^{(new)}= \left\{ \begin{array}{ll}
        \begin{pmatrix}
            G_{y,q}  & g_y \cr \mathbf{0} & \gamma_y
        \end{pmatrix} 
        U_G\left(:, 1:r \right), & \gamma_y >0 \cr 
        \begin{pmatrix}
            G_{y,q}  & g_y 
        \end{pmatrix} 
        U_G\left(:, 1:r \right), & \gamma_y=0
        \end{array} \right. ,
    \end{equation}
where $r\in \{r_x, r_x+1\}$ is the rank of $X_{new}$ and $U_G$ is the product of $r_x$ Givens rotation matrices. In the case of RQ factorization and $\gamma_x>0$, $U_G$ is the $(r_x+1)\times (r_x+1)$ identity matrix.
\end{proposition}
\begin{proof}
 Let $\gamma_x \neq 0$.  Then $\widetilde X=T\widetilde Q^T$  factors (\ref{eq:tildeXnew}) into 
    \[
    \widetilde X_{new} =
    \begin{pmatrix}
        \widetilde X & g_x \\
        \mathbf{0} & \gamma_x 
    \end{pmatrix} =
    \begin{pmatrix}
        T & g_x \\
        \mathbf{0} & \gamma_x 
    \end{pmatrix}\begin{pmatrix}
        \widetilde Q^T & \mathbf{0} \\
        \mathbf{0} & 1
    \end{pmatrix} = \overline T \ \overline Q^T .\;\;(\overline T=\left(\begin{smallmatrix} x & x & x & +\cr
    0 & x & x & + \cr 0 & 0 & x& +\cr 0 & 0 & 0 & *\end{smallmatrix}\right)\;\mbox{or}\; 
    \overline T=\left(\begin{smallmatrix} x & 0 & 0 & +\cr
    x & x & 0 & + \cr x & x & x& +\cr 0 & 0 & 0 & *\end{smallmatrix}\right))
    \]
Hence, in the RQ case ($T=R$) the new factorization is readily available: $\widetilde X_{new}=R_{new}\widetilde Q_{new}^T$ with  $R_{new} = \overline T$, $\widetilde Q_{new} = \overline Q$. In the case of LQ factorization ($T=L$), $\overline{T}$ must be transformed into the lower triangular form.
This is accomplished by a sequence of Givens rotations $G_i$ of the column pairs $(i, r_x+1)$, $i=1,2,..., r_x$, that sequentially nullify the first $r_x$ elements of the last column. Then 
$$
\widetilde X_{new} = (\overline{T} G_1^T G_2^T \cdots G_{r_x}^T )(G_{r_x} \cdots G_2 G_1 \overline Q^T) = L_{new} \widetilde Q_{new}^T. \;\;
(L_{new}=\left(\begin{smallmatrix} x & 0 & 0 & 0\cr
    x & x & 0 & 0 \cr x & x & x& 0\cr x & x & x & x\end{smallmatrix}\right))
$$
If $\gamma_x = 0$, corrections are needed for both the RQ and the LQ factorization, because
\[
    \widetilde X_{new} = \begin{pmatrix}
        \widetilde X & g_x
    \end{pmatrix} 
    =\begin{pmatrix}
        T & g_x 
    \end{pmatrix}\begin{pmatrix}
        \widetilde Q^T & \mathbf{0} \\
        \mathbf{0} & 1
    \end{pmatrix} = \overline T \ \overline Q^T. \;\;(\overline T=\left(\begin{smallmatrix} x & x & x & +\cr
    0 & x & x & + \cr 0 & 0 & x& + \end{smallmatrix}\right)\;\mbox{or}\; 
    \overline T=\left(\begin{smallmatrix} x & 0 & 0 & +\cr
    x & x & 0 & + \cr x & x & x& +\end{smallmatrix}\right))
\]
Here, $\overline T$ is $r_x \times (r_x+1)$ matrix with triangular principal $r_x\times r_x$ submatrix. In both cases ($T=R$ and $T=L$) Givens rotations are used, this time to completely nullify the last column of $\overline T$.
The result is an  updated orthogonal factorization
    \[
    \widetilde X_{new} = (\overline{T} G_1^T G_2^T \cdots G_{r_x}^T )(G_{r_x} \cdots G_2 G_1 \overline Q^T) 
    = \begin{pmatrix}
        T_{new} & \mathbf{0} 
    \end{pmatrix} \begin{pmatrix}
        Q_{new}^T \\ q_2^T
    \end{pmatrix} = T_{new} Q_{new}^T .
    \]
If $T=R$, the $G_i$'s are applied bottom-to-top and right-to-left, i.e. nullifying elements of final column from the last element to the first by performing rotations on columns $(i, r_x+1)$ for $i=r_x,...,2,1$.
If $T=L$, the pattern is top-to-bottom and left-to-right, with pivotal indices $(i, r_x+1)$, $i=1,2,..., r_x$ (as in the case $\gamma_x \neq 0$).
Clearly, $Q_{new}$ is orthonormal.
This proves the first relation in \eqref{update-L}, where we set $U_G=G_1^TG_2^T\cdots G_{r_x}^T$.
For the update of $G_{y,q} = \widetilde Y \widetilde Q^T$, in the case $\gamma_y \neq 0$,  use $\widetilde Q_{new} = \overline Q U_G$ to obtain
    \begin{equation*}
    G_{y,q}^{(new)} = \widetilde Y_{new} \widetilde Q_{new} 
    = \begin{pmatrix}
        \widetilde Y & g_y \\ \mathbf{0}^T & \gamma_y
    \end{pmatrix}\begin{pmatrix}
        \widetilde Q & \mathbf{0} \\ \mathbf{0} & 1
    \end{pmatrix}U_G
    = \begin{pmatrix}
        G_{y,q} & g_y \cr \mathbf{0} & \gamma_y
    \end{pmatrix} U_G .
    \end{equation*}
    In a completely analogous way one can show that the relation (\ref{update-L}) holds for $\gamma_y = 0$, as claimed. 
\end{proof}
\noindent When using (\ref{QxTAQx-sa-LQ}) instead of (\ref{eq:QxTAQX-2}), the inverse is used with  $\kappa_2(T) = \kappa_2(\widetilde X) = \kappa_2(X)=\sqrt{\kappa_2(G_x)}$. 
The Givens rotation matrices $G_i$ only rely on $T$ and $g_x$, $\gamma_x$---the matrix  $\widetilde Q$ is implicit. Thanks to this, the dimension $n$ is completely avoided. In the scenario $r_x,r_y \ll n \ll m$ applying Givens rotations is inexpensive and requires $\mathcal{O}(r^2)$ operations ($r \sim r_x, r_y$). In fact, in the case $T=R$ and $\gamma_x>0$ no extra effort is needed. In that sense the RQ factorization is preferred to LQ.

On the other hand, in an application scenario where before the streaming phase we start with a relatively large batch $(X,Y)$ and its low rank representations, the initial RQ factorization of $\widetilde X$ must be computed. Some software packages do not contain the RQ factorization (such as Matlab), or do (such as LAPACK, MAGMA, scipy.linalg in Python) but it may be slower\footnote{for $100$ instances of random matrix $A \in \R^{10^5 \times 200}$ it took $\approx 321$s for QR factorizations and $\approx 721$s for RQ factorizations of a $100$ random $A \in \R^{200 \times 10^5}$} than the QR factorization. In addition, some additional care is needed when using these subroutines for our purposes. For example, in our case of $r_x\times n$ matrix $\widetilde X$, the LAPACK's \texttt{xGERQF} computes the RQ factorization as 
$$
\widetilde X = \left( \begin{smallmatrix} 0 & 0 & 0 & * & * \cr
0 & 0 & 0 & 0& *\end{smallmatrix}\right) 
\left( \begin{smallmatrix} \bullet & \bullet & \bullet & \bullet & \bullet \cr 
\bullet & \bullet & \bullet & \bullet & \bullet \cr
\bullet & \bullet & \bullet & \bullet & \bullet \cr
\star & \star & \star & \star & \star \cr 
\star & \star & \star & \star & \star\end{smallmatrix}\right), \;\;\mbox{which can be used as}\;\;
\widetilde X = 
\left( \begin{smallmatrix}  * & * \cr
 0& *\end{smallmatrix}\right)
\left( \begin{smallmatrix} 
\star & \star & \star & \star & \star \cr 
\star & \star & \star & \star & \star\end{smallmatrix}\right).
$$
Also, LAPACK contains the function \texttt{xLASWLQ} for computing the LQ factorization of wide matrices. The LQ factorization of $\widetilde X$ is also easily obtained by transposing the QR factorization of $\widetilde X^T$.

\begin{remark}
    {\em
    It is well known that QR factorization without pivoting can be numerically unstable when performed on matrices with increasing row norms \cite{higham-asna}. This is equivalent to the problem of computing RQ factorization of a matrix whose column norms decrease significantly. Therefore, we suggest using LQ factorization if declining norm of snapshots is expected.}
\end{remark}

\begin{remark}
    {\em
    Initialization can also be done using SVD. If $X = U \Sigma V \approx U_k \Sigma_k V_k$ where $k$ is the numerical rank, then $Q_x = U_k$, $T = \Sigma_k$ and $\widetilde Q = V_k$. Since $T$ is diagonal it can be regarded as both upper and lower triangular matrix.}
\end{remark}

If the first factorization is the LQ, it can be transformed into the RQ at the beginning, and the streaming part will maintain the upper triangular form. A computation-free method is 
to use the $r_x\times r_x$ permutation matrix $\Pi$ that reverses the order and write 
$X=Q_x\widetilde X = Q_x L\widetilde Q^T = (Q_x\Pi)(\Pi^T L\Pi)(\Pi^T\widetilde Q^T)$, where $R=\Pi^T L\Pi$ is upper triangular. This reshuffles the columns of $Q_x$, making this method less attractive because operations that involve possibly large dimensional $Q_x$ should be avoided.\footnote{One could make this permutation implicit and keep $Q_x$ as the pair $(Q_x,\Pi)$, with the necessary adjustments when using it.}
Therefore, the transformation from the LQ into the RQ form is preferably based on multiplications of $L$ from the right using suitable orthogonal matrices. This can be done by a sequence of updates 
$
L = L (H_i\oplus I_{r_x-i}),\;\;i=r_x, r_x-1,\ldots 2,
$
where $H_i$ is the Householder reflector that transforms the last row of $L(1:i,1:i)$ by annihilating 
$L(i,1:i-1)$.

The new procedure is summarized in Algorithm \ref{ALG:Hemati-DMD-LQ}.
We will use $\mathbf \Gamma(\alpha_i, \alpha_j)$ to denote the $2 \times 2$ Givens rotation matrix that when applied to vector $\alpha = (\alpha_1,..., \alpha_{end})^T$ nullifies $\alpha_j$ onto $\alpha_i$. Algorithm \ref{ALG:Hemati-DMD-LQ} can be expanded so that $Q_{x,y}$ is updated as well. For simplicity we leave that out, but it is equivalent to procedure described in algorithm \ref{ALG:Hemati-DMD}. It is recommended to update $Q_{x,y}$ as well when $m \gg n$ and when DMD eigenvalues and modes are calculated frequently.

\begin{algorithm}
    \caption{\label{ALG:updateYQT}[$G_{y,q}, T$] = \texttt{updateYQT}($G_{y,q}, T, \widetilde x, \widetilde y, \gamma_x, \gamma_y, \mathrm{tri}$)}
    \begin{algorithmic}[1]
        \REQUIRE Matrices $G_{y,q}, T$ from \eqref{QxTAQx-sa-LQ}; low-rank representation of new snapshots $\widetilde x, \widetilde y$; nonnegative numbers $\gamma_x, \gamma_y$ as described in section \ref{SS=AddNewData};  character tri $\in \{\texttt{'U', 'L'}\}$ indicating whether $T$ is upper or lower triangular;
        \ENSURE Updated matrices $G_{y,q}, T$
        \STATE \algorithmicif \ $\gamma_x > 0$ \ \algorithmicthen \ $T = \left(\begin{smallmatrix}
        T \\ \mathbf{0}
    \end{smallmatrix}\right)$ \algorithmicend
    \STATE \algorithmicif \ $\gamma_y > 0$ \ \algorithmicthen \ $G_{y,q} = \left(\begin{smallmatrix}
        G_{y,q} \\ \mathbf{0}
    \end{smallmatrix}\right)$ \algorithmicend 
    \STATE $G_{y,q} = \begin{pmatrix}
        G_{y,q} & \widetilde y 
    \end{pmatrix}$;  $T = \begin{pmatrix} 
        T & \widetilde x
    \end{pmatrix}$; $r_x = \texttt{size}(T, 1)$
    \IF{\texttt{tri} $=$ \texttt{'L'}}
    \FOR{$i = 1,2,..., r_x-1$} 
    \STATE $G_{y,q}(:, [i \ r_x]) = G_{y,q}(:, [i \ r_x]) \mathbf \Gamma\left(T_{i,i}, T_{i, r_x}\right)^T$; \ $T(i:r_x, [i \ r_x]) = T(i:r_x, [i \ r_x]) \mathbf  \Gamma\left(T_{i,i}, T_{i, r_x}\right)^T$
    \ENDFOR
    \ELSIF{$\gamma_x = 0$}
    \FOR{$j = r_x-1,...,2,1$}
    \STATE $G_{y,q}(:, [j \ r_x] = G_{y,q}(:, [j \ r_x]) \mathbf \Gamma\left(T_{j,j}, T_{j, r_x}\right)^T$; \ $T(1:j, [j \ r_x] = T(1:j, [j \ r_x]) \mathbf \Gamma\left(T_{j,j}, T_{j, r_x}\right)^T$
    \ENDFOR
    \ENDIF
    \STATE \algorithmicif \ $\gamma_x = 0$ \ \algorithmicthen \ $T = T(:, 1:end-1)$; $G_{y,q} = G_{y,q}(:, 1:end-1)$ \algorithmicend
    \end{algorithmic}
\end{algorithm}

\begin{algorithm}[ht]
\caption{\label{ALG:Hemati-DMD-LQ} [$Q_x, Q_y, G_{y,q}, T$] = \texttt{AddxyTQ}($Q_x, Q_y, G_{y,q}, T, \x, \y , \tol{}_1, \tol{}_2$, tri)}
    \begin{algorithmic}[1]
    \REQUIRE orthonormal $Q_x, Q_y$ from (\ref{eq:RRF-XY}); $G_{y,q}$, $T$ from \eqref{QxTAQx-sa-LQ}; new data  $\x, \y$; tolerances $\tol{}_1, \tol{}_2, \tol{}_3$; character tri $\in \{\texttt{'U', 'L'}\}$ indicating whether $T$ is upper or lower triangular
    \ENSURE  Updated matrices $Q_x, Q_y, G_{y,q}, T$ 
    \STATE $[\q_{x},\widetilde x, \gamma_x]$ = GS\_update($Q_x, \x, \tol{}_1, \tol{}_2$);
    \STATE $[\q_{y},\widetilde y, \gamma_y]$ = GS\_update($Q_y, \y, \tol{}_1, \tol{}_2$)
    \IF{$\gamma_x > 0$}
    \STATE $Q_x = \begin{pmatrix}
        Q_x & \q_x
    \end{pmatrix}$;  $\widetilde x = \left( \begin{smallmatrix}
        \widetilde x \\ \gamma_x
    \end{smallmatrix}\right)$
    \ENDIF
    \IF{$\gamma_y > 0$}
    \STATE  $Q_y = \begin{pmatrix}
        Q_y & \q_y
    \end{pmatrix}$; $\widetilde y = \left( \begin{smallmatrix}
        \widetilde y \\ \gamma_y
    \end{smallmatrix}\right)$
    \ENDIF
    \STATE $[G_{y,q}, T] = \texttt{updateYQT}(G_{y,q}, T, \widetilde x, \widetilde y, \gamma_x, \gamma_y, \mathrm{tri})$
    \end{algorithmic}
\end{algorithm}

\subsubsection{Using updated Cholesky factors}\label{SS:UseUpdatedCholFactor}
The update described in Proposition \ref{PROP:NewUpdate} seems more expensive than the original update of $G_{y,x}$ and $G_x$ (see Algorithm \ref{ALG:Hemati-DMD}).  It is not. Namely, applying $G_x^{-1}$ requires Cholesky factorization ($O(r_x^3)$) and two triangular systems ($O(r_x^2r_y)$). This motivates another, more efficient, approach: to keep updating the Cholesky factor of $G_x$. 
This is done as follows. Let $G_x+g_x g_x^T=L_x L_x^T$. Then the updated  Cholesky factor is 
$$
G_{x,new}=\begin{pmatrix} G_x+g_x g_x^T & g_x\gamma_x\cr \gamma_x g_x^T & \gamma_x^2\end{pmatrix} = 
\begin{pmatrix} L_x & \mathbf{0} \cr b_x^T & \beta_x\end{pmatrix}
\begin{pmatrix} L_x^T & b_x \cr \mathbf{0} & \beta_x\end{pmatrix}\;\;\mbox{with}\;\;
b_x=\gamma_x L_x^{-1}g_x,\;\;\beta_x=\sqrt{\gamma_x^2-b_x^T b_x},
$$
and the $O(r_x^3)$ contribution to the cost (Line 2 in Algorithm \ref{ALG:Hemati-DMD-2}) can be replaced with $O(r_x^2)$. Note that this update may fail in finite precision computation if $G_{x,new}$ is ill-conditioned. In other words, the problem of the squared condition numbers has not been removed. This makes the method proposed in \S \ref{SS=AvoidkappaG} competitive: more numerical robustness is achieved with the same complexity. Computation of $L_x$ from the Cholesky factor of $G_x$ is outlined in \S \ref{SS=NewZhang-CholeskyUpdate}.
\subsubsection{TQ based updates with enforced maximal ranks}\label{SSS=TQ-Rank-Enforce}
It remains to adapt the new formulas to the upper limit $\widehat r$ for the number of columns of $Q_x$, as in section \ref{SS=Hemati-low_rank_constraints}. Here we do not have $G_x$, but we can use the orthogonal decomposition $X\approx Q_x T\widetilde Q^T$. It suffices to compute the SVD of $T$, $T = U_T \Sigma_T V_T^T$.  Then $T \approx U_T(:, 1\! :\!\rho_x) \Sigma_T(1\! :\!\rho_x, 1\! :\!\rho_x) V_T(:, 1\! :\!\rho_x)^T = \widetilde U_{\rho_x} \widetilde \Sigma_{\rho_x} \widetilde V_{\rho_x}$, where 
$\rho_x$ is the new rank as in \S \ref{SS=Hemati-low_rank_constraints}. Since $X = (Q_x U_T) \Sigma_T (\widetilde Q V_T)^T$ is the SVD, we have optimal low rank approximation
\begin{equation*}
    X \approx Q_x \widetilde U_{\rho_x} \widetilde \Sigma_{\rho_x} (\widetilde Q \widetilde V_{\rho_x})^T = Q_{x,new} T_{new} \widetilde Q_{new}^T,\;\;
    \mbox{where}\;\; Q_{x,new} = Q_x \widetilde U_{\rho_x}, \; T_{new} = \widetilde \Sigma_{\rho_x}, \; \widetilde Q_{new} = \widetilde Q \widetilde V_{\rho_x}.
\end{equation*}
Now $G_{y,q}^{(new)} = \widetilde Y\widetilde Q_{new} = (Y\widetilde Q)\widetilde V_{\rho_x} = G_{y,q}\widetilde V_{\rho_x}$.

To enforce the upper limit $\widehat r$ for the number of columns of $Q_y$, we may proceed as before and keep updating the matrix $G_y$; see Remark \ref{REM-ry-constraint} and Remark \ref{REM-OnGxGy}. If squaring the condition number is to be avoided, then we can use the TQ decomposition of $\widetilde Y = T_y \widetilde Q_y^T$. This requires 
changing the global structure of the algorithm, in order to suppress the dimension $n$. To that end, using the TQ factorizations $\widetilde X=T_x\widetilde Q_x^T$ and 
$\widetilde Y = T_y \widetilde Q_y^T$,  we rewrite \eqref{QxTAQx-sa-LQ}  as
\begin{equation}\label{QxTAQx-sa-LQ-Y}
    Q_x^TAQ_x = (Q_x^TQ_y)T_y(\widetilde Q_y^T \widetilde Q_x) T_x^{-1} ,
\end{equation}
where $T_y \in \R^{r_y \times r_y}$, $\widetilde Q_y^T \widetilde Q_x \in \R^{r_y \times r_x}$ and $T_x \in \R^{r_x \times r_x}$.  Using the SVD $T_y = U_{T_y} \Sigma_{T_y} V_{T_y}^T$, we proceed as above with the approximation 
$$
Y \approx (Q_y U_{T_y}(:, 1\! :\!\rho_y)) \Sigma_{T_y}(1\! :\!\rho_y, 1\! :\!\rho_y) (\widetilde Q_y V_{T_y}(:, 1\! :\!\rho_y))^T = Q_{y,new} T_{y,new} \widetilde Q_{y,new}^T \ ,
$$
where $Q_{y,new} = Q_y U_{T_y}(:, 1\! :\!\rho_y)$, $T_{y,new}=\Sigma_{T_y}(1\! :\!\rho_y, 1\! :\!\rho_y)$, $\widetilde Q_{y,new}=\widetilde Q_y V_{T_y}(:, 1\! :\!\rho_y)$. 
Note that $\widetilde Q_{y,new}$ is not computed. Instead, it suffices to update $\widetilde Q_{y,new}^T \widetilde Q_x = V_{T_y}(:, 1\! :\!\rho_y)(\widetilde Q_y^T \widetilde Q_x)$.
Hence, in this version three (instead of two) ``small'' matrices must be updated: $T_x$, $T_y$ and $\widetilde Q_{y,x} = \widetilde Q_y^T \widetilde Q_x$.

\subsubsection{Example - vorticity data from Example \ref{EX:Cylinder-kappaG}}\label{SSS:example-vorticity}
To simulate working in single precision, we update $G_{y,x}$ and $G_x$ in double precision, but then round them to single precision after each update and before computing modes and predictions. The same thing is done to $L$ and $\widetilde Q$. The results are shown in figures \ref{fig:vorticity_single} and \ref{fig:rec_errors_vorticity}. In figure \ref{fig:vorticity_single}, the algorithm starts from a batch of $n_0=15$ snapshots. One snapshot is added then a prediction one step ahead is made,  as outlined in \S \ref{SSS=DMD-task2}. The displayed results are obtained after adding $15$ snapshots. In figure \ref{fig:rec_errors_vorticity} two graphs show relative errors starting from a batch of $n_0 = 20$ snapshots when making predictions (a) one step ahead and (b) five steps ahead using the original proposition (HWR-sDMD) with $G_x$ and using TQ decomposition (TQ-sDMD), both in simulated single precision. For comparison, relative errors in double precision for HWR-sDMD are also depicted and coincide with errors for TQ-sDMD in simulated single precision\footnote{Relative errors of prediction using TQ-sDMD in double precision are omitted since they agree with the ones obtained when using single precision}. 
\vspace{-3mm}
\begin{figure}[ht]
    \centering
    \includegraphics[width=0.9\linewidth]{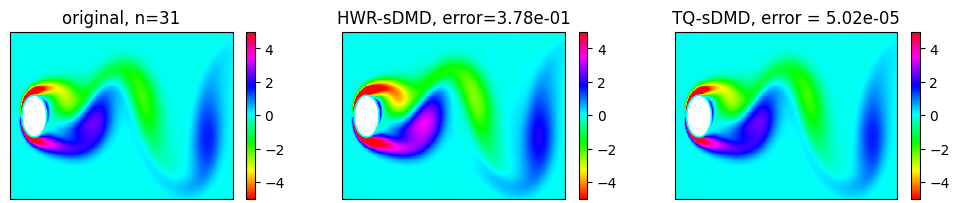}
    \caption{Comparison of predicted snapshots using the original algorithm HWR-sDMD (middle) and TQ-sDMD (using the LQ decomposition) (right). Left figure is the original snapshot at $n=31$. 
  The apparent differences in the prediction errors ($\texttt{5.02e-5}$ vs. $\texttt{3.78e-1}$)  show that TQ-sDMD has improved robustness
  to ill-conditioning and finite precision roundoff errors. While HWR-sDMD has difficulties even for one step prediction, TQ-sDMD retains accuracy even for multiple steps, see figure \ref{fig:pred-many-steps}.
  }
    \label{fig:vorticity_single}
\end{figure}
\begin{figure}[!h]
    \centering
    \begin{subfigure}[t]{0.5\textwidth}
        \centering
        \includegraphics[width=1\linewidth]{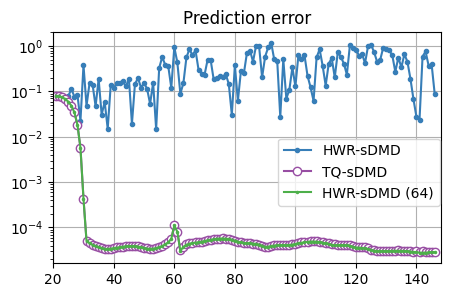}
    \caption{one step ahead}
    \end{subfigure}%
    \begin{subfigure}[t]{0.5\textwidth}
        \centering
        \includegraphics[width=1\linewidth]{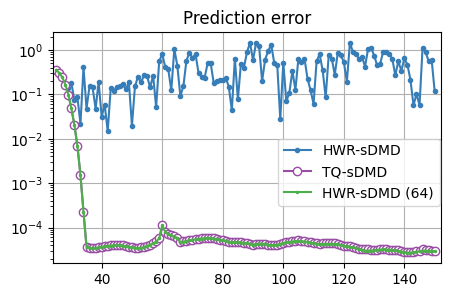}
    \caption{five steps ahead}
    \end{subfigure}
    \caption{Relative error when predicting one or five steps ahead when updating $A$ by using HWR-sDMD and TQ-sDMD and rounding to single precision before calling DMD, compared to updates by HWR-sDMD with calculation done entirely in double precision. 
    Here HWR-sDMD (64) denotes execution in the 64-bit double precision arithmetic.\label{fig:pred-many-steps}}
\label{fig:rec_errors_vorticity}
\end{figure}

\subsection{Streaming DMD with Exact DMD vectors}\label{SS=ExactDMD}
An often used variation of the DMD is the Exact DMD  \cite{tu-rowley-dmd-theory-appl-2014}, with the Exact DMD vectors, that are defined in the range of $Q_y$---see Lines 4 and 8 in Algorithm \ref{ALG:DMD:RRR} and the low rank representation (\ref{eq:RRF-XY}) of $Y$. In this section we show how to compute the Exact DMD vectors in the framework of the compressed streaming DMD \cite{hemati_williams_rowley_2014}, using both the cross products (\ref{eq:QxTAQX-2}) from \cite{hemati_williams_rowley_2014}  and  the new compression from \S \ref{SS=AvoidkappaG}, which avoids squaring the condition number of $\widetilde X$.

\begin{proposition}\label{prop:exact-sDMD}
Let $(w_i,\lambda_i)$, be an eigenpair of $Q_x^T A Q_x=Q_{x,y}G_{y,x}G_x^{-1}$. Then the corresponding Exact DMD eigenvector is 
$
\z_i^{(ex)} = Q_y\widetilde Y\widetilde X^\dagger w_i = Q_y G_{y,x}G_x^{-1}w_i,
$
and we have $A \z_i^{(ex)} = \lambda_i \z_i^{(ex)}$, where $A=YX^\dagger$. Hence, if $\z_i^{(ex)}\neq \mathbf{0}$ then it is an eigenvector of $A$. In terms of the compression (\ref{QxTAQx-sa-LQ}) 
\begin{equation}\label{eq:zex-compact-2}
\z_i^{(ex)} = Q_y  w_i^{(ex)},\;\;\;w_i^{(ex)}=G_{y,q}T^{-1}w_i.
\end{equation}
\end{proposition}
\begin{proof}
The corresponding Ritz vector of the DMD matrix is $\z_i=Q_x w_i$. The Exact DMD vector is the result of one step of the power method (see \cite{Drmac-DMD-TOMS-2024}) i.e. 
\begin{equation*}
\z_i^{(ex)}=YX^\dagger Q_x w_i=Q_y\widetilde Y\widetilde X^\dagger Q_x^T Q_x w_i=Q_y \widetilde Y\widetilde X^\dagger w_i 
=
Q_y \widetilde Y\widetilde X^T (\widetilde X\widetilde X^T)^{-1}w_i = Q_y G_{y,x}G_x^{-1}w_i,
\end{equation*}
as claimed. Checking the formula (\ref{eq:zex-compact-2}), based on (\ref{QxTAQx-sa-LQ}), is straightforward. Further, 
$$
YX^\dagger \z_i^{(ex)} =Q_y\widetilde Y\widetilde X^\dagger {Q_x^T Q_y\widetilde Y\widetilde X^\dagger w_i}= \lambda_i Q_y\widetilde Y\widetilde X^\dagger  w_i = \lambda_i \z_i^{(ex)}.\;\;(Q_x^T Q_y\widetilde Y\widetilde X^\dagger w_i=Q_x^T A Q_x w_i=\lambda_i w_i) .
$$
\end{proof}
\noindent It is important to keep in mind that the $\z_i^{(ex)}$'s are exact eigenvectors only for the matrix $A_{ex}=YX^\dagger$. Since $m > n$, the matrix that actually generated the data $X$, $Y$ is not uniquely determined, and the exactness for $A_{ex}$ does not imply exactness for the actual underlying operator.  In that sense, calling these vectors \emph{exact} would be a  misnomer. However, the Exact DMD vectors can be used as alternative to the classical DMD eigenvectors. In the context of the (compressed) streaming DMD with low rank representations (\ref{eq:RRF-XY}), considered here, these vectors have an interesting and useful advantage, which we describe next. 

\subsubsection{KMD with the Exact DMD vectors} Suppose that we want to perform the KMD analysis (\S \ref{SSS=DMD-task2}) of the last received snapshot $y_n=Y(:,n)$. (More generally, we might want to analyze certain number of most recent snapshots in order to make forecasting and to apply control.)
Consider the following two representations that must find the coefficients $\alpha_1^{(n)},\ldots,\alpha_\ell^{(n)}$ to minimize the least squares errors
\begin{eqnarray}
\!\!\!\!\!\!\!    \|y_n - \sum_{i=1}^\ell \alpha_i^{(n)} \z_i\|_2 \!\!&=& \!\!\|Q_y \widetilde y_n - \sum_{i=1}^\ell \alpha_i^{(n)} Q_x w_i\|_2 \;\;\mbox{or}\label{eq:ynDMDz}\\
\!\!\!\!\!\!\!    \|y_n - \sum_{i=1}^\ell \alpha_i^{(n)} \z_i^{(ex)}\|_2 \!\! &=& \!\!\|Q_y \widetilde y_n - \sum_{i=1}^\ell \alpha_i^{(n)} Q_y w_i^{(ex)}\|_2 = \|\widetilde y_n - \sum_{i=1}^\ell \alpha_i^{(n)} w_i^{(ex)}\|_2 .\label{eq:ynDMDzex}
\end{eqnarray}
In (\ref{eq:ynDMDz}) the dimension of the problem is $m\times \ell$, and in (\ref{eq:ynDMDzex}) the invariance of the Euclidean norm under orthogonal transformations reduces the problem to $r_y\times\ell$. In other words, KMD can be performed in the low rank representation of the data.

Similarly, in the case of single trajectory ($y_i=x_{i+1}$) if the reconstruction (\ref{eq:modal-dcomp}) includes $y_n$, then dropping the first snapshot we have 
(using $\y_i=Q_y\widetilde y_i$, $\z_j^{(ex)}=Q_y w_j^{(ex)}$ and the unitary invariance of the norm)
$$
\sum_{i=1}^{n} w_i^2 \| \y_i - \sum_{j=1}^{\ell} \z_{\varsigma_j}^{(ex)} \alpha_j \lambda_{\varsigma_j}^{i-1}\|_2^2  = 
\sum_{i=1}^{n} w_i^2 \| \widetilde y_i - \sum_{j=1}^{\ell} w_{\varsigma_j}^{(ex)} \alpha_j \lambda_{\varsigma_j}^{i-1}\|_2^2 \longrightarrow \min_{\alpha_j} .
$$
With the DMD vectors $\z_i=Q_x w_i$, this reduction of the dimension of the least squares problem is not possible because the involved vectors are represented in two bases; see (\ref{eq:ynDMDz}). This reduction into a smaller subspace can significantly reduce computation time as will be shown later in figure \ref{fig:blast-error-single-onebas}.

\subsection{Discussion and need for further improvements}\label{SS=DiscussFurtherImprove}
The introduction of the Exact DMD vectors and the discussion of the KMD in \S \ref{SS=ExactDMD} raise an important practical issue. 
That is, keeping two separate bases $Q_x$ and $Q_y$ contains an unnecessary overhead that increases the memory footprint of the algorithm and its computational complexity.

This is particularly apparent in the case of the data from a single log trajectory. In the single trajectory setting, $\x_{new}=Y(:,end)$ and $\y_{new}$ is the new received snapshot. Note that $Y(:,end)$ has representation in the basis $Q_y$, and the updating procedure requires its representation in the $Q_x$--basis. This means that each vector is processed twice.  
Moreover, trajectory with $n$ pairs $(\x_i,\y_i\equiv \x_{i+1})$ contains $n+1$ vectors and yet two $m\times n$ orthonormal matrices are computed to represent them. Clearly, the ranges of $Q_x$ and $Q_y$ must have large intersection -- it is of dimension (at most) $n-1$. Maintaining two orthonormal bases for the ranges of $X$ and $Y$ is not optimal, in particular in large scale out-of-core computations, and in fact it is not necessary.

As discussed in \S \ref{SSS=DMD-task-1}, the  residuals (\ref{eq:DMD-residuals}) provide information on the quality of the approximate eigenpairs and should be returned together with the computed Ritz pairs. Using all computed eigenpairs in the KMD (\S \ref{SSS=DMD-task2}), including those with large residuals, precludes correct identification of coherent structures and induces large error if the KMD is used for forecasting.  

In this framework the residuals can be computed using (\ref{eq:AQx}) as $\|r_i\|_2$, where
\begin{equation}\label{eq:resid-two-bases}
r_i = A Q_x w_i - \lambda_i Q_x w_i = Q_y\widetilde Y \widetilde X^\dagger w_i - \lambda_i Q_x w_i = \z_i^{(ex)} - \lambda_i Q_x w_i.
\end{equation}
This requires lifting the vectors into the $m$--dimensional space, which makes computation of the residuals expensive, even  in  out-of-core computation for extremely large dimensions.

Other functionalities of an on-line streaming data processing are of interest.
For example, to capture changes in short-term dynamics, DMD based analysis must be flexible and allow for dynamically varying data window that also includes discarding oldest data (single snapshot or a block of snapshots). Removing $\ell$ oldest snapshot pairs is in batched processing fairly simple: $\widetilde X$, $\widetilde Y$ are simply updated by $\widetilde X\leftarrow \widetilde X(:,\ell+1:end)$ and $\widetilde Y\leftarrow \widetilde Y(:,\ell+1:end)$, respectively.
This may change the ranks of $\widetilde X$, $\widetilde Y$. A technical difficulty arises if the matrices are used in the cross--products $G_x=\widetilde X\widetilde X^T$, $G_{y,x}=\widetilde Y\widetilde X^T$. Namely, the update
$G_x\leftarrow G_x - \widetilde X(:,1:\ell)\widetilde X(:,1:\ell)^T$
is simply not feasible because all information on individual snapshots is blended into the cross products $G_x, G_y$ and $G_{y,x}$. This limits the functionality of the application;  it is impossible to gradually keep forgetting the oldest data or simply resize the active window by truncating all but a few most recent snapshots. 
We address these questions in a separate work.

%% file: sections/one_basis.tex
\graphicspath{{\subfix{../}}}

\newcommand{\rx}{\ensuremath{\widetilde S_x}}
\newcommand{\ry}{\ensuremath{\widetilde S_y}}

\section{One basis low-rank sDMD}\label{S=Hemati_ONE_BASIS}
Consider a sequence of data snapshots $\s_1,\s_2,\ldots,\s_n, \s_{n+1}$ from a single long trajectory, so that $\x_i=\s_i$, $\y_i=\s_{i+1}$ for all $i$, i.e. $X = S(:, 1:n)$ and $Y = S(:, 2:n+1)$. In the same way as in \S \ref{SS=Hemati-review}, the data matrix $S=(\s_1 \ \ldots \ \s_{n+1})$ can be written in a low-rank factored form 
\begin{equation}\label{eq:S=QtS}
    S = Q\widetilde S, \quad Q\in \R^{m \times r}, \;\;\widetilde S \in \R^{r \times (n+1)}.
\end{equation}
If this factorization is computed from scratch, i.e. in the streaming mode from the very first snapshot, then
initially $\widetilde S_{11}=\|\s_1\|_2$, $Q(:,1)=\s_1/\widetilde S_{11}$, and the updating procedures
from \S \ref{SS=Hemati-review} apply. The matrix $S$ will have an upper-staircase structure that depends on whether the new added snapshot adds new dimension and a new column of $Q$. We illustrate a few scenarios in the simple case of $n=5$:
$$
S = \left(\begin{smallmatrix} q_1 & q_2 & q_3 & q_4 & q_5 & q_6 \end{smallmatrix}\right)
\left(\begin{smallmatrix} 
\circledast & * & * & * & * & * \cr 
0 & \circledast & * & * & * & *  \cr 
0 & 0 & \circledast & * & * & *  \cr 
0 & 0 & 0 & \circledast & * & *  \cr
0 & 0 & 0 & 0 & \circledast & *  \cr
0 & 0 & 0 & 0 & 0 & \circledast  \cr
\end{smallmatrix}\right),\;
S = \left(\begin{smallmatrix} q_1 & q_2 & q_3 & q_4\end{smallmatrix}\right)
\left(\begin{smallmatrix} 
\circledast & * & * & * & * &  *\cr 
0 & \circledast & * & * & * &  * \cr 
0 & 0 & 0 & \circledast & * &  * \cr 
0 & 0 & 0 & 0 &  0 & \circledast \end{smallmatrix}\right),\;
S = \left(\begin{smallmatrix} q_1 & q_2 & q_3 \end{smallmatrix}\right)
\left(\begin{smallmatrix} 
\circledast & * & * & * & * &  *\cr 
0 & \circledast & * & * & * &  * \cr 
0 & 0 & 0 & \circledast & * &  * 
\end{smallmatrix}\right) .
$$
If in the initial step a batch of snapshot becomes available, it is important for compression of $A$ to first perform rank-revealing factorization of $X$ and then extend the rank-revealing factorization by $\y = Y(:, end)$ to obtain rank revealing factorization of $S = (X \ \y) = Q\widetilde S$. In this way, orthogonal basis for $X$ is easily separated from orthogonal basis for $S$. Extension for $\y$ is done using Gram-Schmid (re)orthogonalization. Visually, for a starting batch of $4$ snapshots and after two new snapshots obtained, a few scenarios for $S$ are
$$
S = \left(\begin{smallmatrix} q_1 & q_2 & q_3 & q_4 & q_5 & q_6 \end{smallmatrix}\right)
\left(\begin{smallmatrix} 
+ & + & + & * & * & * \cr 
+ & + & + & * & * & *  \cr 
+ & + & + & * & * & * \cr 
0 & 0 & 0 & \circledast & * & *  \cr
0 & 0 & 0 & 0 & \circledast & *  \cr
0 & 0 & 0 & 0 & 0 & \circledast  \cr
\end{smallmatrix}\right),\;
S = \left(\begin{smallmatrix} q_1 & q_2 & q_3 & q_4\end{smallmatrix}\right)
\left(\begin{smallmatrix} 
+ & + & + & * & * & *\cr 
+ & + & + & * & * & *\cr 
+ & + & + & * & * & * \cr 
0 & 0 & 0 & 0 &  0 & \circledast \end{smallmatrix}\right),\;
S = \left(\begin{smallmatrix} q_1 & q_2 & q_3 \end{smallmatrix}\right)
\left(\begin{smallmatrix} 
+ & + & + & * & * & *\cr 
+ & + & + & * & * & * \cr 
+ & + & + & * & * & *
\end{smallmatrix}\right) .
$$

\subsection{Low rank snapshots -- change of basis}\label{SS=LowRankSnapshotsBasisChange}
The factorization (\ref{eq:S=QtS}) is built ``on the fly'' by processing incoming  snapshots and, unlike the method from \S \ref{SS=Hemati-review}, it requires only one orthonormal basis. It implicitly contains (potentially low-rank) representations analogous to (\ref{eq:RRF-XY}):
\begin{equation}
X = Q \rx,\;\mbox{with}\; \rx = \widetilde S(1:r,1:n),\;\mbox{and}\; Y=Q\ry,\; \ry = \widetilde S(1:r, 2:n+1).
\end{equation}
Even in the full-rank case ($r=n+1$), the dimension reduction may be substantial in large dimensional simulations with $m\gg n$.
Working with $\widetilde S$ instead of $S$ can be viewed as a change of basis. This is best seen in the formulation via the Exact DMD matrix $YX^\dagger=Q \ry \rx^\dagger Q^T$, with the Rayleigh quotient 
$Q^T (YX^\dagger)Q=\ry\rx^\dagger$. 
In general, for any DMD matrix\footnote{Recall \S \ref{SS=DMDReview}.} $A$, 
\begin{equation}\label{eq:AQtSx}
AQ\rx = Q\ry P_{\mathcal{R}(X^T)}= Q\ry \rx^T(\rx^T )^\dagger ,
\end{equation}
where we have to read this relation carefully. Namely, the action of $A$ is defined by the data only on the range of $X$. Since the factorization (\ref{eq:S=QtS}) does not assume full rank of $\rx$, the rank $r$ and the structure of $S$ are determined during the computation. Depending on the last processed snapshot, the last row of $\rx$ can be zero or non-zero.

If the last row of $\rx$ is nonzero, then $\mathrm{range}(Q)=\mathrm{range}(X)$, $\rx$ is of full row rank, $AQ\rx$ is defined by the data and 
$AQ\rx \rx^T = Q\ry \rx^T$,  $AQ = Q\ry \rx^T (\rx \rx^T)^{-1}$. The Rayleigh quotient reads
\begin{equation}\label{eq:QTAQ-nonzero-row}
    Q^TAQ = \ry \rx^\dagger = (\ry \rx^T) (\rx \rx^T)^{-1} ,
\end{equation}
and each eigenpair $(\lambda,w)$ of $\ry \rx^\dagger$ yields the corresponding Ritz pair $(\lambda,Qw)$.

If the last row of $\rx$ is zero, then $\mathrm{range}(Q) \supsetneq \mathrm{range}(X)=\mathrm{range}(Q(:,1:r-1))$. In that case, $\rx = \begin{pmatrix} \widehat S_x \cr \mathbf{0}\end{pmatrix}$, where $\widehat S_x$ is of full row rank. Further, 
$AQ\rx = A Q(:,1:r-1)\widehat S_x$, and 
$$
\rx^\dagger = \begin{pmatrix} \widehat S_x^\dagger & \mathbf{0}\end{pmatrix},\;\;
(\rx\rx^T)^\dagger = \begin{pmatrix} (\widehat S_x\widehat S_x^T)^{-1} & \mathbf{0}\cr \mathbf{0} & 0\end{pmatrix} .
$$
In this case, from (\ref{eq:AQtSx}) we have (instead of (\ref{eq:QTAQ-nonzero-row}))
\begin{eqnarray*}
AQ\begin{pmatrix} \widehat S_x\cr \mathbf{0}\end{pmatrix}\begin{pmatrix} \widehat S_x^T &  \mathbf{0}\end{pmatrix}
&=& Q\ry \begin{pmatrix} \widehat S_x^T &  \mathbf{0}\end{pmatrix} \begin{pmatrix} \widehat S_x^{T\dagger}\cr \mathbf{0}\end{pmatrix} \begin{pmatrix} \widehat S_x^T &  \mathbf{0}\end{pmatrix} = 
Q\ry \begin{pmatrix} \widehat S_x^T &  \mathbf{0}\end{pmatrix} , \\
A Q \begin{pmatrix} I_{r-1} &  \mathbf{0} \cr \mathbf{0} & 0 \end{pmatrix} &=&  Q \ry \begin{pmatrix} \widehat S_x^T (\widehat S_x\widehat S_x^T)^{-1}&  \mathbf{0}\end{pmatrix} .
\end{eqnarray*}
Only  $A Q(:,1:r-1)$ is given by the data, $A Q(:,1:r-1)=Q\ry \widehat S_x^T (\widehat S_x\widehat S_x^T)^{-1}$, and it is natural to take the test space as the range of $Q(:,1:r-1)$ and the Rayleigh 
quotient 
\begin{equation}\label{eq:QxTAQx-Sx}
Q(:,1:r-1)^T A Q(:,1:r-1)  =Q(:,1:r-1)^T Q\ry \widehat S_x^T (\widehat S_x\widehat S_x^T)^{-1} =\ry(1:r-1,:)\widehat S_x^T (\widehat S_x\widehat S_x^T)^{-1} .
\end{equation}
Each eigenpair $(\lambda,\widehat w)$ of (\ref{eq:QxTAQx-Sx}) yields the corresponding Ritz pair $(\lambda,Qw)$ with $w=\left(\begin{smallmatrix} \widehat w\cr \mathbf{0}\end{smallmatrix}\right)$.

Further on,  $Q_x$ will denote the orthonormal basis of $X$. It is equal to $Q_x = Q(:, 1:r-1)$ if the last row of $\rx$ is zero, or $Q_x = Q$ if the last row is non-zero. Notation $\overline S_x$ will be used to denote the non-zero rows of $\rx$ (i.e. $\overline S_x = \rx$ if the last row is non-zero, and $\overline S_x = \rx(1:r-1,:) $ otherwise). The notation $G_{y,x}$ and $G_x$ from (\ref{eq-cross-products}) will be "reused" to denote the sum of cross products of snapshots in lower dimension, i.e.  $G_{y,x} = \ry \overline S_x^T$ and $G_x = \overline S_x \overline S_x^T$. The DMD algorithm for one-basis version of streaming DMD is given below. 

\begin{algorithm}[H]
\caption{\label{ALG:One_basis-DMD} 
$[Z, \Lambda]$ = \texttt{DMD\_one\_basis}($Q, G_{y,x}, G_x$)}
    \begin{algorithmic}[1]
    \REQUIRE orthonormal $Q$ from  \eqref{eq:S=QtS}; cross product matrices $G_{y,x}, G_x$
    \STATE  $r_x = \texttt{size}(G_x, 1)$; 
    $B = G_{y,x}(1:r_x, :) G_x^{-1}$;
    \STATE $[W, \Lambda] = \texttt{eig}(B)$;
    \STATE $Z = Q(:, 1:r_x)W$.
     \ENSURE $Z$, $\Lambda$.
    \end{algorithmic}
\end{algorithm}

\begin{remark}
{\em
If  $(\rx,\ry)$, where the last row of $\rx$ is zero, is forwarded to Algorithm \ref{ALG:DMD:RRR}, then the matrix of the dominant left singular vectors (corresponding to the nonzero singular values) of $\rx$ has the last row zero and the computed eigenvectors are of the form $w=\left(\begin{smallmatrix} \widehat w\cr \mathbf{0}\end{smallmatrix}\right)$. The formulations (\ref{eq:QTAQ-nonzero-row}), (\ref{eq:QxTAQx-Sx}) are in the spirit of \S \ref{SS=Hemati-review}, tailored for large number of high-dimensional low rank data.
}
\end{remark}

\subsubsection{TQ decomposition}\label{TQ-one_basis}

It is again possible to introduce TQ decomposition of $\overline S_x$ to avoid the squaring of the condition number. With $r_x$ denoting the number of columns in $Q_x$, in both \eqref{eq:QTAQ-nonzero-row} and \eqref{eq:QxTAQx-Sx}, it holds that the Rayleigh-Ritz matrix is equal to $Q_x^T A Q_x 
= \widetilde S_y(1:r_x,:) \overline{S}_x^\dagger$. If $\overline{S}_x = T\widetilde Q^T$ is the TQ decomposition  of $\overline S_x$, and $G_{y,q} = \ry \widetilde Q$, then $Q_x^TAQ_x = G_{y,q}(1:r_x,:) T^{-1}$. Algorithm \ref{ALG:One_basis-DMD} can be called as $[Z, \Lambda] = \texttt{DMD\_one\_basis}(Q, G_{y,q}, T)$. 

\subsection{The Exact DMD vectors and the residuals}\label{SS=OneBasisExactDMDResiduals}
The scheme outlined in \S \ref{SS=LowRankSnapshotsBasisChange} addresses the problem of keeping two separate bases, as discussed in 
\S \ref{SS=DiscussFurtherImprove}. Another important issue raised in \S \ref{SS=DiscussFurtherImprove} are computation of the residuals, the Exact DMD vectors and the KMD. The following assertions are analogous to Proposition \ref{prop:exact-sDMD}.
\begin{proposition}
Let $(\lambda_i, w_i)$, be an eigenpair of $Q_x^T AQ_x= G_{y,x}(1\!\! :\!\! r_x,1\!\! :\!\! r_x)G_x(1\!\! :\!\! r_x, 1\!\! :\!\! r_x)^{-1} = \overline G_{y,x}\overline G_x^{-1}$  where $r_x$ is the rank of $X$. Then the corresponding Exact DMD eigenvector is 
$
\z_i^{(ex)} = Q\ry\rx^\dagger w_i = Q  \overline G_{y,x}\overline G_x^{\dagger}w_i,
$
and $A \z_i^{(ex)} = \lambda_i \z_i^{(ex)}$, where $A=YX^\dagger$. Hence, if $\z_i^{(ex)}\neq \mathbf{0}$ then it is an eigenvector of $A$. 
\end{proposition}
\noindent 
Consider now the residuals. The residual $r_i$ of Ritz pair $(\lambda_i, Q_xw_i)$ is equal to 
\begin{equation}\label{rez_ex_Q_not-expanded}
    r_i = AQ_xw_i - \lambda_iQ_x w_i = Q\ry \overline Sx^{\dagger}w_i - \lambda_iQ_xw_i = z_i^{(ex)} - \lambda_i Q_x w_i .
\end{equation}
When a new snapshot does not expand the orthogonal basis for $S$ ($Q = Q_x$), the subspace spanned by $X$ is invariant for the current DMD matrix $A$ since $Q_x^TAQ_x = \ry \rx^\dagger$ and
\[
r_i = AQ_xw_i - \lambda_i Q_x w_i = Q_x \ry \rx^\dagger w_i - Q_x(\lambda_i w_i) = Q_x(Q_x^TAQ_x w_i - \lambda_i w_i) =  \mathbf{0}.
\]
If $Q = (Q_x \ \ q)$ then, in a similar way,
\begin{equation}\label{rez_ex_Q_expanded}
    r_i = Q \ry \widehat S_x^\dagger w_i - \lambda_i Q \left(\begin{smallmatrix}
        w_i \\ 0
    \end{smallmatrix} \right) = Q\begin{pmatrix}
    \ry(1:r-1,:)\widehat S_x^\dagger w_i -\lambda_i w_i
    \\ \ry(r, :) \widehat S_x^\dagger w_i
\end{pmatrix}  = Q\begin{pmatrix}
    \mathbf{0} \\ \widetilde y_{r} \widehat S_x^\dagger w_i 
\end{pmatrix} ,
\end{equation}
where $\widetilde y_{r}$ denotes the last row of $\ry$. 
We have thus proven the following proposition
\begin{proposition}\label{PROP-resids-formulas}
    Let $(\lambda_i, Q_xw_i)$ be the Ritz pair of $A$ via orthonormal basis $Q_x$. Then the norms of the residuals can be computed using representations in the basis $Q$ as
    \begin{equation}\label{eq:residuals_one_basis}
    \norm{r_i}_2 = \begin{cases}
         \|\ry \rx^\dagger w_i - \lambda_i w_i\|_2 = 0, \quad \text{if } Q=Q_x \\
         \|\ry \widehat S_x^\dagger w_i - \lambda_i\left(\begin{smallmatrix}
             w_i \\ \mathbf{0}
         \end{smallmatrix}\right)\|_2 =  \|\widetilde y_r \widehat S_x^\dagger w_i \|_2 , \quad \text{if }\; Q=(Q_x \ q) .
    \end{cases}
    \end{equation}
\end{proposition}
\begin{proof}
    The proof is obtained by taking the norm in  equations \eqref{rez_ex_Q_not-expanded} and \eqref{rez_ex_Q_expanded} and using the fact that 2-norm is invariant to orthogonal transformations. 
\end{proof}

\begin{example}
Figure \ref{fig:vorticity_residuals} illustrates the formulas from Proposition \ref{PROP-resids-formulas} using the vorticity data (Example \ref{EX:Cylinder-kappaG}) with streaming snapshots, starting with $n_0 = 20$. 

\begin{SCfigure}[2][!htb]
        \includegraphics[width=0.50\linewidth, height=0.4\linewidth
        ]{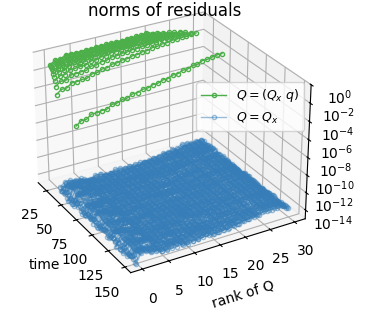}
    \caption{\\
    If a new snapshot contributes with a new direction, the basis is expanded and the second formula in (\ref{eq:residuals_one_basis}) applies -- the residuals are not necessarily small. Otherwise, an invariant subspace is reached and all Ritz vectors are exact eigenvectors of the current DMD matrix -- the first formula in (\ref{eq:residuals_one_basis}) applies and the computed (floating-point) residuals are at the level of the roundoff.}
    \label{fig:vorticity_residuals}
\end{SCfigure}
\end{example}

\subsection{Updating}\label{SS:hemati_one_basis_updating}

 Assume new snapshot $s_{new}$ becomes available and we have the $S = Q\widetilde S$ decomposition for all snapshots obtained beforehand. First we check if $s_{new}$ is already in the subspace spanned by $Q$. If not, we expand $Q$ to $Q_{new} = \begin{pmatrix}
     Q & q
 \end{pmatrix}$ using Gram-shmidt procedure (possibly with reorthogonalization). 
 If it is already in the subspace spanned by $Q$, then $Q_{new} = Q$. In both cases $Q_{x,new}$ is the previous $Q$ and the representation of $s_{new}$ in basis spanned by columns of $Q_{new}$ is $\widetilde s_{new} = Q_{new}^Ts_{new}$.

 The updates are similar as to what was proposed in section \ref{S=Modifications} therefore only the algorithms follow. 
 Notice that in one basis scenario with low rank constraints it is not necessary to update $G_y$ when working with $G_x = XX^T$ nor perform TQ decomposition of $\ry$ when working with TQ decompositions. That is due to the fact that as $Q$ reaches critical number of columns $\widehat r$ in the beginning of next iteration, after (only) $G_x$ (or $T$, depending on the method) is updated, it contains all the information on snapshots spanning orthonormal basis $Q$ and it can easily be updated.
  \vspace{-2mm}
 \begin{algorithm}[ht]
\caption{\label{ALG:One_basis-add} 
$[Q, G_{y,x}, G_x, \widetilde x]$ = 
\texttt{Addx\_one\_basis}($Q, G_{y,x}, G_x, \x, \widetilde x_{-1}, \tol{}_1, \tol{}_2, \tol{}_3, \widehat r$ )}
    \begin{algorithmic}[1]
    \REQUIRE orthonormal $Q$ from \eqref{eq:S=QtS}; cross products $G_{y,x}, G_x$ from section \ref{SS:hemati_one_basis_updating}; new snapshot  $\x$; last low-rank snapshot $\widetilde x_{-1}$ from; tolerances $\tol{}_1, \tol{}_2, \tol{}_3$; maximal dimension $\widehat r$
    \ENSURE  Updated matrices $Q, G_{y,x}, G_x$, low-rank representation of new snapshot $\widetilde x$ 
    \IF{$\texttt{size}(Q, 2) > \texttt{size}(G_x,1)$} 
    \STATE $G_{y,x} = \left(\begin{smallmatrix}
            G_{y,x} & \mathbf{0}
        \end{smallmatrix}\right)$; $G_x = \left(\begin{smallmatrix}
        G_x  & \mathbf{0} \cr \mathbf{0} & 0
    \end{smallmatrix}\right)$  
    \ENDIF
    \STATE $G_x = G_x + \widetilde x_{-1}\widetilde x_{-1}^T$ 
    \IF{$\texttt{size}(Q, 2) = \widehat r$}  
    \STATE $[\Omega, \Theta] = \texttt{eig}(G_x)$; $\rho = \max \{k : \Theta_{kk} > \tol_3\Theta_{11} \}$;
    \STATE $Q = Q\Omega(:, 1:\rho)$; $G_x = \Theta(1:\rho, 1:\rho)$ ;
    \STATE $\widetilde x_{-1} = \Omega(:, 1:\rho)^T\widetilde x_{-1}$
    \ENDIF
    \STATE $[\q, \widetilde x, \gamma]$ = \texttt{GS\_update}($Q, \x, \tol{}_1, \tol{}_2$);
    \IF{$\gamma > 0$}
    \STATE $Q = \left(\begin{smallmatrix}
            Q & \q
        \end{smallmatrix}\right)$;
    $G_{y,x} = \left(\begin{smallmatrix}
            G_{y,x} \cr \mathbf{0}
        \end{smallmatrix}\right)$; $\widetilde x = \left( \begin{smallmatrix}
            \widetilde x \\ \gamma
        \end{smallmatrix}\right)$ 
    \ENDIF
    \STATE $G_{y,x} = G_{y,x} + \widetilde x \widetilde x_{-1}^T$.
    \end{algorithmic}
\end{algorithm}
 \vspace{-3mm}
 \begin{algorithm}[ht]
\caption{\label{ALG:One_basis-add-TQ} 
$[Q, G_{y,q}, T, \widetilde x]$ = 
\texttt{AddxTQ\_one\_basis}($Q, G_{y,q}, T, \x, \widetilde x_{-1}, \tol{}_1, \tol{}_2, \tol{}_3$, tri, $\widehat r$)}
    \begin{algorithmic}[1]
    \REQUIRE orthonormal $Q$ from \eqref{eq:S=QtS}; cross products $G_{y,q}, T$ as described in section \ref{SS:hemati_one_basis_updating}; new snapshot  $\x$; last low-rank snapshot $\widetilde x_{-1}$ from; tolerances $\tol{}_1, \tol{}_2, \tol{}_3$; character tri indicating if $T$ is upper (\texttt{'U'}) or lower (\texttt{'L'}) triangular; maximal dimension $\widehat r$; 
    \ENSURE  Updated matrices $Q, G_{y,q}, T$, low-rank representation of new snapshot $\widetilde x$ 
    \STATE $r = \texttt{size}(Q, 2)$; $r_x = \texttt{size}(T, 1)$
    \IF{$r > r_x$}
    \STATE $T = \left( \begin{smallmatrix}
        T \\ \mathbf{0}
    \end{smallmatrix}\right)$
    \ENDIF
    \STATE $T = \left( \begin{smallmatrix}
        T & \widetilde x_{-1}
    \end{smallmatrix}\right)$    
    \IF{$r= \widehat r$ and $r>r_x$}  
    \STATE $[U, \Sigma, V] = \texttt{svd}(T)$; $\rho = \max\{ k: \Sigma_{kk} > \tol_3 \Sigma_{11}\}$; 
    \STATE $Q = QU(:, 1:\rho)$; $G_{y,q} = U(:, 1:\rho)^T G_{y,q }V(:, 1:\rho)$; $T = \Sigma(1:\rho, 1:\rho)$;
    \STATE $\widetilde x_{-1} = T(:, end)$
    \ENDIF
    \STATE $[\q, \widetilde x, \gamma] = $\texttt{GS\_update}$(Q, x, \tol{}_1, \tol{}_2)$
    \IF{$\gamma > 0$}
    \STATE $Q = \left( \begin{smallmatrix}
        Q & \q
    \end{smallmatrix}\right)$; $\widetilde x = \left( \begin{smallmatrix}
        \widetilde x \\ \gamma
    \end{smallmatrix}\right)$
    \ENDIF
    \STATE $[G_{y,q}, T] = \texttt{updateYQT}(G_{y,q}, T(:, 1:end-1), \mathrm{tri},  \widetilde x_{-1}, \widetilde x, 0, \gamma)$. 
    \end{algorithmic} 
\end{algorithm}

\subsection{KMD and forecasting}\label{S=SingleBasisForecasting}

Use of one basis instead of two circumvents the usage of exact DMD vectors in reconstruction and subsequently forecasting. The problem from equation \eqref{eq:ynDMDz} in the case of $Q = Q_x$ is
\begin{equation}\label{eq:forecasting_one_basis_not-expanded}
    \|Q_x \widetilde y_n - \sum_{i=1}^\ell \alpha_i^{(n)}Q_x w_i\|_2 = \|\widetilde y_n - \sum_{i=1}^\ell \alpha_i^{(n)} w_i\|_2,
\end{equation}
or, in the case of $Q = (Q_x \ q)$,
\begin{equation}\label{eq:forecasting_one_basis_expanded}
    \|\begin{pmatrix}
        Q_x & q 
    \end{pmatrix} \, \widetilde y_n - \sum_{i=1}^\ell \alpha_i^{(n)}\begin{pmatrix}
        Q_x & q
    \end{pmatrix} \begin{pmatrix}
        w_i \\ 0
    \end{pmatrix}\|_2 = \|\widetilde y_n - \sum_{i=1}^\ell \alpha_i^{(n)} \begin{pmatrix}
        w_i \\ 0
    \end{pmatrix}\|_2 .
\end{equation}
Using $\widetilde y_n^T = \left(g_n^T \  \gamma_n \right)^T$ from \texttt{GS\_update} (algorithm \ref{ALG:GS-Reorthog}), the minimization problem from the expanded basis case \eqref{eq:forecasting_one_basis_expanded} is equal to minimizing $\|g_n - \sum_{i}^\ell \alpha_i^{(n)}w_i\|_2$. In both cases, KMD is done in the subspace of dimension \texttt{size}$(Q_x, 1)$ instead of $m$. 

\subsubsection{Example: Lifted hydrogen jet flame (BLAST dataset)}

Lifted hydrogen jet flame data from BLASTnet database was used with Reynolds numbers $5000$ and $7000$ (\cite{chung_wai_tong_2023_8034232}, \cite{chung2023turbulence}). Only velocities $u_x$, $u_y$ on a grid $1100 \times 600$ were used, resulting in snapshots of length $m=1320000$. Starting from $n_0 = 20$ snapshots, reconstruction with rounding to single precision as described in section \ref{SSS:example-vorticity} was done, this time in one orthogonal basis. For $Re = 5000$, the norm of errors for reconstruction $k$ steps ahead ($k = 1,2,3,4,5$) were the same (around $10^{-2}$).  In the case $Re = 5000$ we used formulas \eqref{eq:forecasting_one_basis_not-expanded} and \eqref{eq:forecasting_one_basis_expanded} for reconstruction to compare the time saving when extracting orthonormal basis. The results are shown in figure  \ref{fig:blast-error-single-onebas}.  In each case we used all previous snapshots for reconstruction. After $140$ snapshots were added ($n=n_0+140=160$), time it took to predict $5$ steps ahead was over $50$ seconds in full space and only around $5$ seconds in the orthonormal basis, independent of whether used DMD modes and eigenvalues came from TQ or HWR approach.
When using data with $Re = 7000$, the first attempt to round $G_x$ to single precision triggered overflows, meaning that any method based on $G_x$ must fail in single precision. On the other hand, TQ was successfully updated with all snapshots and had relative error norms around $10^{-1}$ for prediction one step ahead.
\begin{SCfigure}[2][!htb]
        \includegraphics[width=0.5\textwidth]{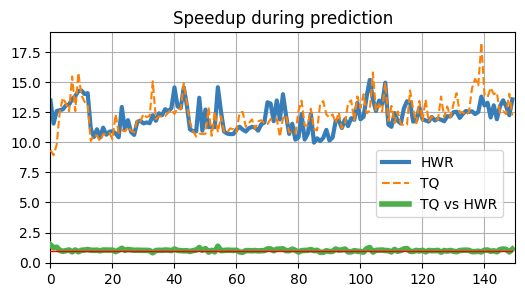}
    \caption{\\
    Speedup when running predictions 5 steps ahead in basis spanned by $Q$ or in full space $\mathbb R^m$ for both HQR-sDMD and TQ-sDMD. Extracting orthonormal basis and performing KMD in a smaller subspace significantly reduces run time resulting in algorithms on average 11 times faster. There is no notable run time difference when performing prediction in smaller subspace with DMD modes and eigenvalues obtained by TQ-sDMD vs. HWR-sDMD.}\label{fig:blast-error-single-onebas}
\end{SCfigure}
\subsubsection{Example: Gray-Scott model}\label{SSS:Gray-Scott}

For another example in one basis we generate snapshots using Gray-Scott model of reaction-diffusion system with parameters $D_u = 1, D_v=0.5, F=0.062, k = 0.061$ initialized with three randomly placed squares on a $350 \times 350$ grid ($m=122500$). Starting with $n_0 = 30$, we add up to $300$-th snapshot. After each new snapshot we make a prediction up to 5 steps ahead using only the last $30$ snapshots in reconstruction. The predicted snapshots after observing $n=200$ snapshots for HWR-sDMD in double precision and TQ-sDMD in simulated single is shown in figure \ref{fig:gray-scott-visualization-5}. Norms of all prediction errors are calculated and shown in figure \ref{gray-scott}, including for HWR-sDMD in single precision. Note that using HWR-sDMD in single precision produced \texttt{NaN}s in $144$ predicted snapshots when predicting one step ahead and in $191$ predicted snapshots when predicting five steps ahead.

\begin{figure}[h]
    \centering
    \includegraphics[width=1\linewidth]{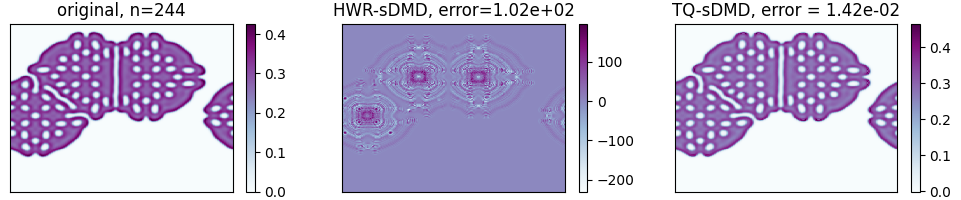}
    \caption{Visualization of predicted dynamics of the Gray-Scott model, using HWR-sDMD (64) and TQ-sDMD (32) five steps ahead after $n=200$ observed snapshots.}
    \label{fig:gray-scott-visualization-5}
\end{figure}
\begin{figure}[h]
    \centering
    \begin{subfigure}[t]{0.49\textwidth}
        \centering
        \includegraphics[width=1\linewidth]{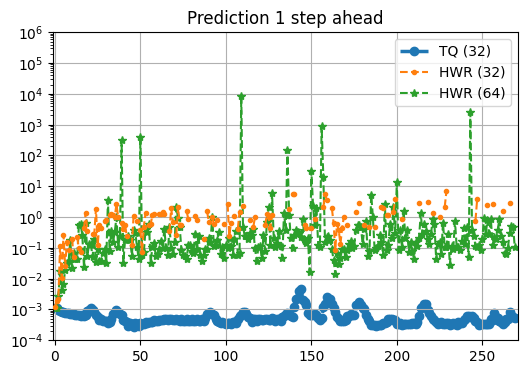}
    \end{subfigure}
    \hfill
    \begin{subfigure}[t]{0.49\textwidth}
        \centering
        \includegraphics[width=1\linewidth]{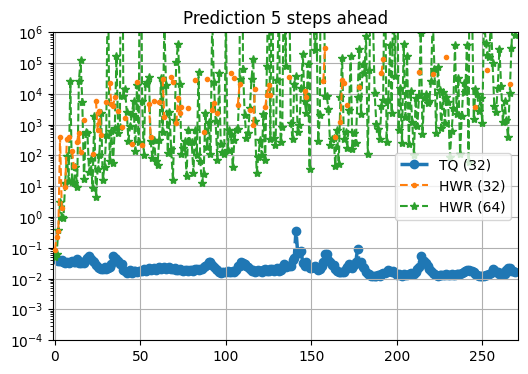}
    \end{subfigure}
    \caption{\textit{Left:} When predicting one step ahead, TQ-sDMD in simulated single basis maintains relative error norms around $10^{-3}$. Meanwhile, the errors for HWR-sDMD in both simulated single and fully double precision result in errors between $0.1$ and $1$ with occasional oscillation to larger errors. In HWR-sDMD (32) $144$ error norms were \texttt{NaN}. Error norms for TQ-sDMD in double precision are omitted since they coincide with errors when working in single precision. \textit{Right:} When predicting 5 steps ahead, all errors are larger, but TQ-sDMD errors stay almost always under $0.1$, while both HWR-sDMD errors stay above $10$ after first few iterations.  HWR-sDMD (32) error norms contains $191$ \texttt{NaN}s.}\label{gray-scott}
\end{figure}

%% file: sections/zhang_updates.tex
\graphicspath{{\subfix{../}}}

\newcommand{\G}{\ensuremath{\widehat G}}
\newcommand{\dl}{\ensuremath{d_L}}

\section{Online DMD for long low dimensional trajectory ($m \ll n$)}\label{S=Zhang-New}

Assume the data is such that $m<n$ and that $X$ is of full (row) rank. Then $A = YX^\dagger$ is the unique solution of least squares equation for DMD. Since $X$ is full rank, $XX^\top$ is invertible and  $X^\dagger = X^\top\left(XX^\top \right)^{-1}$. Inserting that into $A = YX^\dagger$ gives 
$A = \left(Y X^\top\right)\left(XX^\top \right)^{-1} = \G_{y,x} \G_x^{-1}$.
\subsection{Sherman-Morrison updates of $\G_x^{-1}$}
Due to the assumption of full row rank, the updates of $\G_x$, $\G_{y,x}$ are always as in \S \ref{SS=AddNewData}, case (2) (relation (\ref{eq:XmewXnewT-2})). Instead of updating $\G_{x}$ and then inverting it when needed, Zhang et al. \cite{zhang_DMD} used Sherman-Morrison formula for direct updates of $\G_{x}^{-1}$ after rank-one additive change of $\G_x$:
\begin{equation}\label{eq:Sherman-Morrison-G}
    \G_{x,new}^{-1} = ({\G_x} + x_{new}x_{new}^\top)^{-1} = \G_x^{-1} - \frac{\G_x^{-1} x_{new} x_{new}^\top \G_x^{-1}}{1+x_{new}^\top \G_x^{-1} x_{new}}.
\end{equation}
Using this in $A_{new} = Y_{new}X_{new}^\dagger$ with $X_{new}=(X\; x_{new})$, $Y_{new}=(Y\; y_{new})$ and $P_x= \G_x^{-1}$, \cite{zhang_DMD} derived an explicit update formula for the DMD matrix:
\begin{equation}\label{eq:A-new-Zhang}
A_{new} = A + \gamma (y_{new} - A x_{new})x_{new}^T P_x,\;\;\gamma = 1/(1+x_{new}^T P_x x_{new}).
\end{equation}
For the next step, $P_x$ is updated, using (\ref{eq:Sherman-Morrison-G}), as 
\begin{equation}\label{eq:Pxnew}
P_{x,new}=P_x - \gamma z z^T, \;\;z = P_x x_{new}
\end{equation}
Note that this updating scheme keeps updating the inverse $P_x=\G_x^{-1}$ of the positive definite (Gram) matrix $\G_x=XX^T$ whose condition number is the squared condition number of the raw data matrix $X$.  Unlike the update by augmentation as in \S \ref{SS:UseUpdatedCholFactor}, here we have persistent rank-one (positive semi-definite) additive correction, meaning $G_{x,new}\succeq G_x$ (in the sense of the Loewner partial ordering: $G_{x,new}- G_x$ is positive semidefinite). This implies that $P_x \succeq P_{x,new}$, which is clearly seen in (\ref{eq:Sherman-Morrison-G}), and this means that, in the process of
repeated updating $P_x\leftarrow P_{x,new}$ in floating-point arithmetic, the matrices $P_x$ can 
lose numerical definiteness. In fact, this mishap can occur in the very first step when $P_x$ is initialized as explicitly computed inverse of $XX^T$. 

\subsubsection{Extension to HWR-sDMD ($m>n$)}\label{S6=SMforNgtM}

Assume the original setting from HWR-sDMD where $X = Q_x \widetilde X, \; Y = Q_y\widetilde Y$. Instead of updating $G_x$ and inverting it, it is possible to directly update $G_x^{-1}$. In the scenario where the new snapshot does not expand $Q$ the updated matrix $G_{x,new} 
= G_x + \widetilde x \widetilde x^T$ and Sherman-Morrison formula \eqref{eq:Sherman-Morrison-G} can be used.
If the new snapshot expands the matrix $Q$, 
then $
    G_{x,new} = \left(\begin{smallmatrix}
        G_x & \mathbf{0} \\ \mathbf{0} & 0
    \end{smallmatrix} \right) + \widetilde x\widetilde x^T$.
Updates are less direct in this case since matrix $\left(\begin{smallmatrix}
        G_x & \mathbf{0} \\ \mathbf{0} & 0
    \end{smallmatrix} \right)$ is singular and does not have an inverse. It does however have a pseudoinverse $\left( \begin{smallmatrix}
         G_x^{-1} & \mathbf{0} \\ \mathbf{0} & 0
    \end{smallmatrix}\right)$.
Back in 1973, 
Meyer \cite{meyer_update_pinv} proved the following theorem.
\begin{theorem}\label{THM:pinv-updates}
    Let $M$ be $l_1 \times l_2$ matrix, and $c, d$ vectors of length $l_1$ and $l_2$ respectively. If $c \notin \mathcal{R}(M)$ and $d \notin \mathcal{R}(M^T)$, then
      $(M + cd^T)^\dagger = M^\dagger - M^\dagger c u^\dagger - v^\dagger d^TM^\dagger + (1+d^TM^\dagger c)v^\dagger u^\dagger$, 
    with $u = (I - MM^\dagger)c$, $v = d^T(I-M^\dagger M)$
\end{theorem}
\noindent If we denote $\widehat G_x = \left(\begin{smallmatrix}
        G_x & \mathbf{0} \\ \mathbf{0} & 0
    \end{smallmatrix} \right)$, then $M = \widehat G_x$, $c = d = \widetilde x = \begin{pmatrix}
        g_x^T & \gamma_x
    \end{pmatrix}^T$ and 
\begin{equation}
    u = (I - \widehat G_x \widehat G_x^\dagger) \widetilde x = \begin{pmatrix}
        \mathbf{0}_{nxn} & \mathbf{0} \\ \mathbf{0} & 1
    \end{pmatrix} \begin{pmatrix}
        g_x \\ \gamma_x
    \end{pmatrix} = \begin{pmatrix}
        \mathbf{0} \\ \gamma_x
    \end{pmatrix}; \; \; v = \widetilde x^T (I- \widehat G_x^\dagger \widehat G_x) = u^T = \begin{pmatrix}
        \mathbf{0} & \gamma_x
    \end{pmatrix}.
\end{equation}
Using $u^\dagger = v/\gamma_x^2$, $v^\dagger = u/\gamma
_x^2$, it is easy to show
\begin{proposition} Assume $G_x^{-1}$ as above is available and let $X$ denote data matrix of snapshots received so far. If a new snapshot pair $(x_{new}, y_{new})$ becomes available where $x_{new} \notin \mathcal{R}(X)$, then $G_{x,new}^{-1}$ can be computed as
    \begin{equation}\label{Meyer}
        G_{x,new}^{-1} = \begin{pmatrix}
        G_x^{-1} & -\frac{G_x^{-1}g_x}{\gamma_x} \\
        -\frac{g_x^TG_x^{-1}}{\gamma_x} & \frac{1+g_x^TG_x^{-1}g_x}{\gamma_x^2}
    \end{pmatrix}.
    \end{equation}
\end{proposition}
\begin{proof}
    Since the last row of $\widehat G_x$ are zeros and $\gamma_x \neq 0$, there cannot exist a vector $v$ such that $\widehat G_x v = \widetilde x$. Therefore, theorem \ref{THM:pinv-updates} can be applied.
\end{proof}

\subsection{Updating Cholesky factor of $\widehat G_x$}\label{SS=NewZhang-CholeskyUpdate}
It is known that the Sherman-Morrison formula can exhibit erratic behavior \cite{ma2025notestabilityshermanmorrisonwoodburyformula}.
In the case of the online DMD, this means that computed $P_x$ can lose the positive definiteness, which makes the computed results uninterpretable, in particular not in the backward error sense. 

We resolve this problem by recalling the well known technique of using positive (semi)definite matrices implicitly. In this concrete case, it means updating the Cholesky factor of $G_x$ instead of updating its inverse. 
If $\widehat G_x = R_x^T R_x$ is the Cholesky factorization, then the updating formula (\ref{eq:A-new-Zhang}) can be used, but $z=P_x x_{new}$ is computed as $R_x^{-1}(R_x^{-T} x_{new})$ (by solving two triangular systems). The \emph{flop} count of solving two triangular systems is $2m^2$, which is nearly the same as for one matrix-vector multiplications ($2m^2-m$). Having $z$, the cost of computing $A_{new}$ in (\ref{eq:A-new-Zhang}) is $4m^2+3m+1$ \emph{flops}.

At the beginning, the upper triangular factor $R_x$ is computed by the QR factorization of the tall $n\times m$ matrix $X^T$. The cost of this is $2nm^2-2m^3/3$. On the other hand, the cost of computing $G_x=XX^T$, using symmetry, is $nm^2+nm$, and (\ref{eq:A-new-Zhang}) requires initial computation of $P_x=\G_x^{-1}$ which amounts $m^3$ flops (computing the Cholesky factor, inverting it and multiplying triangular factors, and using symmetry). Initially, $m$ is close to $n$ -- the snapshots are collected until $X$ becomes full row rank. Hence, if the first step is done, e.g., with $n=m$ then the overhead is smaller. 

The key operation in the online phase is updating the Cholesky factor of $\G_x$. 
Since 
$$
\G_{x,new} = (X\; x_{new}) (X\; x_{new})^T = \G_x + x_{new}x_{new}^T = \begin{pmatrix} R_x^T &  x_{new}\end{pmatrix} \begin{pmatrix} R_x \cr x_{new}^T\end{pmatrix} ,
$$
the Cholesky factor $R_{x,new}$ of $\G_{x,new}$ is computed implicitly by the QR factorization
\begin{equation}\label{eq:Rnew-scheme}
G_m\cdots G_2 G_1 \begin{pmatrix} R_x \cr x_{new}^T\end{pmatrix} = G_m\cdots G_2 G_1
\left( \begin{smallmatrix} 
* & * & * & * \cr
0 & * & * & * \cr
0 & 0 & * & * \cr
0 & 0 & 0 & * \cr
+ & + & + & +
\end{smallmatrix}\right) = \left( \begin{smallmatrix} 
* & * & * & * \cr
0 & * & * & * \cr
0 & 0 & * & * \cr
0 & 0 & 0 & * \cr
0 & 0 & 0 & 0
\end{smallmatrix}\right)=\begin{pmatrix} R_{x,new}\cr \mathbf{0} \end{pmatrix},\;\;U_G=G_m\cdots G_2 G_1,
\end{equation}
using the sequence of Givens rotations $G_i$  with pivotal positions $((i,i); (m+1,i))$ for $i=1,...,m$. The cost of such an update, without using fast scaled rotations, is $3m^2+O(m)$ \emph{flops}.
{With fast rotations, this can be reduced to $2.5m^2+O(m)$ \emph{flops}.} For comparison, direct update of $P_x$ requires, by exploiting symmetry\footnote{In LAPACK, the symmetry of low rank updates is used in the subroutine DSYRK.} and using already computed $z=P_x x_{new}$, $m^2+O(m)$ \emph{flops}. To summarize:

\begin{proposition}\label{PROP:Chol-update}
Let $X$ and $Y$ be data matrices and let $X$ have full row rank. Assume that the Cholesky decomposition $XX^T = R_x^TR_x$ is available. After receiving a new pair of snapshots $(x_{new}, y_{new})$, the matrices $R_x$ and $A=YX^\dagger=(YX^T)(XX^T)^{-1}$ can be updated as
\begin{equation}\label{eq:Chol-update}
     A_{new} = A+\gamma(y_{new}-Ax_{new})p_x^T , \quad  p_x = (R_{x}^{-1}(R_{x}^{-T}x_{new})), \quad  \gamma = 1/(1+x_{new}^Tp_x),
\end{equation}
 where $U_G$ is the product of Givens rotations from (\ref{eq:Rnew-scheme}) and $R_{x, new} = U_G(1:m, 1:m+1) \left(\begin{smallmatrix}
R_{x} \\ x_{new}^T
\end{smallmatrix}\right)$ is the update of $R_x$ to be used in next iteration.
\end{proposition}

\begin{remark}\label{REM:cond-est}
{\em
 In addition to maintaining the positive definiteness of $P_x$, which is implicitly assured, we have avoided the squared condition number because $\kappa_2(R_x) = \kappa_2(\G_x)^{1/2} = \kappa_2(P_x)^{1/2}= \kappa_2(X)$. Example \ref{example-chua} illustrates the significance of this.
Since $\sqrt{\kappa_2(P_x)}=\kappa_2(R_x)$ and $R_x$ is upper triangular, we can use fast condition estimators for triangular matrices and always have useful error estimate. 
}
\end{remark}

\subsection{Updating TQ decomposition}\label{SS=NewZhang-LQUpdate}

Another approach that avoids using the Sherman-Morrison formula entirely is to use TQ decomposition of $X$. Assume TQ decomposition of $X$ is available, $X^T = QT^T$, where $T \in \R^{m \times m}$ is a triangular matrix and $Q \in \R^{n \times m}$ orthonormal ($Q^TQ = I$). Then, 
    $A = \left( YQ\right)T^{-1} = \G_{y,q}T^{-1}$.
When a new snapshot pair becomes available $(x_{new}, y_{new})$, 
\begin{equation}
    \begin{pmatrix}
    X & x_{new} 
\end{pmatrix} = \begin{pmatrix}
    T & x_{new}
\end{pmatrix}\begin{pmatrix}
    Q^T & \mathbf{0} \\ \mathbf{0} & 1
\end{pmatrix}, \; \; \begin{pmatrix}
    Y & y_{new}
\end{pmatrix}\begin{pmatrix}
    Q & \mathbf{0} \\ \mathbf{0} & 1
\end{pmatrix} = \begin{pmatrix}
    \widehat G_{y,q} & y_{new}
\end{pmatrix}.
\end{equation}
The expression above can be returned to proper TQ factorization using Givens rotations. The schema for application of rotations if $T$ is upper triangular is
\begin{equation}\label{eq:Tnew-scheme-upper}
 \begin{pmatrix} T & x_{new}\end{pmatrix}G_m^T\cdots G_2^T G_1^T  = 
\left( \begin{smallmatrix} 
* & * & * & * & + \cr
0 & * & * & * & +\cr
0 & 0 & * & * & + \cr
0 & 0 & 0 & * & +\cr
\end{smallmatrix}\right) G_m^T\cdots G_2^T G_1^T = \left( \begin{smallmatrix} 
* & * & * & * &0 \cr
0 & * & * & * &0\cr
0 & 0 & * & * &0\cr
0 & 0 & 0 & * &0\cr
\end{smallmatrix}\right)=\begin{pmatrix} T_{new} & \mathbf{0} \end{pmatrix},
\end{equation} 
and in the case of lower triangular $T$, 
\begin{equation}\label{eq:Tnew-scheme-lower}
 \begin{pmatrix} T & x_{new}\end{pmatrix}G_1^TG_2^T\cdots  G_m^T  = 
\left( \begin{smallmatrix} 
* & 0 & 0 & 0 & + \cr
* & * & 0 & 0 & +\cr
* & * & * & 0 & + \cr
* & * & * & * & +\cr
\end{smallmatrix}\right) G_1^TG_2^T\cdots  G_m^T  = \left( \begin{smallmatrix} 
* & 0 & 0 & 0 & 0 \cr
* & * & 0 & 0 & 0\cr
* & * & * & 0 & 0 \cr
* & * & * & * & 0\cr
\end{smallmatrix}\right)=\begin{pmatrix} T_{new} & \mathbf{0} \end{pmatrix}.
\end{equation} 
Updates to $\widehat G_{y,q}$ and $T$ are then obtained by
\begin{equation}
    T_{new} = \begin{pmatrix}
        T & x_{new}
    \end{pmatrix}U_G^T (:, 1:m), \; \;
    \widehat G_{y,q}^{(new)} = 
      \begin{pmatrix}
         YQ & y_{new}
     \end{pmatrix} U_G^T(:, 1:m) ,
\end{equation}
where the notation $U_G$ for a series of Givens rotations was reused. Therefore $\G_{y,q}$ and $T$ can be updated by calling $[\G_{y,q}, T] =\texttt{updateYQT}(\G_{y,q}, T, x_{new}, y_{new}, 0, 0, \mathrm{tri})$ (Algorithm \ref{ALG:updateYQT}). 
\vspace{-1mm}
\subsubsection{Example: Chua's circuit}\label{example-chua}
\vspace{-1mm}
The discussion from this section can be illustrated using one challenging example from dynamical systems: the Chua's circuit \cite{Chua-1980}, 
\cite{Matsumoto-1984}. The system is numerically simulated on a discrete uniform grid in $[0,t_{\max}]$ with spacing $\Delta t$ and initial condition $\x_0=(x_0,y_0,z_0)$, using 
RK45 solver. The observables are $x, y, z, x^2, y^2, z^2$. The initial blocks $X_0$ and $Y_0$ are $6\times n_0$. 
The updating formulas are tested simply by comparing the updated matrix $A_{new}$ with the explicitly computed reference value $Y_{new}X_{new}^\dagger$. 
The purpose of the example is to illustrate the numerical aspects of the updating formulas and 
the advantage of the methods proposed in this and previous section. All of the calculations were done entirely in double precision.

The results of experiments with the Chua's circuit (using a concrete example from \cite{chuacircuits_matlab}) are shown in Figure \ref{FIG:EX:Chua py}. In the first two, the condition numbers of $X$ behaved well - there was almost no growth. But, as pointed out in  \S \ref{SSS=kappaGrowthTheory} (Proposition \ref{PROP:kappaG-growth}), this cannot be guaranteed. Indeed, the last panel shows a persistent growth. Furthermore, figure \ref{subfig:chua-n0-250}  shows that even if a larger number of initial snapshots $n_0$ is used, the initial error for classical SM method might be smaller in the beginning but still grow. Meanwhile, Cholesky and TQ decompositon approaches display almost the same, relatively small error as before.
\begin{figure}[ht]
    \centering
    \begin{subfigure}[t]{0.32\textwidth}
        \centering
        \includegraphics[width=\textwidth]{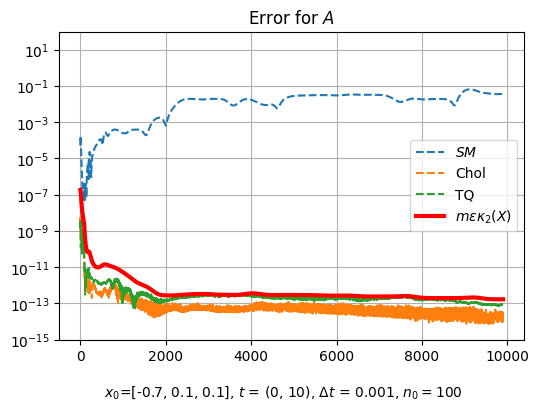}
        \caption{$x_0 = [-0.7, 0.1, 0.1], \Delta t =10^{-3}$}
    \end{subfigure}
    \hfill
    \begin{subfigure}[t]{0.32\textwidth}
        \centering
        \includegraphics[width=\textwidth]{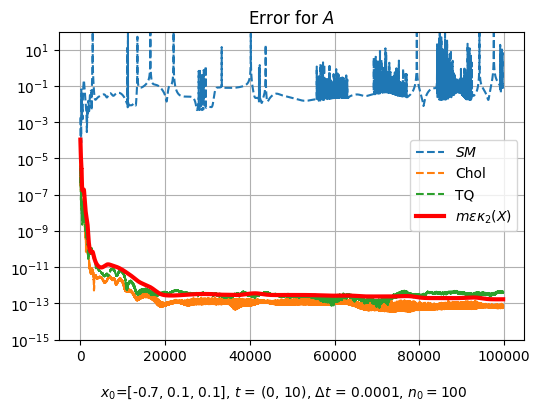}
        \caption{$x_0 = [-0.7, 0.1, 0.1], \Delta t =10^{-4}$}
    \end{subfigure}
    \hfill
    \begin{subfigure}[t]{0.32\textwidth}
        \centering
        \includegraphics[width=\textwidth]{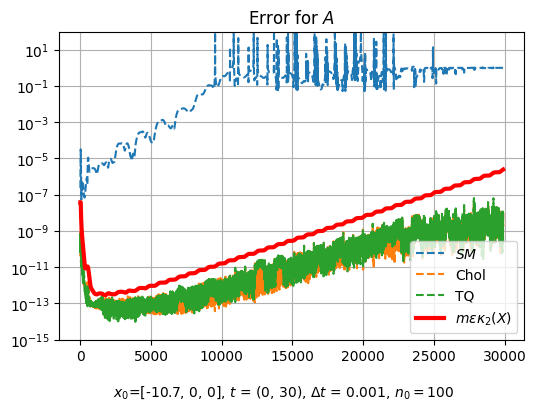}
        \caption{$x_0 = [-10.7, 0, 0], \Delta t =10^{-3}$}
    \end{subfigure}  
    \vspace{-2mm}
    \caption{Relative error for $||A_{new}-Y_{new}X_{new}^\dagger||_2/||Y_{new}X_{new}^\dagger||_2$ obtained using classical SM formula ($A_{SM}$), SM with Cholesky ($A_{Chol}$) and using TQ decomposition ($A_{TQ}$). The relative error is low for $A_{Chol}$ and $A_{TQ}$ decomposition approach and consistently below or near $m\varepsilon  \kappa_2(X)$, while the error for $A_{SM}$ grows noticeably in all three examples. We also observed that matrix $P_x$ lost its positive definiteness either during initialization (middle scenario) or at some point during the updates (left and right scenarios).}\label{FIG:EX:Chua py}
\end{figure}

Clearly, if the data, selected observables and other parameters (e.g., $n_0$) are such that the problem is ill-conditioned the computation will fail. This is illustrated in figures \ref{subfig:chua-k1} and \ref{subfig:chua-k2}. An important advantage of our proposed method is that the $O(m)\epsilon\kappa_2(X)$ bound can be used as a reliable indicator of accuracy. It can be efficiently computed as discussed in Remark \ref{REM:cond-est}. Note how in the last two panels $m\epsilon\kappa_2(X)$ determines the accuracy as long as it is below one. Once it passes over one, no method can guarantee any accuracy in the computed matrix $A$.
\begin{figure}[h]
    \centering
    \begin{subfigure}[t]{0.32\textwidth}
        \centering
        \includegraphics[width=\textwidth]{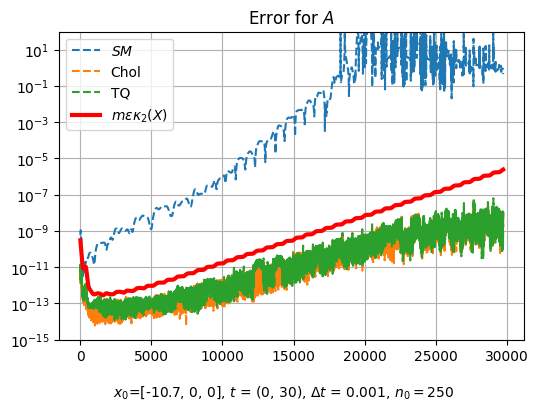}
        \caption{$n_0 = 250$}\label{subfig:chua-n0-250}
    \end{subfigure}
    \hfill
    \begin{subfigure}[t]{0.32\textwidth}
        \centering
        \includegraphics[width=\textwidth]{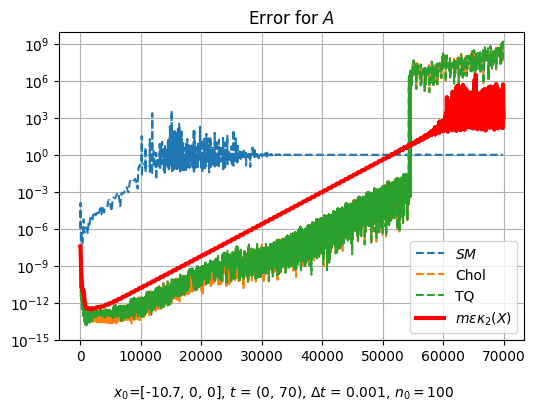}
        \caption{$t_{max}=70, n_0 = 100$}\label{subfig:chua-k1}
    \end{subfigure}
    \hfill
    \begin{subfigure}[t]{0.32\textwidth}
        \centering
        \includegraphics[width=\textwidth]{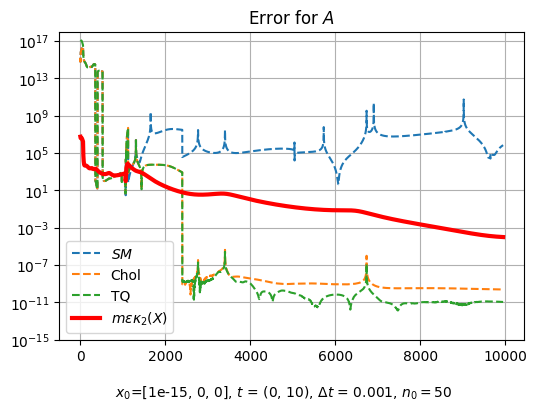}
        \caption{$x_0 = [10^{-15}, 0, 0], n_0 = 50$}\label{subfig:chua-k2}
    \end{subfigure}
    \vspace{-2mm}
    \caption{\textit{Left:} 
    For larger $n_0$ the initial error for $A_{SM}$ is smaller. It still starts growing at some point, while errors for $A_{Chol}$ and $A_{TQ}$ remain relatively small. 
    \textit{Middle:} Errors for $A_{TQ}, A_{Chol}$ follow  $m\varepsilon \kappa_2(X)$ curve while it is smaller than $\approx 1$. 
    \textit{Right:} All errors start high. As $m\varepsilon \kappa_2(X)$ approaches $\approx 1$, errors for $A_{Chol}$ and $A_{TQ}$ dampen, while $A_{SM}$ does not manage to recover. \label{fig:chua-larger_n0}}
\end{figure}
\section{Concluding remarks and outlook}\label{S=Conclude}
We have provided new numerically robust core subroutines for streaming applications of the (E)DMD.
The next steps in our work include out-of-core implementations, and extending these methods to dynamically changed data window widths, where
the oldest data are gradually or abruptly discarded, as dictated by sudden changes in the dynamics and 
guided by the forecast performance. This will be done for several variations of the (E)DMD, and for applications such as  e.g. adaptive data driven model predictive control.

%% file: sections/appendix.tex
\newpage

\section{Appendix}

\subsection{Generalization of Sherman-Morrison updates}

It is possible to extend the Sherman-Morrison updates proposed by Zhang et al. (section \ref{S=Zhang-New}) to 
accommodate changes in the rank of the streaming data matrix $X$. When $X$ isn't full rank, the DMD matrix $A$ is not uniquely determined since the least-squares problem is under-determined. If $X$ is decomposed into orthogonal matrix and low-rank component $X = Q_x \widetilde X$,
then snapshot matrix can be decomposed as $S = Q\widetilde S$, with $Q$ obtained by (potentially) expanding $Q_x$ as in section \ref{S=Hemati_ONE_BASIS}. It is simple to show that the formulas for the Rayleigh-Ritz matrix $Q_x^TAQ_x$ are exactly \eqref{eq:QTAQ-nonzero-row} or \eqref{eq:QxTAQx-Sx} depending on the last row of $\widetilde S$. As new snapshots arrive, the rank of $S_{new}$ might increase, expanding $Q_x$ and $Q$. The method we propose manages to update $G_x$ and $\widetilde A$ in both scenarios (expansion or no expansion) and keeps good numerical properties. This method can also be used in tall data matrix scenario ($m>n$)  to update $G_x^{-1}$ (instead of $G_x$ and then inverting) since the same formulas are used (\eqref{eq:QTAQ-nonzero-row} and \eqref{eq:QxTAQx-Sx}).

\begin{proposition}
    Using the notation from section \ref{S=Hemati_ONE_BASIS}, let $G_x = (\overline S_x \overline S_x^T)$ and $\widetilde A = \ry \overline S_x^{\dagger} = (\ry \overline S_x^T)G_x^{-1}$. Assume new snapshot pair becomes available $(x_{new}, y_{new})$, and let $g_x, \gamma_x$ and $g_y, \gamma_y$ be the results of GS procedure as described in algorithm \ref{ALG:GS-Reorthog}. Using $R_x$ from the QR decomposition of  $ \overline S_x^T  = \widetilde QR_x$, if $\gamma_x >0$
\begin{equation}
    \widetilde A_{new} = \begin{cases}
        \begin{pmatrix}
        \widetilde A & -\frac{1}{\gamma_x}\left(\widetilde Ag_x -  g_y \right)
    \end{pmatrix}, \quad  \gamma_y=0,\\ 
        \begin{pmatrix}
        \widetilde A & -\frac{1}{\gamma_x}\left(\widetilde Ag_x -  g_y \right)\\
        \mathbf 0  & \gamma_y/\gamma_x
    \end{pmatrix}, \quad  \gamma_y>0,
    \end{cases} \quad \quad R_{x, new} = U_G\begin{pmatrix}
        R_{x} & \mathbf{0} \\ g_x^T & \gamma_x
    \end{pmatrix}.
\end{equation}
And with $z=R_x^{-1}(R_x^{-T}g_x)$ and $\alpha = 1+(g_x^TR_x^{-1})(R_x^{-T}g_x)$, if $\gamma_x =0$ 

\begin{equation}
    \widetilde A_{new} =\begin{cases}
       \widetilde A - \alpha\left(\widetilde Ag_x - g_y\right)z^T, \quad \gamma_y = 0, \\
       \begin{pmatrix}
           \widetilde A - \alpha\left(\widetilde Ag_x - g_y)z^T\right) \\
           (\gamma_y/\alpha)z^T
       \end{pmatrix}, \quad \gamma_y > 0,
    \end{cases} \quad \quad R_{x, new} = U_G(1:end-1,:)\begin{pmatrix}
        R_x  \\ g_x^T 
    \end{pmatrix}.
\end{equation}

\end{proposition}
\begin{remark}
   {\em These formulas use the same "trick" to avoid updating (and keeping) the matrix $G_{y,x}$ as Zhang et al. \cite{zhang_DMD} by applying Meyer's or Sherman-Morrison formula directly into $\widetilde A_{new}$ formulated via $\widetilde A$ and new snapshots.}
\end{remark}

\begin{proof}
    Assume $\gamma_x>0$ ($\overline S_x = \widehat S_x$) and $\gamma_y = 0$. Using the Meyer's formula for $G_{x,new}^{-1}$ gives 
\begin{align*}
\widetilde A_{new} &= \begin{pmatrix}
    \ry & g_y
 \end{pmatrix} \begin{pmatrix}
     \widehat S_x^T & \mathbf 0 \\ g_x^T & \gamma_x
 \end{pmatrix}\begin{pmatrix}
     G_x^{-1} & -\frac{G_x^{-1}g_x}{\gamma_x} \\
        -\frac{g_x^TG_x^{-1}}{\gamma_x} & \frac{1+g_x^TG_x^{-1}g_x}{\gamma_x^2}
 \end{pmatrix} 
 = \begin{pmatrix}
     G_{y,x} & \mathbf{0}
 \end{pmatrix}G_{x,new}^{-1} + g_y \begin{pmatrix}
     g_x^T & \gamma_x
 \end{pmatrix}G_{x, new}^{-1} \\
&= \begin{pmatrix}
    G_{y,x}G_x^{-1} & -\frac{G_{y,x}G_x^{-1}}{\gamma_x} 
\end{pmatrix} + g_y\begin{pmatrix}
    \mathbf{0} & 1/ \gamma_x\end{pmatrix} = \begin{pmatrix}
        \widetilde A & -\frac{1}{\gamma_x}\left( \widetilde Ag_x - g_y\right)
    \end{pmatrix}.
\end{align*}
If $\gamma_y > 0$, starting with
$$
\widetilde A_{new} = \begin{pmatrix}
    \ry & g_y \\ \mathbf{0} & \gamma_y
 \end{pmatrix} \begin{pmatrix}
     \widehat S_x^T & \mathbf 0 \\ g_x^T & \gamma_x
 \end{pmatrix}G_{x,new}^{-1}$$ 
 the similar result can be proved. In both cases $\overline S_{x, new}^T = \left( \begin{smallmatrix}
     \widetilde Q & \mathbf 0 \\ \mathbf 0 & 1
 \end{smallmatrix}\right)\left(\begin{smallmatrix}
     R & \mathbf 0 \\ g_x^T & \gamma_x
 \end{smallmatrix} \right) = \widetilde Q_{new} R_{x,new}$, and $R_{x,new}$ is returned to upper-triangular by a series of Givens transformations denoted by $U_G$ (as in \eqref{eq:Rnew-scheme}, except the last row is non-zero, and therefore not discarded). 

 In the scenario where $\gamma_x = 0$, Cholesky updates and Sherman-Morrison formula can be used. In this case $G_{x, new}^{-1}=
     G_x^{-1} - \frac{1}{\alpha}G_x^{-1}\widetilde x_{new} \widetilde x_{new}^T G_x^{-1}$ and, for $\gamma_y = 0$
\begin{align*}
    \widetilde A_{new} &= \begin{pmatrix}
        \ry & g_y
    \end{pmatrix} \begin{pmatrix}
        \rx^T \\ g_x^T
    \end{pmatrix}\left( G_x^{-1} - \frac{1}{\alpha}G_x^{-1} g_x g_x^T G_x^{-1} \right) \\
    &= G_{y,x}G_x^{-1} - \frac{1}{\alpha}G_{y,x}G_x^{-1}g_x g_x^TG_x^{-1} + g_yg_x^TG_x^{-1} - \frac{1}{\alpha}(\alpha - 1) g_y g_x^T G_x^{-1} \\
    &= \widetilde A - \frac{1}{\alpha}\left( \widetilde Ag_x - g_y\right)g_x^TG_x^{-1} = \widetilde A - \frac{1}{\alpha}\left(\widetilde A g_x - g_y \right)z^T.
\end{align*}
The similar result can be shown for $\gamma_y > 0$ starting with
$$
\widetilde A_{new} = \begin{pmatrix}
    \ry & g_y \\ \mathbf{0} & \gamma_y
\end{pmatrix}\begin{pmatrix}
    \rx^T \\ g_x^T
\end{pmatrix}G_{x,new}^{-1}. 
$$
In this case $\widetilde X_{new}^T = \left( \begin{smallmatrix}
    Q & \mathbf 0 \\ \mathbf{0} & 1
\end{smallmatrix}\right)\left(\begin{smallmatrix}
     R_x \\ g_x^T
\end{smallmatrix} \right)$ and, after updating the last row of $\left(\begin{smallmatrix}
     R_x \\ g_x^T
\end{smallmatrix} \right)$ to zeros, only the upper triangle is saved (as in \eqref{eq:Rnew-scheme}).
\end{proof}

\begin{remark}
    {\em To use this method in practice, $Q$ must also be updated. If $\gamma_y > 0$, $Q$ is expanded by new orthogonal column (obtained by GS on $y_{new}$), otherwise it stays the same.  This can be done at the end of updates described above.}
\end{remark}

This method (denoted as Exp-sDMD) was applied on Gray-Scott data from subsection \ref{SSS:Gray-Scott} with the same parameters and produced similar results to TQ-sDMD method. Since only the Cholesky factor is updated and used, instead of explicitly forming $\overline S_x \overline S_x^T$, we expect this algorithm to be more robust than HWR-sDMD in finite precision, so long as the Sherman-Morrison and Meyer's formula perform well numerically.

\begin{figure}[h]
    \centering
    \includegraphics[width=0.95\linewidth]{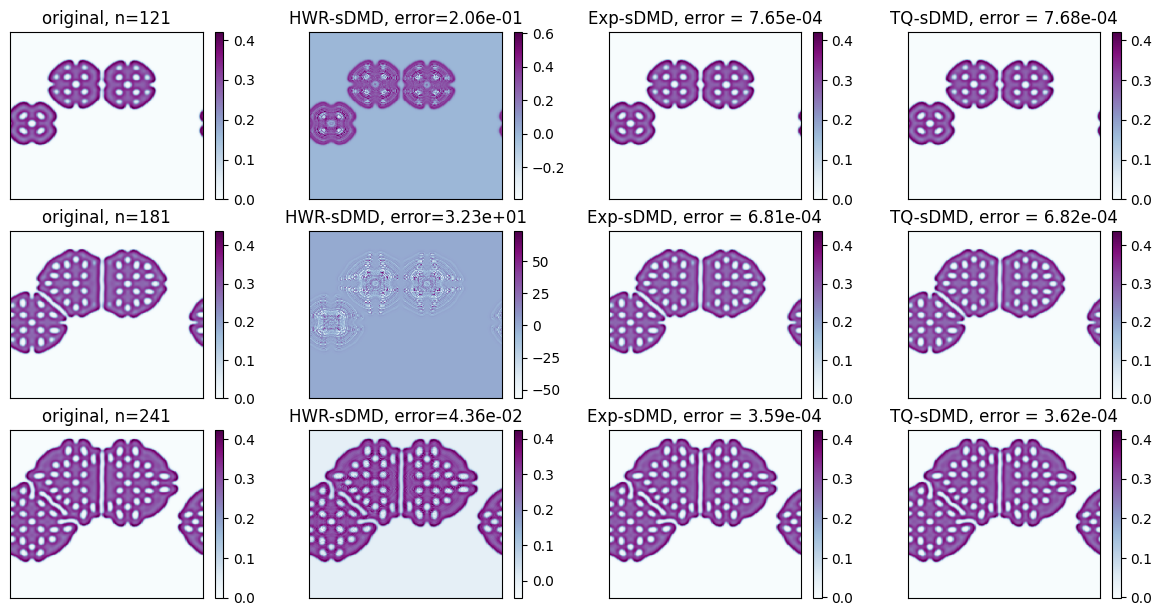}
    \caption{Prediction $1$ step ahead when using HWR-sDMD (64), Exp-sDMD (32) described in this section and TQ-sDMD (32). Exp-sDMD predicts similar snapshots as TQ-sDMD.}
    \label{fig:suppl-gs-4}
\end{figure}

\begin{figure}[h]
    \centering
    \includegraphics[width=0.95\linewidth]{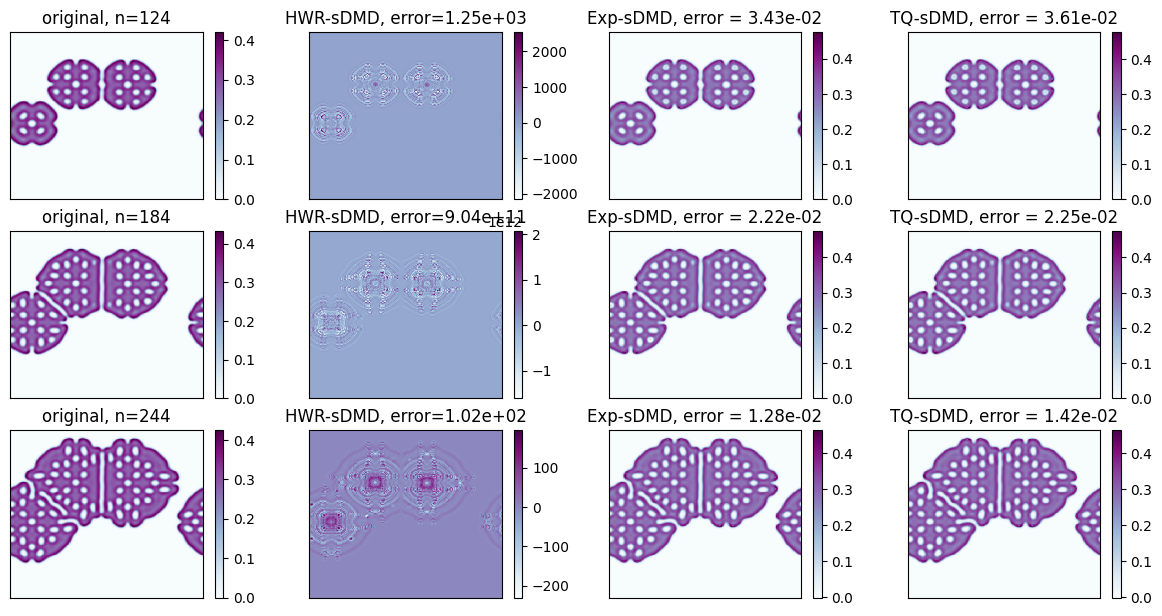}
    \caption{Prediction $5$ steps ahead when using HWR-sDMD (64), Exp-sDMD (32) described in this section and TQ-sDMD (32). Exp-sDMD predicts similar snapshots as TQ-sDMD.}
    \label{fig:suppl-gs-5}
\end{figure}

\begin{figure}[h]
    \centering
    \begin{subfigure}[t]{0.49\textwidth}
        \centering
        \includegraphics[width=1.05\linewidth, height=0.65\linewidth]{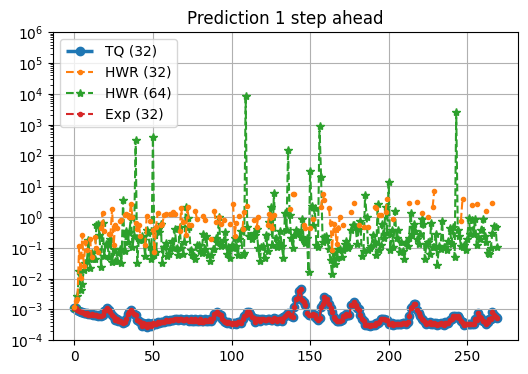}
    \end{subfigure}
    \hfill
    \begin{subfigure}[t]{0.49\textwidth}
        \centering
        \includegraphics[width=1.05\linewidth, height=0.65\linewidth]{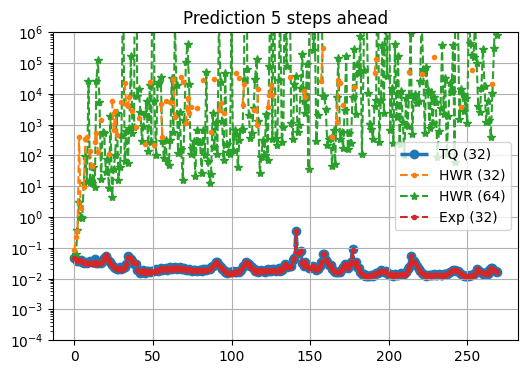}
    \end{subfigure}
    \caption{Relative error norms for prediction $1$ step ahead (\textit{left}) and $5$ steps ahead (\textit{right}) when using HWR-sDMD (64), Exp-sDMD (32) and TQ-sDMD (32). Exp-sDMD has almost the same error norms as TQ-sDMD.}\label{fig:suppl-gs-norms}
\end{figure}
\subsection{Initialization}
All schemes for fast updating of the DMD matrix $A$ require initialization 
$A= YX^\dagger$, which is the solution of the least squares problem $\|AX-Y\|_F\rightarrow\min_{A}$.
There are two typical ways to do that in practice. The theoretically sound way is to compute the pseudoinverse of $X$ and then multiply, $A^{(pinv)} = Y\cdot \texttt{pinv}(X)$. This is appealing because of the rich theory of the pseudoinverse and availability of the function that computes it in popular software packages such as Python, Matlab, Octave. The other is to solve the LS problem $\|X^T A^T-Y^T\|_F\rightarrow\min_{A}$ using rank revealing QR factorization (implemented as \texttt{solve} in python's scipy or $'\backslash{}'$ in matlab) -- thus obtained solution will be denoted by $A^{(solve)}$. 

In general, using the pseudoinverse in situations like this one is ill-advised. We 
illustrate this by showing how the second option  provides better initial approximation and, consequently, smaller errors during the subsequent updates. Similarly, when using SM method, $P_x = G_x^{-1}$ can be initialized by calling \texttt{pinv}, which we will denote by $P_x^{(pinv)}$, or by using \texttt{solve} with the added information that $G_x$ is a positive definite matrix, which we will denote by $P_x^{(solve)}$, implicitly using Cholesky decomposition.

\begin{example}
    We observe Chua's system with $x_0 = [-0.7, 0.1, 0.1], t_{max}=10$, $\Delta t = 10^{-3}$ and $n_0 = 100$. Observables are $x,y,z, x^2, y^2, z^2$. Figure \ref{fig:chua-solve-ls} displays the error difference when using the same algorithms but different initializations for $A$ and $P_x$. In the left figure, $A_0$ is initialized using \texttt{pinv} and in the right using \texttt{solve}. Each figure displays norm error with $P_x$ initialized using \texttt{pinv} and \texttt{solve} for SM method. The figures confirm it is better to use \texttt{solve} for both $A_0$ and $P_x$. Algorithm TQ-sDMD is calculated in the same matter in both figures. 
\begin{figure}[h]
    \centering
    \begin{subfigure}[t]{0.45\textwidth}
        \centering
        \includegraphics[width=0.9\textwidth]{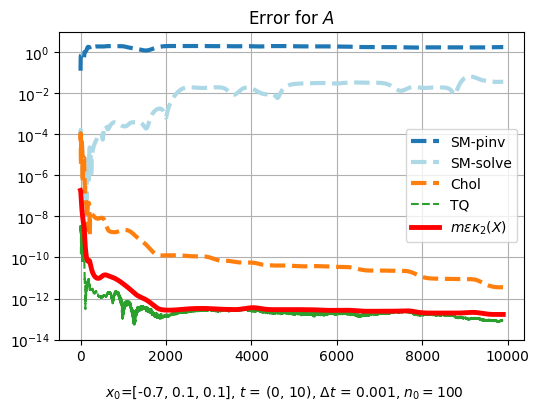}
        \caption{$A_0$ initialized using \texttt{pinv}}
    \end{subfigure}
    \hfill
    \begin{subfigure}[t]{0.45\textwidth}
        \centering
        \includegraphics[width=0.9\textwidth]{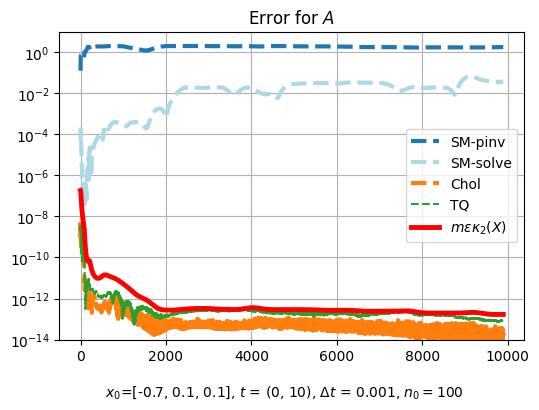}
        \caption{$A_0$ initialized using \texttt{solve}}
    \end{subfigure}
    \caption{Initializing $A_{Chol}$ by \texttt{pinv} results in larger starting error that is inherited by further iterations.  Initialization of $A_{SM}$ does not affect SM as much as the initialization of $P_x$ -- when SM is initialized with \texttt{pinv} (SM-pinv), it produces larger errors than when using \texttt{solve} (SM-solve).}\label{fig:chua-solve-ls}
\end{figure}
\end{example}
\begin{example}
   Next dynamical system we observe is Lorenz system. As initial point we take $x_0 = [1,1,1]$, $t_{max} = 20$, $\Delta t = 0.002$ and classical parameters of the system $\sigma=10, \rho=28, \beta=8/3$. Again we take $x,y, z, x^2, y^2, z^2$ as observables and $n_0 = 100$.

Unlike in Chua's system, the initialization of $A_0$ makes little difference accuracy wise. There is a difference whether we use $P_x^{(pinv)}$ or $P_x^{(solve)}$. The reason for that is the implicit usage of Cholesky factorization of $G_x$ when using solve, replacing calculated $G_x$ with the triangular matrix $\widehat R_x$ with condition number $\kappa_2(\widehat R_x) = \sqrt{\kappa_2(G_x)}$. Keep in mind that, in practice, $\widehat R_x$ is not the same triangular matrix as $R_x$ obtained by QR factorization of $X^T$ due to rounding errors that have occurred as $G_x = XX^T$ was formed. The errors in $A$ and $P_x$ depending on initialization are observed in figures \ref{fig:lorenz-error-A} and \ref{fig:lorenz-error-P}. 
\vspace{-2mm}
\begin{figure}[h]
    \centering
    \begin{subfigure}[t]{0.45\textwidth}
        \centering
        \includegraphics[width=0.9\textwidth]{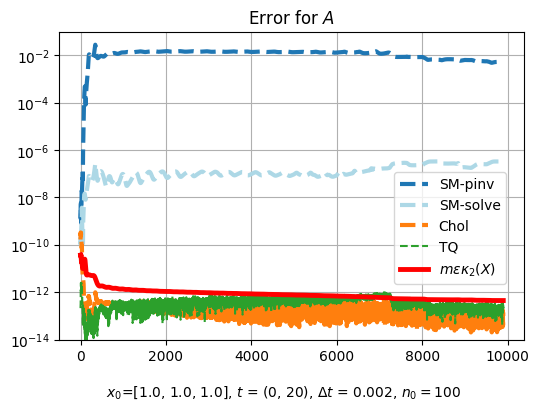}
        \caption{$A_0$ initialized using \texttt{pinv}}
    \end{subfigure}
    \hfill
    \begin{subfigure}[t]{0.45\textwidth}
        \centering
        \includegraphics[width=0.9\textwidth]{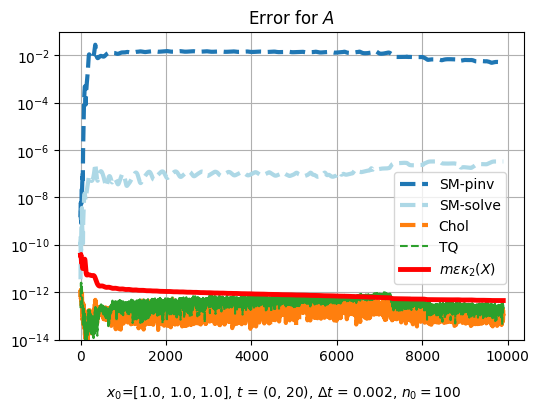}
        \caption{$A_0$ initialized using \texttt{solve}}
    \end{subfigure}
    \caption{Starting $A_0$ with \texttt{solve} or \texttt{pinv} makes very little difference in this scenario. On the other hand, initialization of $P_x$ matters and ultimately produces norm of error in $A_{SM}$ of $10^{-2}$ (\texttt{pinv}) instead of $10^{-7}$ (\texttt{solve}).}\label{fig:lorenz-error-A}
\end{figure}
\vspace{-2mm}
 \begin{figure}[H]
     \centering
     \begin{subfigure}[t]{0.45\textwidth}
         \centering
         \includegraphics[width=0.9\textwidth]{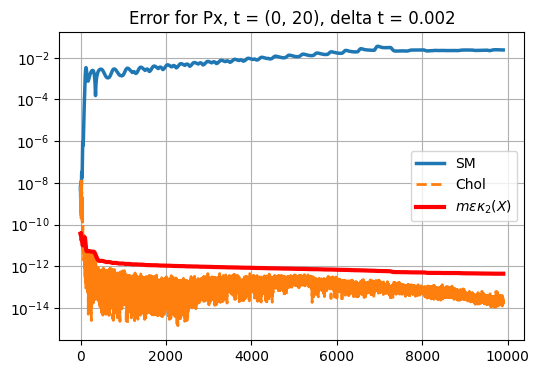}
         \caption{$P_x$ initialized using \texttt{pinv}.}
     \end{subfigure}
     \hfill
     \begin{subfigure}[t]{0.45\textwidth}
         \centering
         \includegraphics[width=0.9\textwidth]{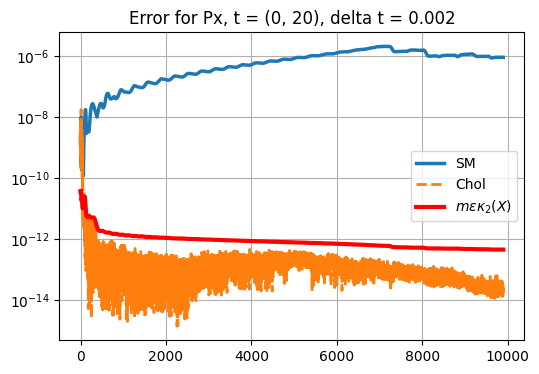}
         \caption{$P_x$ initialized using \texttt{solve}.}
     \end{subfigure}
     \caption{The error for updated $P_x$ starts lower when initializing with \texttt{solve}. The error accumulates up to $10^{-2}$ with \texttt{pinv} initialization, and only up to $10^{-6}$ with \texttt{solve}. In both cases, it is worse than the errors of $\approx 10^{-13}$ when updating using Cholesky factor in SM formula.}\label{fig:lorenz-error-P}
 \end{figure}
\end{example}

%% file: References.bib
@article{SCHMID_2010, 
title={Dynamic mode decomposition of numerical and experimental data}, 
volume={656}, 
DOI={10.1017/S0022112010001217}, 
journal={Journal of Fluid Mechanics}, 
author={Schmid, P. J.}, 
year={2010}, 
pages={5-28}}

@article{Schmid-Sesterhenn-DMD-2008,
	author = {Schmid, P. J. and Sesterhenn, J.},
	year = {2008},
	pages = {208},
	title = {{Dynamic Mode Decomposition} of numerical and experimental data},
	journal = {Bull. Amer. Phys. Soc., 61st APS meeting, San Antonio.}
}

@article{Baddoo-pidmd-2023,
    author = {Baddoo, P. J. and Herrmann, B. and McKeon, B. J. and Nathan, K. J. and Brunton, S. L.},
    title = {Physics-informed dynamic mode decomposition},
    journal = {Proceedings of the Royal Society A: Mathematical, Physical and Engineering Sciences},
    volume = {479},
    number = {2271},
    pages = {20220576},
    year = {2023},
    month = {03},
    issn = {1364-5021},
    doi = {10.1098/rspa.2022.0576}
}

@article{hemati_williams_rowley_2014, 
title={Dynamic mode decomposition for large and streaming datasets}, 
volume={26}, DOI={https://doi.org/10.1063/1.4901016}, 
number={11}, 
journal={Physics of Fluids}, 
publisher={American Institute of Physics}, 
author={Hemati, M. S. and Williams, M. O. and Rowley, C. W.}, year={2014} }

@book{Horn-Johnson-MA-90,
author    = {R. A. Horn and C. R. Johnson},
title     = {Matrix Analysis},
publisher = {Cambridge University Press},
year      = 1990}

@article{zhang_DMD,
author = {Zhang, H. and Rowley, C. W. and Deem, E. A. and Cattafesta, L. N.},
title = {Online Dynamic Mode Decomposition for Time-Varying Systems},
journal = {SIAM Journal on Applied Dynamical Systems},
volume = {18},
number = {3},
pages = {1586-1609},
year = {2019},
doi = {10.1137/18M1192329}
}

@article{chung2023turbulence,
      title={Turbulence in Focus: {B}enchmarking Scaling Behavior of {3D} Volumetric Super-Resolution with {BLASTNet} 2.0 Data}, 
      author={Wai Tong Chung and Bassem Akoush and Pushan Sharma and Alex Tamkin and Ki Sung Jung and Jacqueline H. Chen and Jack Guo and Davy Brouzet and Mohsen Talei and Bruno Savard and Alexei Y. Poludnenko and Matthias Ihme},
      year={2023},
      journal={Advances in Neural Information Processing Systems ({NeurIPS})},
      volume = {36}
}

@book{higham-asna,
author = {Higham, N. J.},
title = {Accuracy and Stability of Numerical Algorithms},
publisher = {Society for Industrial and Applied Mathematics},
year = {2002},
doi = {10.1137/1.9780898718027},
address = {},
edition   = {Second},
URL = {https://epubs.siam.org/doi/abs/10.1137/1.9780898718027},
eprint = {https://epubs.siam.org/doi/pdf/10.1137/1.9780898718027}
}

@ARTICLE{jovschnicPOF14,
AUTHOR = {M. R. Jovanovi\'c and P. J. Schmid and J. W. Nichols},
TITLE = {Sparsity-Promoting Dynamic Mode Decomposition},
JOURNAL = {Phys. Fluids},
VOLUME = {26},
NUMBER = {2},
PAGES = {024103 (22 pages)},
YEAR = {2014}
}

@article{Williams_2015-EDMD,
   title={A Data–Driven Approximation of the Koopman Operator: Extending Dynamic Mode Decomposition},
   volume={25},
   ISSN={1432-1467},
   url={http://dx.doi.org/10.1007/s00332-015-9258-5},
   DOI={10.1007/s00332-015-9258-5},
   number={6},
   journal={Journal of Nonlinear Science},
   publisher={Springer Science and Business Media LLC},
   author={Williams, M. O. and Kevrekidis, I. G. and Rowley, C. W.},
   year={2015},
   month=jun, pages={1307–1346} }

@misc{drmac-koopman-schur,
      title={A data driven Koopman-Schur decomposition for computational analysis of nonlinear dynamics}, 
      author={Z. Drma\v{c} and I. Mezi\'{c}},
      year={2024},
      eprint={2312.15837},
      archivePrefix={arXiv},
      primaryClass={math.NA},
      url={https://arxiv.org/abs/2312.15837}, 
}

@article{Drmac-DMD-TOMS-2024,
	author = {Z. Drma\v{c}},
	title = {A {LAPACK} Implementation of the {Dynamic Mode Decomposition}},
	year = {2024},
	issue_date = {March 2024},
	publisher = {Association for Computing Machinery},
	address = {New York, NY, USA},
	volume = {50},
	number = {1},
	issn = {0098-3500},
	url = {https://doi.org/10.1145/3640012},
	doi = {10.1145/3640012},
	journal = {ACM Trans. Math. Softw.},
	month = mar,
	articleno = {1},
	numpages = {32}
}

@article{dmd_enhancements-qrdmd,
author = {Drma\v{c}, Z. and Mezi\'{c}, I. and Mohr, R.},
title = {Data Driven Modal Decompositions: Analysis and Enhancements},
journal = {SIAM Journal on Scientific Computing},
volume = {40},
number = {4},
pages = {A2253-A2285},
year = {2018},
doi = {10.1137/17M1144155},
URL = { https://doi.org/10.1137/17M1144155},
eprint = { https://doi.org/10.1137/17M1144155
}
}

@INPROCEEDINGS{Ryu-etal-OnboardDMD-2022,
  author={Ryu, T. and Ali, W. H. and Haley, P. J. and Mirabito, C. and Charous, A. and Lermusiaux, P. F. J.},
  booktitle={OCEANS 2022, Hampton Roads}, 
  title={Incremental Low-Rank Dynamic Mode Decomposition Model for Efficient Forecast Dissemination and Onboard Forecasting}, 
  year={2022},
  volume={},
  number={},
  pages={1-8},
  doi={10.1109/OCEANS47191.2022.9977224}}

@article{Quadcoper-dmd-2023,
title = {Quadcopters Control Using Online Dynamic Mode Decomposition},
journal = {IFAC-PapersOnLine},
volume = {56},
number = {3},
pages = {589-594},
year = {2023},
note = {3rd Modeling, Estimation and Control Conference MECC 2023},
issn = {2405-8963},
doi = {https://doi.org/10.1016/j.ifacol.2023.12.088},
url = {https://www.sciencedirect.com/science/article/pii/S2405896323024217},
author = {B. S. Guevara and L. F. Recalde and J. Varela-Aldás and D. C. Gandolfo and J. M. Toibero}
}

@article{Li_Deng_2024, 
title={Active control of the flow past a circular cylinder using online dynamic mode decomposition}, 
volume={997}, 
DOI={10.1017/jfm.2024.738}, 
journal={Journal of Fluid Mechanics}, 
author={Li, X. and Deng, J.}, 
year={2024}, 
pages={A26}}

@article{Mignacca-videoDMD-2025, 
title={Real-time motion detection using dynamic mode decomposition}, 
volume={10}, 
DOI={10.1186/s13640-025-00673-4}, 
journal={J Image Video Proc}, 
author={Mignacca, M. and Brugiapaglia, S. and Bramburger, J.J.}, 
year={2025}, 
pages={1--18}}

@article{Avilla-Mezic-2020, 
title={Data-driven analysis and forecasting of highway traffic dynamics}, 
volume={11}, 
number = {2090},
DOI={10.1038/s41467-020-15582-5}, 
journal={Nature Communications }, 
author={Avila, A.M. and Mezi\'{c}, I.}, 
year={2020}}

@misc{usdot_ngsim_2016,
  author       = {{U.S. Department of Transportation Federal Highway Administration}},
  title        = {{Next Generation Simulation (NGSIM) Vehicle Trajectories and Supporting Data}},
  year         = {2016},
  howpublished = {\url{http://doi.org/10.21949/1504477}},
  note         = {Dataset provided by ITS DataHub through Data.transportation.gov. Accessed 2025-07-12},
}

@ARTICLE{2023arXiv231118715S,
       author = {{Suh}, S. W. and {Chung}, S. W. and {Bremer}, P.-T. and {Choi}, Y.},
        title = "{Accelerating Flow Simulations using Online Dynamic Mode Decomposition}",
      journal = {arXiv e-prints},
         year = 2023,
        month = nov,
       doi = {10.48550/arXiv.2311.18715},
}

@article{windfarms-2022,
author = {Liew, J. and Göçmen, T. and Lio, W. H. and Larsen, G. C.},
title = {Streaming dynamic mode decomposition for short-term forecasting in wind farms},
journal = {Wind Energy},
volume = {25},
number = {4},
pages = {719-734},
doi = {https://doi.org/10.1002/we.2694},
year = {2022}
}

@article{HUHN2023111852,
title = {Parametric dynamic mode decomposition for reduced order modeling},
journal = {Journal of Computational Physics},
volume = {475},
pages = {111852},
year = {2023},
issn = {0021-9991},
doi = {https://doi.org/10.1016/j.jcp.2022.111852},
url = {https://www.sciencedirect.com/science/article/pii/S0021999122009159},
author = {Q. A. Huhn and M. E. Tano and J. C. Ragusa and Y. Choi}
}

@ARTICLE{Matsumoto-fly-dmd-2017,
       author = {{Matsumoto}, D. and {Indinger}, T.},
        title = "{On-the-fly algorithm for Dynamic Mode Decomposition using Incremental Singular Value Decomposition and Total Least Squares}",
      journal = {arXiv e-prints},
         year = 2017,
        month = mar,
          eid = {arXiv:1703.11004},
        pages = {arXiv:1703.11004},
          doi = {10.48550/arXiv.1703.11004},
archivePrefix = {arXiv},
       eprint = {1703.11004},
 primaryClass = {physics.flu-dyn},
       adsurl = {https://ui.adsabs.harvard.edu/abs/2017arXiv170311004M}
}

@Article{Nedzhibov-SVDupdate-DMD-2023,
AUTHOR = {Nedzhibov, G.},
TITLE = {Extended Online DMD and Weighted Modifications for Streaming Data Analysis},
JOURNAL = {Computation},
VOLUME = {11},
YEAR = {2023},
NUMBER = {6},
ARTICLE-NUMBER = {114},
URL = {https://www.mdpi.com/2079-3197/11/6/114},
ISSN = {2079-3197},
DOI = {10.3390/computation11060114}
}

@Article{Chaugule-PIV-DMD-2023,
AUTHOR = {Chaugule, V. and Duddridge, A. and Sikroria, T. and Atkinson, C. and Soria, J.},
TITLE = {Investigating the Linear Dynamics of the Near-Field of a Turbulent High-Speed Jet Using Dual-Particle Image Velocimetry (PIV) and Dynamic Mode Decomposition (DMD)},
JOURNAL = {Fluids},
VOLUME = {8},
YEAR = {2023},
NUMBER = {2},
ARTICLE-NUMBER = {73},
URL = {https://www.mdpi.com/2311-5521/8/2/73},
ISSN = {2311-5521},
DOI = {10.3390/fluids8020073}
}

@Article{Ali-PSPD-DMD-2016,
AUTHOR = {Ali, M. Y. and Pandey, A. and Gregory, J. W.},
TITLE = {Dynamic Mode Decomposition of Fast Pressure Sensitive Paint Data},
JOURNAL = {Sensors},
VOLUME = {16},
YEAR = {2016},
NUMBER = {6},
ARTICLE-NUMBER = {862},
URL = {https://www.mdpi.com/1424-8220/16/6/862},
PubMedID = {27294939},
ISSN = {1424-8220},
DOI = {10.3390/s16060862}
}

@book{Kutz-SIAM-BOOK-2016,
	author = {J. N. Kutz and S. L. Brunton and B. W. Brunton and J. L. Proctor},
	title = {Dynamic Mode Decomposition},
	publisher = {Society for Industrial and Applied Mathematics},
	year = {2016},
	doi = {10.1137/1.9781611974508},
	address = {Philadelphia, PA},
	edition   = {},
	URL = {https://epubs.siam.org/doi/abs/10.1137/1.9781611974508},
}

@misc{DMD-book-supplement,
author = {J. N. Kutz and S. L. Brunton and B. W. Brunton and J. L. Proctor},
title = {Dynamic Mode Decomposition -- supplementary material},
howpublished = {\url{http://dmdbook.com/}},
}

@article{mezic_etal-covid_2024,
  author       = {Mezi\'{c}, I. and Drma\v{c}, Z. and \v{C}rnjari\'{c}, N. and Ma\'{c}e\v{s}i\'{c} S. and Fonoberova, M. and Mohr, R. and Avila, A. and Manojlovi\'{c}, I. and Andrej\v{c}uk, A.},
  title        = {A Koopman operator-based prediction algorithm and its application to COVID-19 pandemic and influenza cases},
  journal      = {Scientific Reports},
  volume       = {14},
  pages        = {5788},
  year         = {2024},
  doi          = {10.1038/s41598-024-55798-9},
  url          = {https://doi.org/10.1038/s41598-024-55798-9},
}

@article{ROWLEY_MEZIC_BAGHERI_SCHLATTER_HENNINGSON_2009, 
	title={Spectral analysis of nonlinear flows}, 
	volume={641}, 
	DOI={10.1017/S0022112009992059}, 
	journal={Journal of Fluid Mechanics}, 
	author={Rowley, C. W. and Mezi\'{c}, I. and Bagheri, S. and Schlatter, P. and Henningson, D. S.}, 
	year={2009}, 
	pages={115--127}}

@Article{Schmid2011,
	author="Schmid, P. J.
	and Li, L.
	and Juniper, M. P.
	and Pust, O.",
	title="Applications of the dynamic mode decomposition",
	journal="Theoretical and Computational Fluid Dynamics",
	year="2011",
	volume="25",
	number="1",
	pages="249--259",
	issn="1432-2250",
	doi="10.1007/s00162-010-0203-9",
	url="http://dx.doi.org/10.1007/s00162-010-0203-9"
}

@article{schmid-2011-dmd-exp-fluids,
	author = {P. J. Schmid},
	title = {Application of the dynamic mode decomposition to experimental data},
	journal = {Experiments in Fluids},
	volume = {50},
	number = {4},
	pages = {1123--1130},
	URL = {https://doi.org/10.1007/s00348-010-0911-3},
	doi = {10.1007/s00348-010-0911-3},
	year = {2011}
}

@incollection{SCHMID2021243,
	title = {Data--driven and operator--based tools for the analysis of turbulent flows},
	editor = {Paul Durbin},
	booktitle = {Advanced Approaches in Turbulence},
	publisher = {Elsevier},
	pages = {243--305},
	year = {2021},
	isbn = {978-0-12-820774-1},
	doi = {https://doi.org/10.1016/B978-0-12-820774-1.00012-4},
	url = {https://www.sciencedirect.com/science/article/pii/B9780128207741000124},
	author = {P. J. Schmid}
}

@article{schmid-2022-dmd-variants,
	author = {P. J. Schmid},
	title = {Dynamic Mode Decomposition and Its Variants},
	journal = {Annual Review of Fluid Mechanics},
	volume = {54},
	number = {1},
	pages = {225--254},
	year = {2022},
	doi = {10.1146/annurev-fluid-030121-015835},
	URL = {https://doi.org/10.1146/annurev-fluid-030121-015835},
	eprint = {https://doi.org/10.1146/annurev-fluid-030121-015835}
}

@article{mezic_annual_reviews-2013,
	Author = {I. Mezi{\'c}},
	Journal = {Annual Reviews of Fluid Mechanics},
	Pages = {357--378},
	Title = {Analysis of Fluid Flows via Spectral Properties of the {K}oopman Operator},
	Volume = {45},
	Year = {2013}}

@article{Jovanovic:2012wy,
	author = {Jovanovi{\'c}, M. R. and Schmid, P. J. and Nichols, J. W.},
	title = {{Low--rank and sparse dynamic mode decomposition}},
	journal = {Center for Turbulence Research, Annual Research Briefs},
	year = {2012},
	pages = {139--152}
}

@article{GIANNAKIS2020132211,
	title = {Extraction and prediction of coherent patterns in incompressible flows through space--time {Koopman} analysis},
	journal = {Physica D: Nonlinear Phenomena},
	volume = {402},
	pages = {132211},
	year = {2020},
	issn = {0167-2789},
	doi = {https://doi.org/10.1016/j.physd.2019.132211},
	url = {https://www.sciencedirect.com/science/article/pii/S0167278917303317},
	author = {D. Giannakis and S. Das}
}

@article{tu-rowley-dmd-theory-appl-2014,
	title = {On dynamic mode decomposition:  Theory and applications},
	journal = {Journal of Computational Dynamics},
	volume = {1},
	number = {2},
	pages = {391-421},
	year = {2014},
	author = {J. H. Tu and C. W. Rowley and D. M. Luchtenburg and S. L. Brunton and J. N. Kutz},
}

@incollection{Akshay2021,
	author="Akshay, S. and Soman, K. P. and Mohan, N. and Sachin Kumar, S.",
	editor="Kumar, Raman and Paiva, Sara",
	title="Dynamic Mode Decomposition and its application in various domains: An overview",
	booktitle="Applications in Ubiquitous Computing",
	year="2021",
	publisher="Springer International Publishing",
	address="Cham",
	pages="121--132",
	isbn="978-3-030-35280-6",
	doi="10.1007/978-3-030-35280-6_6",
	url="https://doi.org/10.1007/978-3-030-35280-6_6"
}

@article{brunton-budisic-kaiser-kutz-sirew-2022,
	author = {S. L. Brunton and M. Budi\v{s}i\'{c} and E. Kaiser and J. N. Kutz},
	title = {Modern {Koopman} Theory for Dynamical Systems},
	journal = {SIAM Review},
	volume = {64},
	number = {2},
	pages = {229-340},
	year = {2022},
	doi = {10.1137/21M1401243},
	URL = {https://doi.org/10.1137/21M1401243},
	eprint = {https://doi.org/10.1137/21M1401243}
}

@article{Mezic-Spectral-MOR-2005,
	author = {Mezi\'{c}, I.},
	year = {2005},
	pages = {309--325},
	title = {Spectral Properties of Dynamical Systems, Model Reduction and Decompositions},
	volume = {41},
	journal = {Nonlinear Dynamics},
	doi = {10.1007/s11071-005-2824-x}
}

@article{Mezic-Koop-Spectrum-2020,
	author = {Mezi\'{c}, I.},
	year = {2020},
	pages = {2091--2145},
	title = {Spectrum of the {Koopman} Operator, Spectral Expansions in Functional Spaces, and State--Space Geometry},
	volume = {30},
	number = {5},
	journal = {Nonlinear Science},
	doi = {10.1007/s00332-019-09598-5}
}

@article{Sayadi-etal-param-DMD-2015,
    author = {Sayadi, T. and Schmid, P. J. and Richecoeur, F. and Durox, D.},
    title = {Parametrized data-driven decomposition for bifurcation analysis, with application to thermo-acoustically unstable systems},
    journal = {Physics of Fluids},
    volume = {27},
    number = {3},
    pages = {037102},
    year = {2015},
    month = {03},
    issn = {1070-6631},
    doi = {10.1063/1.4913868},
}

@article{Gram-schmidt-numerics,
title = {Numerics of Gram-Schmidt orthogonalization},
journal = {Linear Algebra and its Applications},
volume = {197-198},
pages = {297-316},
year = {1994},
issn = {0024-3795},
doi = {https://doi.org/10.1016/0024-3795(94)90493-6},
author = {{\AA}. Bj{\"o}rck}
}

@dataset{chung_wai_tong_2023_8034232,
  author       = {Chung, W. T. and
                  Ihme, M. and
                  Jung, K. S. and
                  Chen, J. H. and
                  Guo, J. and
                  Brouzet, D. and
                  Talei, M. and
                  Jiang, B. and
                  Savard, B. and
                  Poludnenko, A. Y. and
                  Akoush, B. and
                  Sharma, P. and
                  Tamkin, A.},
  title        = {BLASTNet Simulation Dataset},
  month        = jun,
  year         = 2023,
  note         = {URL: https://blastnet.github.io/},
  publisher    = {Zenodo},
  version      = {2.0},
  doi          = {10.5281/zenodo.8034232},
  url          = {https://doi.org/10.5281/zenodo.8034232}
}

@article{wynn_pearson_ganapathisubramani_goulart_2013, 
	title={Optimal mode decomposition for unsteady flows}, 
	volume={733},
	 DOI={10.1017/jfm.2013.426}, 
	 journal={Journal of Fluid Mechanics}, 
	 publisher={Cambridge University Press}, 
	 author={Wynn, A. and Pearson, D. S. and Ganapathisubramani, B. and Goulart, P. J.}, 
	 year={2013}, 
	 pages={473--503}}

@Article{Dawson2016,
	author="Dawson, S. T. M.
	and Hemati, M. S.
	and Williams, M. O.
	and Rowley, C. W.",
	title="Characterizing and correcting for the effect of sensor noise in the dynamic mode decomposition",
	journal="Experiments in Fluids",
	year="2016",
	volume="57",
	number="3",
	pages="42",
	issn="1432-1114",
	doi="10.1007/s00348-016-2127-7",
	url="http://dx.doi.org/10.1007/s00348-016-2127-7"
}

@ARTICLE{Colbrook-Drmac-Horning-2025,
       author = {{Colbrook}, M. J. and {Drma{\v{c}}}, Z. and {Horning}, A.},
        title = "{An Introductory Guide to Koopman Learning}",
      journal = {arXiv e-prints},
         year = 2025,
        month = oct,
          eid = {arXiv:2510.22002},
        pages = {arXiv:2510.22002},
          doi = {10.48550/arXiv.2510.22002},
archivePrefix = {arXiv},
       eprint = {2510.22002},
 primaryClass = {cs.NA},
       adsurl = {https://ui.adsabs.harvard.edu/abs/2025arXiv251022002C}
}

@article{Yang-Zhang-strDMD-2022,
author = {Yang, R. and Zhang, X. and Reiter, P. and Lohse, D. and Shishkina, O. and Linkmann, M.},
title = {Data-driven identification of the spatiotemporal structure of turbulent flows by streaming dynamic mode decomposition},
journal = {GAMM-Mitteilungen},
volume = {45},
number = {1},
pages = {e202200003},
doi = {https://doi.org/10.1002/gamm.202200003},
year = {2022}
}

@misc{ma2025notestabilityshermanmorrisonwoodburyformula,
      title={A Note on the Stability of the {Sherman-Morrison-Woodbury} Formula}, 
      author={L. Ma and C. Boutsikas and M. Ghadiri and P. Drineas},
      year={2025},
      eprint={2504.04554},
      archivePrefix={arXiv},
      primaryClass={math.NA},
      url={https://arxiv.org/abs/2504.04554}, 
}

@misc{chuacircuits_matlab,
  author       = {{ChuaCircuits.com}},
  title        = {MATLAB Simulation of Chua's Circuit},
  howpublished = {\url{https://chuacircuits.com/matlab/}},
  year         = {2025},
  note         = {Accessed: 2026-03-09}
}

@ARTICLE{Chua-1980,
  author={Chua, L.},
  journal={IEEE Transactions on Circuits and Systems}, 
  title={Dynamic nonlinear networks: State-of-the-art}, 
  year={1980},
  volume={27},
  number={11},
  pages={1059-1087},
  doi={10.1109/TCS.1980.1084745}}

@ARTICLE{Matsumoto-1984,
  author={Matsumoto, T.},
  journal={IEEE Transactions on Circuits and Systems}, 
  title={A chaotic attractor from Chua's circuit}, 
  year={1984},
  volume={31},
  number={12},
  pages={1055-1058},
  doi={10.1109/TCS.1984.1085459}}

@article{meyer_update_pinv,
author = {Meyer, Jr., C. D.},
title = {Generalized Inversion of Modified Matrices},
journal = {SIAM Journal on Applied Mathematics},
volume = {24},
number = {3},
pages = {315-323},
year = {1973},
doi = {10.1137/0124033},

URL = { 
        https://doi.org/10.1137/0124033
},
eprint = {    
        https://doi.org/10.1137/0124033
}
}
